%% file: arxiv_LRE.tex
\documentclass[a4paper, 11pt]{article}

\usepackage{algpseudocode,algorithm}
\usepackage{hyperref}
\usepackage{authblk}
\usepackage[numbers]{natbib}
\usepackage{blindtext}
\usepackage[utf8]{inputenc} 
\usepackage[T1]{fontenc}    
\usepackage{url}            
\usepackage{booktabs}       
\usepackage{amsfonts,amsmath,amsthm}       
\usepackage{mathrsfs}
\usepackage{multirow}
\usepackage{graphicx,subfigure}       
\newcommand{\bo}{\mathcal O}
\newcommand\expinv[2]{\exp^{-1}_{#1} #2}
\newcommand{\ft}{\mathbf{f}_t}
\newcommand{\M}{\mathcal{M}}
\newcommand{\N}{\mathcal{N}}
\newcommand{\K}{\mathcal{K}}
\newcommand{\RegD}{{\sf{Reg}_D}}
\newcommand{\RegS}{{\sf{Reg}_S}}
\newcommand{\R}{\mathbb R}
\newcommand{\X}{\mathcal{X}}

\newcommand{\f}{\mathbf{f}}
\newcommand{\F}{\mathbf{F}}

\usepackage{amssymb}

\DeclareFontFamily{OT1}{pzc}{}
\DeclareFontShape{OT1}{pzc}{m}{it}{<-> s * [1.200] pzcmi7t}{}
\DeclareMathAlphabet{\mathpzc}{OT1}{pzc}{m}{it}

\newtheorem{corollary}{Corollary}
\newtheorem{definition}{Definition}
\newtheorem{assumption}{Assumption}

\newtheorem{lemma}{Lemma}
\newtheorem{theorem}{Theorem}

\begin{document}

\title{\bf Riemannian Optimistic Algorithms}

\author{Xi Wang\thanks{X. Wang is with Australian Center for Robotics, School of Aerospace, Mechanical and Mechatronic Engineering, The University of Sydney, NSW 2006, Australia; Academy of Mathematics and Systems Science, Chinese Academy of Sciences, Beijing 100190, P.~R.~China. (Email:  wangxi14.ucas@gmail.com)}, Deming Yuan\thanks{D. Yuan is with School of Automation, Nanjing University of Science and Technology, Jiangsu 210023, P.~R.~China (Email: dmyuan@njupt.edu.cn)}, Yiguang Hong\thanks{Y. Hong is with Shanghai Research Institute for Intelligent Autonomous Systems, Tongji University, Shanghai 201210, P.~R.~China (Email: yghong@iss.ac.cn)}, Zihao Hu\thanks{Z. Hu is with College of Computing, Georgia Institute of Technology, Atlanta, GA 30339, The United States (Email: zihaohu@gatech.edu)},  Lei Wang\thanks{L. Wang is with College of Control Science and Engineering, Zhejiang University, Zhejiang 310058, P.~R.~China (Email: lei.wangzju@zju.edu.cn)}, Guodong Shi\thanks{G. Shi is with Australian Center for Robotics, School of Aerospace, Mechanical and Mechatronic Engineering, The University of Sydney, NSW 2006, Australia (Correspondence author, Email: guodong.shi@sydney.edu.au)} }

\date{}

\maketitle

\begin{abstract}
In this paper, we consider Riemannian online convex optimization with dynamic regret. First, we propose two novel algorithms, namely the Riemannian Online Optimistic Gradient Descent (R-OOGD) and the Riemannian Adaptive Online Optimistic Gradient Descent (R-AOOGD), which combine the advantages of classical optimistic algorithms with the rich geometric properties of Riemannian manifolds. We analyze the dynamic regrets of the R-OOGD and R-AOOGD in terms of regularity of the sequence of cost functions and comparators. Next, we apply the R-OOGD to Riemannian zero-sum games, leading to the Riemannian Optimistic Gradient Descent Ascent algorithm (R-OGDA). We analyze the average iterate and best-iterate of the R-OGDA in seeking Nash equilibrium for a two-player, zero-sum, g-convex-concave games. We also prove the last-iterate convergence of the R-OGDA for g-strongly convex-strongly concave problems. Our theoretical analysis shows that all proposed algorithms achieve results in regret and convergence that match their counterparts in Euclidean spaces. Finally, we conduct several experiments to verify our theoretical findings.
\end{abstract}

\section{Introduction}
    Online optimization has become increasingly important in recent decades, as it aims to optimize a sequence of decision variables in real-time, despite uncertainty and limited feedback. The online optimization has numerous applications in fields such as machine learning, signal imaging, and control systems \citep{agmon1954relaxation,hazan2019introduction,NEURIPS2019_1dba5eed}.

    The decision variables in online learning may be defined on Riemannian manifolds. Modeling signals on Riemannian manifolds can enhance data representation capabilities \citep{liu2019hyperbolic} and reduce problem dimension \citep{li2021multi,hu2020brief}. Moreover, Riemannian optimization benefits from the property of \textit{geodesic convexity} (g-convexity) \citep{allen2018operator}, which permits conversion of Euclidean non-convex optimization problems into g-convex ones by appropriately choosing the Riemannian metric on the manifold. In this paper, we focus on the \textit{Riemannian online convex optimization} (R-OCO) problem, defined as:

\begin{align}\label{eq:ROCO}
\min_{x_t\in \K \subset \M} \ft(x_t),
\end{align}
    where a learner plays against an adversary or nature. In each round $t \in 1,2,\dots,T$, the learner selects an action $x_t$ from a geodesically convex (g-convex) subset $\K$. The adversary or nature then produces a geodesically convex (g-convex) function $\ft$ defined on $\K$ for which the learner has no prior knowledge. Finally, the learner receives feedback on $\ft$ and incurs a corresponding loss $\ft(x_t)$. The problem (\ref{eq:ROCO}) of Riemannian online convex optimization (R-OCO) is an extension of the classic online convex optimization in Euclidean spaces, with potential applications in machine learning, including robotic control, medical imaging, and neural networks \citep{lee2005online,fan2020efficient,shin2022robust}.

    In the context of the R-OCO problem, one important metric is the \textit{dynamic regret} \citep{jadbabaie15online}, which measures the difference in cumulative loss between an online optimization algorithm and a sequence of comparators $\{u_1,\dots,u_T\}$, that is:
    \begin{align*}
    \RegD(u_1,\dots,u_T) = \sum_{t=1}^T \ft(x_t) - \sum_{t=1}^T \ft(u_t).
    \end{align*}
    Compared to the well-known \textit{static regret} \citep{zinkevich2003online} 
    \begin{align*}
        \RegS(T) = \sum_{t=1}^T \ft(x_t) - \min_{x\in \M} \sum_{t=1}^T\ft(x),
    \end{align*}  the dynamic regret provides a more comprehensive evaluation of online algorithms, as it takes into account the adjustments and adaptations of the environment at each time step.

  While it has been demonstrated that online convex optimization algorithms may result in $\Omega(T)$ dynamic regret bound in the worst case, it is also possible to bound dynamic regrets related to quantities that reflect the regularity of the problem \citep{jadbabaie15online,zhang2018adaptive,zhao2020dynamic}, such as the 
    \textit{path-length}
    \begin{align*}
    P_T := \sum_{t=2}^{T} d(u_t,u_{t-1}),
    \end{align*}
    the \textit{gradient variation}
    \begin{align*}
        V_T:=\sum_{t=2}^{T}\sup_{x\in \K} \|\nabla \ft(x) - \nabla \mathbf{f}_{t-1}(x)\|^2,
    \end{align*} 
    and the \textit{comparator loss}
        \begin{align*}
        F_T:=\sum_{t=1}^{T}\ft(u_t).
    \end{align*}If the comparators $u_t$ and the cost function $\ft$ adapts slowly, the dynamic regret can be greatly reduced.

    In Euclidean online convex optimization (OCO), the Online Optimistic Gradient Descent algorithm (OOGD)\citep{jadbabaie15online,zhao2020dynamic} is a noteworthy example that certifies dynamic regret bounds based on path-length and gradient variation. The OOGD has been extensively studied and shown to achieve near-optimal dynamic regret bounds for a wide range of convex optimization problems. For instance, \cite{jadbabaie15online} presented a dynamic regret bound of $\bo(P_T\sqrt{1+V_T})$ for the vanilla OOGD, and then \cite{zhao2020dynamic} introduced a meta-expert structured OOGD that achieved a dynamic regret bound of $\bo(\sqrt{(1+P_T+V_T)(1+P_T)})$. Moreover, the OOGD algorithm has also been applied to solve zero-sum games \citep{mokhtari2020convergence,wei2021last,gorbunov2022last}, which may be used to study generative adversarial networks \citep{mertikopoulos2018optimistic}. 
    
    Despite its success in Euclidean space, the extension of OOGD to Riemannian manifolds for dynamic regret has received limited attention. Some studies have examined the static regret of the R-OCO problem. \cite{becigneul2018Riemannian} studied the static regret of Riemannian adaptive methods, which required a product manifold structure. \cite{wang2023online} investigated the static regrets of the Riemannian online gradient descent algorithm and Riemannian online bandit methods.

    With dynamic regrets, \cite{maass2020online} explored a zeroth-order dynamic regret bound for strongly g-convex and strongly g-smooth functions on Hadamard manifolds. In recent works, \cite{hu2023minimizing} proposed several algorithms for adaptive dynamic regret bounds, incorporating gradient variation bounds, small loss bounds, and best-of-world bounds. In particular, \cite{hu2023minimizing} introduced the gradient variation regret bounds via a Riemannian extragradient algorithm framework. The extragradient framework, similar to optimistic algorithms, typically requires two gradients per iteration. One gradient is computed at the current decision point $x_t$, while the other gradient is obtained by extrapolating the current gradient to a midpoint $y_t$. However, it is worth noting that optimistic algorithms can operate solely based on the strategy point $x_t$. Thus, developing a Riemannian version of OOGD is crucial for advancing the field of online optimization on manifolds.

    \paragraph{Contribution} Motivated by above, this paper aims to design optimistic online methods for Riemannian online optimization and derive dynamic regret bounds with respect to the regularity of the problem. The contribution of this paper is summarized as follows:
    \begin{itemize}
        \item We propose the Riemannian online optimistic gradient descent algorithm (R-OOGD), which uses only the gradient of the strategy point. We establish an $\bo(P_T\sqrt{1+V_T})$ dynamic regret bound for g-convex losses.
        \item We introduce the meta-expert framework in \cite{hu2023minimizing} to the R-OOGD algorithm and propose the Riemannian online adaptive optimistic gradient descent (R-AOOGD) algorithm. 
        We then establish a dynamic regret bound of $\bo(\sqrt{(1+V_T+P_T)(1+P_T)})$ of the R-AOOGD on g-convex losses. 
        \item We apply the R-OOGD algorithm to two-player zero-sum games on Riemannian manifolds and obtain the Riemannian Optimistic Gradient Descent Ascent algorithm (R-OGDA) for Nash equilibrium seeking in Riemannian zero-sum games. We prove $\bo(\frac{1}{T})$ average-iterate and $\bo(\frac{1}{\sqrt{T}})$ best-iterate convergence of the R-OGDA for g-convex-concave games. Moreover, we prove linear last-iterate convergence for g-strongly convex-strongly concave games.
    \end{itemize}
    The established regret bounds and convergence rates in our paper match the works in Euclidean space \citep{jadbabaie15online,zhao2020dynamic,mokhtari2020unified,mokhtari2020convergence}. We briefly list them in Table \ref{tab:comparison}.
\begin{table}[ht]
\centering
\begin{tabular}{|l|c|c|}
\hline
Algorithm & Problem setting & Result \\
\hline
\multirow{3}{*}{R-OOGD} & R-OCO, g-convex & $\bo(\frac{\zeta_0}{\sqrt{\sigma_0}}P_T\sqrt{1+V_T})$\\
\cline{2-3}
& \multirow{2}{*}{Euclidean, convex}& $\bo(P_T\sqrt{1+V_T}) $ \\
& &\citep{jadbabaie15online}\\
\hline
\multirow{2}{*}{R-AOOGD} & R-OCO, g-convex & $\bo(\frac{\zeta_0}{\sqrt{\sigma_0}}\sqrt{(1+V_T+P_T)(1+P_T)})$ \\ 
\cline{2-3}
& \multirow{2}{*}{Euclidean, convex}& $\bo(\sqrt{(1+V_T+P_T)(1+P_T)})$  \\
& & \citep{zhao2020dynamic}\\
\hline
\multirow{8}{*}{R-OGDA} & RZS, g-convex-concave,  & \multirow{2}{*}{$\bo(\frac{\zeta_1}{\sigma_1 T})$}\\
 & average-iterate & \\
\cline{2-3}
& Euclidean convex-concave & $\bo(\frac{1}{T})$ \\
&  average-iterate & \citep{mokhtari2020convergence} \\
\cline{2-3}
& RZS, g-convex-concave &  \multirow{2}{*}{$\bo(\frac{1}{\sqrt {\sigma_1}T})$} \\
&  best-iterate & \\
\cline{2-3}
&  Euclidean convex-concave  & $\bo(\frac{1}{\sqrt T})$ \\ 
&  best-iterate & \citep{chavdarova2021last} \\
\cline{2-3}
& RZS, g-SCSC, last-iterate & linear  \\
\cline{2-3}
& Euclidean, SCSC, &linear\\
& last-iterate& \citep{mokhtari2020unified} \\
\hline
\end{tabular}
\vspace{1em}
\caption{Comparison of Algorithms in our work and corresponding Euclidean algorithms. SCSC denotes strongly convex-strongly concave. The constants $\zeta_0,\zeta_1,\sigma_0,\sigma_1$ are related to Riemannian curvature and diameter bounds.}
\label{tab:comparison}
\end{table}

\section{Related Work}
In this section, we will provide a brief review of previous work on online convex optimization and zero-sum games in both Euclidean spaces and Riemannian manifolds.
\subsection{Online Convex Optimization}
 \paragraph{Euclidean OCO} The concept of online optimization and static regret was introduced by \cite{zinkevich2003online}. \cite{zinkevich2003online} also proposed the online gradient descent (OGD) method and constructed an $\bo(\sqrt T)$ regret bound on convex functions. Later, \cite{hazan2006logarithmic} demonstrated that the OGD method achieves an $\bo(\log T)$ regret bound on strongly convex functions. \cite{abernethy2008optimal} proved the universal lower bounds for online algorithms to be $\Omega(\sqrt T)$ and $\Omega(\log T)$ for convex and strongly convex functions respectively, which illustrated that the bounds of OGD are tight.

    Aside from regret bounds related to the time horizon $T$, several papers explored the idea of exploiting regularity in the online optimization problem to achieve better regret bounds. For example, in addressing the online optimization problem where the loss functions have a small deviation, \cite{chiang2012online} proposed an online version of the extragradient algorithm, which achieved a regret bound of $\bo(\sqrt{1+V_T})$ with respect to gradient variation $  V_T:=\sum_{t=2}^{T}\sup_{x\in \K} \|\nabla \ft(x) - \nabla \mathbf{f}_{t-1}(x)\|^2$ on convex and smooth functions. The online extragradient algorithm required constructing two gradients per iteration. To improve upon the extragradient algorithm, \cite{rakhlin2013online} proposed an online optimistic gradient descent (OOGD) method. The OOGD algorithm was also proved to achieve $\bo(\sqrt{1+V_T)}$ regret bounds by requiring only one gradient per iteration, which is suitable for one-point gradient feedback.
 
    For the dynamic regret, \cite{zinkevich2003online} established a dynamic regret bound of $\bo((1+P_T)\sqrt{T})$ for the OGD algorithm, and \cite{jadbabaie15online} derived a dynamic regret bound of $\bo(P_T\sqrt{1+V_T})$ for the OOGD algorithm. However, there is still a gap between these results and the universal lower dynamic regret bound of $\Omega(\sqrt{(1+P_T)T})$ mentioned by \cite{zhang2018adaptive}. To address this issue, \cite{zhang2018adaptive} introduced a meta-expert framework to the OGD algorithm, which led to an optimal dynamic regret bound of $\bo(\sqrt{(1+P_T)T})$. Building upon the work of \cite{zhang2018adaptive}, \cite{zhao2020dynamic} applied the meta-expert framework to the OOGD algorithm using a novel Online Optimistic Hedge technique in the meta-algorithm and obtained a gradient-variation bound $\bo(\sqrt{(1+V_T+P_T)(1+P_T)})$, a small-loss bound $\bo(\sqrt{(1+F_T+P_T)(1+P_T)})$, and a best-of-both-worlds bound $\bo(\sqrt{(1+\min(V_T,F_T)+P_T)(1+P_T)})$ in Euclidean space. In this paper, the proposed R-OOGD algorithm and the R-AOOGD algorithm extend the results of the above Euclidean online optimistic algorithms to Riemannian manifolds beyond the restriction of linear structure.

    \paragraph{Riemannian OCO}Riemannian optimization has garnered significant interest from researchers in the past few decades \citep{zhang2016first,ahn2020nesterov,becigneul2018Riemannian}. In the context of Riemannian online convex optimization, \cite{antonakopoulos2020Online} introduced a regularized method that leverages the Riemann–Lipschitz continuity condition, which specifically targeted convex functions in an ambient Euclidean space. Furthermore, \cite{becigneul2018Riemannian} provided regret analysis for Riemannian versions of the Adagrad and Adam algorithms, which relied on a product manifold structure. Later on, \cite{maass2020online} proposed a Riemannian online zeroth-order algorithm and analyzed the dynamic regret bound of $\bo(\sqrt{T}+P_T^*)$ in the setting of g-strongly convex and g-smooth functions on Hadamard manifolds. Here, $P_T^*$ represents the length of the path between the optimal solutions $P_T^*:=\sum_{t=2}^T d(x_t^*,x_{t+1}^*)$. \cite{wang2023online} proposed Riemannian online gradient methods with a sublinear static regret of $\bo(\sqrt{T})$ in the full information setting, $\bo(T^{2/3})$ in the one-point bandit information setting, and $\bo(\sqrt{T})$ in the two-point bandit information setting for g-convex functions. In this paper, the dynamic regret bound of the proposed R-OOGD algorithm holds for geodesically convex functions on general manifolds and achieves a better static regret $\bo(\sqrt{1+V_T})$ than the above mentioned results. 

    A recent breakthrough in understanding dynamic regret for Riemannian OCO was made by \cite{hu2023minimizing}. \cite{hu2023minimizing} first established a lower bound $\bo(\sqrt{(1+P_T)T})$ for minimax dynamic regrets for the Riemannian OCO. Then, \cite{hu2023minimizing} utilized the Fr\'echet mean and proposed a Riemannian meta-expert structure into the R-OGD, effectively achieving the universal lower bound for minimax dynamic regret. Additionally, \cite{hu2023minimizing} enhanced the meta-algorithm by incorporating Optimistic Hedge methods into Riemannian manifolds, which lead to a gradient-variation regret bound $\bo(\sqrt{(1+V_T+P_T)(1+P_T)})$, a small-loss regret bound  $\bo(\sqrt{(1+F_T+P_T)(1+P_T)})$, and a best-of-both-worlds regret bound $\bo(\sqrt{(1+\min(V_T,F_T)+P_T)(1+P_T)} + \min(V_T,F_T) \log T)$. Moreover, \cite{hu2023minimizing} also extended the regret bounds to the constrained setting using improper learning techniques. In order to obtain gradient-related dynamic regret bounds, \cite{hu2023minimizing} proposed the RADRv algorithm, which utilizes an extragradient-type algorithm (referred to as R-OCEG) as expert algorithm. The main insight of the R-OCEG algorithm is the extrapolation step, where the gradient of the current strategy point is used to extrapolate and obtain a midpoint. The current strategy point is then updated based on the gradient of the extrapolated point. As a result, the R-OCEG algorithm requires two gradient information in each iteration. In contrast, our R-OOGD algorithm does not rely on gradient extrapolation. Instead, the R-OOGD combines the gradient at the current decision point with parallel transported gradients from past decision points to perform update. The approach that parallel transporting past gradients allows our R-OOGD algorithm to achieve the same dynamic regret bound of $\mathcal{O}(P_T\sqrt{1+V_T})$ as the R-OCEG, while requiring only one round of gradient computations each turn. Furthermore, compared to the RADRv, our R-AOOGD algorithm, which employs the R-OOGD as the expert algorithm, achieves the same dynamic regret bound of $\mathcal{O}(\sqrt{(1+V_T+P_T)(1+P_T)})$ while utilizing half of the gradient information.

    \subsection{Optimistic algorithm in zero-sum games}
      \paragraph{ODGA in Euclidean zero-sum games} The Extragradient (EG) and Optimistic Gradient Descent Ascent (OGDA) methods have been extensively researched in the field of finding Nash equilibrium of Euclidean zero-sum games since the work of \cite{korpelevich1976extragradient}. In the unconstrained convex-concave setting, both the EG and the OGDA methods exhibited $\bo(1/T)$ average convergence rates, e.g., \citep[e.g.,][]{mokhtari2020convergence}. \cite{chavdarova2021last} and \cite{gorbunov2022last} proved $\bo(1/\sqrt{T})$ best-iterate convergence rates. In the constrained convex-concave setting, a more recent study by \cite{cai2022tight} demonstrated an $\bo(1/\sqrt{T})$ last-iterate convergence rate of the OGDA. For the strongly convex-strongly concave setting, \cite{mokhtari2020unified} demonstrated linear last-iterate convergence for both the OGDA and EG methods. Furthermore, \cite{wei2021last,gorbunov2022last} extended this linear convergence rate to the constrained setting. In this paper, we propose a Riemannian extension of the OGDA method which preserves the average iterate and best-iterate convergence rate for g-convex-concave games, and linear last-iterate for g-strongly convex-strongly concave games.

     \paragraph{Riemannian zero-sum games} Algorithms for coomputing the Nash equilibria (NE) in Riemannian zero-sum (RZS) games on Riemannian manifolds have been developed in the work of \cite{li2009monotone,wang2010monotone,ferreira2005singularities}. \cite{huang2020gradient} introduced a Riemannian gradient descent ascent algorithm for RZS games where the second variable $y$ lied in a Euclidean space. \cite{zhang2022minimax} presented the Riemannian Corrected Extragradient algorithm (RCEG), which computes two gradients in one iteration and achieves $\bo(\frac{1}{T})$ average-iteration convergence. Furthermore, for non-smooth functions, \cite{NEURIPS2022_2ad9a1a6} presented $\bo(\frac{1}{T})$ average-iterate convergence in the g-convex-concave setting and linear last-iterate convergence in the g-strongly-convex strongly-concave setting using Riemannian gradient descent ascent. \cite{NEURIPS2022_2ad9a1a6} also constructed linear last-iterate convergence of the RCEG in the strongly g-convex setting. \cite{han2022riemannian} introduced Riemannian Hamiltonian methods (RHM) that use the second-order Riemannian Hessian operator and established linear last-iterate convergence when the gradient norm of the payoff function satisfies the Riemannian Polyak–Łojasiewicz (PL) condition. In contrast, our ROGDA algorithm uses first-order information only once in each iteration and achieves the same average-iterate and last-iterate convergence. Moreover, our ROGDA algorithm shows the first best-iterate convergence result for g-convex-concave setting.

\section{Preliminaries}
\paragraph{Riemannian geometry}
A \textit{Riemannian manifold} $(\M,g)$ is a manifold $\M$ with a point-varying Riemannian metric $g$. The Riemannian metric $g$ induces an inner product $\langle u,v\rangle_x = g_x(u,v)$ on every tangent space $T_x\M$. Via the inner product, notions of geometry can be brought onto Riemannian manifolds. For example, the norm of $u\in T_x\M$ is defined as $\|u\| = \sqrt{\langle u,u \rangle_x}$, the angle between $u,v \in T_x\M$ is $\arccos \frac{\langle u,v \rangle_x}{\|u\|\|v\|}$, and the length of a curve $\gamma: [0,1] \to \M$ is defined as $\int_0^1 \|\dot\gamma(t)\| dt$.

A Riemannian manifold $\M$ also enjoys a metric space structure with distance $d(x,y)$, which is the minimum of the lengths of the curves connecting $x$ and $y$. A curve $\gamma$ is called a \textit{geodesic} if it locally reaches the minimum length. An \textit{exponential map} $\exp_x$ acts as a vector addition on Riemannian manifolds, which maps a tangent vector $v\in T_x\M$ to the endpoint $\gamma(1)$ of a geodesic $\gamma$ with the initial tangent vector $v$. \textit{Parallel transport} $\Gamma_\gamma$ along the curve $\gamma$ translates vectors from one tangent space to another while preserving the inner product, i.e., $\langle u,v \rangle = \langle \Gamma_\gamma u, \Gamma_\gamma v\rangle$. In particular, we denote $\Gamma_x^y$ as the parallel transport along the geodesic between $x$ and $y$. The parallel transport determines the covariant derivative of the vector field $X$ along the vector field $Y$, which is defined as $\nabla_X Y(x) = \lim_{t\to0} \frac{1}{t} \big(\Gamma_\gamma X(\gamma(t)) -X(x)\big)$, where $\dot\gamma(0) = Y(x)$.  

One of the most important notions in Riemannian geometry is the curvature tensor, defined as \[R(X,Y,W,Z):= \langle \nabla_X\nabla_Y Z- \nabla_Y\nabla_X Z -\nabla_{[X,Y]} Z, W\rangle. \] The \textit{sectional curvature} is defined as $\frac{R(X,Y,X,Y)}{|X|^2|Y|^2-\langle X,Y\rangle^2}$ and characterizes the non-flatness of a $2$-dimensional Riemannian submanifold. On manifolds with non-positive sectional curvature, geodesics at one point $x$ spread away from each other so that the inverse exponential map $\expinv{x}{}$ can be defined globally, while on manifolds with positive sectional curvature $K$, geodesics at one point gather with each other, making the inverse exponential map well-defined only in a neighborhood of $x$ with diameter less than $\frac{\pi}{\sqrt{K}}$. On the domain where the inverse exponential map is well-defined, we can write the distance function as $d(x,y) = \|\expinv{x}{y}\|$.

\paragraph{Function classes} We introduce some function classes on Riemannian manifolds for the further analysis. First, we introduce the concept of geodesic convexity on Riemannian manifolds. A function $ \f$ is considered to be \textit{geodesically convex} (or g-convex) on $\K$, if for any $x$ and $y$ belonging to $\K$, it satisfies 
\begin{align*}
\f(y) \ge \f(x) + \langle \nabla \f (x), \expinv{x}{y} \rangle.
\end{align*}
A function $\f$ is said to be $\mu$-strongly geodesically convex (or $\mu$-strongly g-convex) on $\K$, if for any $x$ and $y$ belonging to the manifold $\mathcal{M}$, the following inequality holds
\begin{align*}
\f(y) \ge \f(x) + \langle \nabla \f (x), \expinv{x}{y} \rangle + \frac{\mu}{2}d^2(x,y).
\end{align*}
Strong g-convexity also implies that
\begin{align*}
\langle -\nabla \f(x), \expinv{x}{x^*} \rangle \ge \frac{\mu}{2} d^2(x,x^*),\quad \forall x\in \M,
\end{align*}
where $x^*$ is the global minimizer. Furthermore, a function $\f:\K \to \R$ is called \textit{geodesically concave} (g-concave) if $-\f$ is g-convex and a function $\f:\M \to \R$ is called $\mu$-\textit{strongly geodesically concave} ($\mu$-strongly g-concave) if $-\f$ is $\mu$-strongly g-convex.

We now define Lipschitz functions and smooth functions on Riemannian manifolds. We define a function $\f: \M \to \R$ as \textit{geodesically $G$-Lipschitz} (or g-$G$-Lipschitz) if there exists a constant $G>0$ such that, for any $x,x^\prime \in \M$, the inequality $|\f(x) - \f(x^\prime)| \le G \cdot d(x,x^\prime)$ holds. In the case of differentiable functions, this condition is equivalent to $\|\nabla \f(x)\| \le G$ for all $x \in \M$. Similarly, a function $\f: \M \to \R$ is referred to as \textit{geodesically $L$-smooth} if the gradient of $f$ satisfies the g-$L$-Lipschitz property, meaning that for any $x,x^\prime \in \M$, we have $\|\nabla \f(x) -\Gamma_{x^\prime}^x \nabla \f(x^\prime)\| \le L \cdot d(x,x^\prime)$.

We now shift our focus to bivariate functions within the context of Riemannian zero-sum games. We call a bivariate function $\f(x,y):\M \times \N \to \R$ \textit{g-convex-concave} (or \textit{$\mu$-g-strongly convex-strongly concave}), if for every $(x,y)\in\M \times \N$, $\f(\cdot,y): \M \to \R$ is g-convex (or $\mu$-strongly g-convex) and $\f(x,\cdot): \N \to \R$ is g-concave (or $\mu$-strongly g-concave).

\section{Riemannian Online Optimization with Dynamic Regret}\label{sec:oco}

In this section, we first introduce the Riemannian online optimistic gradient descent algorithm (R-OOGD) and the Riemannian adaptive online optimistic gradient descent algorithm (R-AOOGD), which aims to improve the dynamic regret bound by averaging $N$ R-OOGD algorithms with different step sizes. Then, we analyze the dynamic regret bounds of the R-OOGD and the R-AOOGD under g-convex functions.

\subsection{Riemannian Online Optimistic Gradient Descent Method}
The proposed R-OOGD algorithm is described in Algorithm \ref{alg:R-OOGD}. At each iteration $t$, the algorithm collects the gradient $\nabla \ft(x_t)$ and combines it with additional momentum 
\begin{align*}
    \nabla \ft(x_t) - \Gamma_{x_t-1}^{x_t} \nabla \mathbf{f}_{t-1}(x_{t-1}).
\end{align*} This combined gradient is then used in a gradient descent step via the exponential map $\exp_{x_t}$. The R-OOGD algorithm extends the Euclidean optimistic framework \citep{mokhtari2020unified} to Riemannian manifolds.

\begin{algorithm}[H]
    \centering
    \caption{Riemannian Optimistic Gradient Descent Algorithm (R-OOGD)}
    \begin{algorithmic}
    \Require Manifold $\M$, step size $\eta$
    \State Initialize $x_{-1} = x_0 = x_1\in \M$.
    \For {$t$ = $1$ to $T-1$}
        \State Play $x_t$ and receive $\nabla \ft(x_t)$.
        \State Update $x_{t+1} = \exp_{x_t}(-2\eta \nabla \ft(x_t) + \eta\Gamma^{x_t}_{x_{t-1}} \nabla \mathbf{f}_{t-1}(x_{t-1}) )$
    \EndFor
    \Ensure Sequence $(x_t)_{t=1}^T$.
    \end{algorithmic}\label{alg:R-OOGD}
\end{algorithm}

Next, inspired by \cite{hu2023minimizing}, we introduce a meta-expert framework to the R-OOGD. We propose the Riemannian adaptive online optimistic gradient descent algorithm (R-AOOGD) in Algorithms \ref{alg: R-AOOGDmeta} and \ref{alg:R-AOOGDexp}. 

\begin{algorithm}[ht]
\caption{R-AOOGD: Meta Algorithm}\label{alg: R-AOOGDmeta}
\begin{algorithmic}[1]
\Require Manifold $\M$, learning rate $\beta$, step size pool $\mathcal{H} = {\eta_i;i=1,2,\dots,N}$ and parameters $K,\alpha$.
\State Initialize $x_0 \in \M$. Set initial weights $w_{0,1}=w_{0,2}=\dots=w_{0,N} = 1/N$.
\For{$t = 1$ to $T$}
\State Receive $x_{i,t}$ from $N$ expert algorithms with step size $\eta_i$.
\State Set $\bar x_t = \arg\min_{x} w_{t-1,i} d^2(x,x_{t,i})$
\State Update $w_{t,i} \propto e^{\big(-\beta(\sum_{j=1}^{t-1}l_{i,j} + m_{i,t})\big)}$ by
\begin{align*}
    \begin{cases}
        l_{i,t} = \langle  \nabla \ft(x_t) , \expinv{x_t}{x_{i,t}} \rangle\\
        m_{i,t} = \langle \nabla \ft(\bar x_t) , \expinv{\bar x_t}{x_{i,t}} \rangle\\
    \end{cases}
\end{align*}
\State Set $x_t = \arg\min_{x} w_{t,i} d^2(x,x_{t,i})$.
\EndFor
\State \Return $\{x_t\}_{t=1}^{T}$.
\end{algorithmic}
\end{algorithm}

\begin{algorithm}[ht]
\centering
\caption{R-AOOGD: Expert Algorithm}\label{alg:R-AOOGDexp}
\begin{algorithmic}
\Require Manifold $\M$, feasible set $\K$ and step size $\eta_i$ from the pool $\mathcal{H}$.
\State Initialize $x_{i,-1} = x_{i,0} = x_{i,1}\in \K$.
\For {$t$ = $1$ to $T-1$}
\State Send $x_{i,t}$ to the meta algorithm.
\State Update $x_{i,t+1} = \exp_{x_{i,t}}(-2\eta_i \nabla \ft(x_{i,t}) + \eta_i\Gamma^{x_{i,t}}_{x{i,t-1}} \nabla \mathbf{f}_{t-1}(x_{i,t-1}) )$.
\EndFor
\end{algorithmic}
\end{algorithm}

\subsection{Dynamic Regret Analysis}
In order to analyze the regret bounds of the R-OOGD and the R-AOOGD, we impose some assumptions, which are standard in the literature of online learning and Riemannian optimization \citep{antonakopoulos2020Online,mokhtari2020convergence,mokhtari2020unified,ahn2020nesterov,alimisis2021momentum}.

\begin{assumption}\label{asm:start1}
    The function $\ft$ is g-convex, g-$G$-Lipschitz, and g-$L$-smooth over the set $\K \subset \M$.
\end{assumption}

The following two assumptions focus on the geometry of the manifolds $\M$.
\begin{assumption}
    All sectional curvatures of $\M$ are bounded below by a constant $\kappa$ and bounded above by a constant $K$.
\end{assumption}
\begin{assumption}\label{asm:bound}
    The diameter of the feasible set $\K$ is bounded by $D_0$. If $K>0$, the diameter $D_0$ is less than $\frac{\pi}{2\sqrt{K}}$.
\end{assumption}

Assuming boundedness of the feasible set is a fundamental setting in online optimization algorithms \citep{zinkevich2003online,abernethy2008optimal,zhao2020dynamic}, while the additional constraint $D_0 \le \frac{\pi}{2\sqrt{K}}$ is also a common condition adopted in the literature of Riemannian optimization on positively curved manifolds \citep{zhang2016first, alimisis2021momentum, zhang2022minimax}. Assumption \ref{asm:bound} ensures that there are no conjugate points on $\K$ according to the conjugate point theorem \citep{lee2018introduction}, guaranteeing that the inverse exponential map $\exp^{-1}_{x}(\cdot)$ can be defined throughout $\K$. Additionally, the Hessian comparison theorem \citep{lee2018introduction} indicates that when the diameter of $\K$ is greater than $\frac{\pi}{\sqrt{K}}$, the subset $\K$ may be  \emph{``infinitely curved,''} meaning that the Hessian of the distance function $d(x,\cdot)$ may blow up. As in Riemannian optimization, we usually bound the loss by the g-convexity, i.e.,
\begin{align*}
    \ft(x_t)-\ft(u_t) &\le \langle \nabla \mathbf - \f_t(x_t) , \expinv{x_t}{u_t} \rangle = \langle \nabla \mathbf f_t(x_t) , \nabla_{x_t}(\frac{1}{2}d^2(x_t,u_t)) \rangle,
\end{align*} and then use the Hessian of the distance function $d(x,u_t)$ for further analysis. 

We also note that Assumption \ref{asm:bound} does not necessarily affect the applicability of the proposed algorithms in practical problems. Our experiments (see Subsection \ref{exp: ogr}) demonstrate that a feasible set $\K$ with a much larger diameter than $\frac{\pi}{2\sqrt{K}}$ (even the whole manifold $\M$) does not significantly affect the performance of our proposed algorithms.
 
We impose the following assumption on invariance of the set $\K$ during the execution of Algorithm \ref{alg:R-OOGD}. The same assumption has also been used in the work by \cite{ahn2020nesterov,alimisis2021momentum,zhang2022minimax}. In addition, and we do not observe the assumption to be violated in our experiments. 
    
\begin{assumption}\label{asm:end1}
    All iterations of Algorithms \ref{alg:R-OOGD} lie in the set $\K$.
\end{assumption}

Now we begin to analyze regret bounds. Our analysis heavily relies on comparison inequalities \citep{zhang2016first,alimisis2021momentum}, which enable us to quantify the distortion by nonlinear structure with respect to the curvature bound $\kappa$, $K$ and domain diameter $D_0$. Specially, we can bound the minimum and maximum distortion rates by two parameters 
    \begin{align*}
            \sigma(K,D) = 
            \begin{cases}
                \frac{\sqrt{K}D}{\tan(\sqrt{K}D)} & K>0; \\
                1 &  K\le 0,
            \end{cases}
            \quad \text{and} \quad \zeta(\kappa,D) = 
            \begin{cases}
                \frac{\sqrt{-\kappa}D}{\tanh(\sqrt{-\kappa}D)} & \kappa<0; \\
                1 &  \kappa \ge 0.
            \end{cases}
    \end{align*}
After that, we can establish regret bounds for the R-OOGD algorithm.
\subsubsection{Dynamic Regret for R-OOGD}
\begin{theorem}\label{thm: reg-c}
    Let $\sigma_0 = \sigma(K,D_0)$ and $\zeta_0 = \zeta(\kappa,D_0)$. Suppose that Assumptions \ref{asm:start1}-\ref{asm:end1} hold. Then, for the sequence $\{x_t\}_{t=1}^T$ generated by Algorithm \ref{alg:R-OOGD} with step size $\eta \le \frac{\sigma_0}{4\zeta_0 L}$, the following inequality holds for an arbitrary sequence $\{u_t\}_{t=1}^T$:
\begin{align*}
\RegD(u_1,\dots,u_T) \le \frac{D_0^2+2D_0P_T}{\eta} + \eta \frac{4\zeta_0^2}{\sigma_0} (G^2+V_T),
\end{align*}
where $P_T = \sum_{t=2}^{T} d(u_t,u_{t-1})$ is the path-length,  and $V_T= \sum_{t=2}^{T}\sup_{x\in \K} \|\nabla \ft(x) - \nabla \mathbf{f}_{t-1}(x)\|^2$ is the gradient variation. If $V_T$ is known beforehand, the dynamic regret bound can be improved to $\bo(P_T\sqrt{1+V_T})$.
\end{theorem}

\begin{corollary}\label{cor:reg-c}
    Suppose that Assumptions \ref{asm:start1}-\ref{asm:end1} hold. Algorithm \ref{alg:R-OOGD} incurs the static regret with the optimal step size $\eta^* = \min( \sqrt{\frac{D_0^2\sigma_0}{4\zeta_0^2(G^2+V_T)}}, \frac{\sigma_0}{4\zeta_0 L})$, 
    \begin{align*}
        \RegS(T) \le \bo\big(\zeta_0\sqrt{\frac{1}{\sigma_0}(1+V_T) }  \big).
    \end{align*}
\end{corollary}

The proofs of Theorem \ref{thm: reg-c} and Corollary \ref{cor:reg-c} are shown in Appendix \ref{app:reg}. In Theorem \ref{thm: reg-c} and Corollary \ref{cor:reg-c}, we provide an $\bo(P_T\sqrt{1+V_T})$ dynamic regret bound and an $\bo(\sqrt{1+V_T})$ static regret bound for the R-OOGD, respectively. These bounds recover the corresponding work on online optimization in Euclidean space \citep{jadbabaie15online,zhang2018adaptive,zhao2020dynamic}. Compared to the 
extragradient-structured R-OCEG algorithm by \cite{hu2023minimizing} in the unconstrained setting, our method only requires one gradient information per iteration, making it less computationally demanding and more applicable to the one-point gradient feedback online learning model.
\subsubsection{R-OOGD with Parallel Transport}
Conversion of past gradients plays a vital role in designing Riemannian optimistic algorithms.  Apart from the parallel transport method utilized in our R-OOGD algorithm, there is another methodology known as the correction term to transport past graident in Riemannian optimization \citep{zhang2022minimax,hu2023minimizing}. The correction term involves changing the base point of the exponential map to avoid parallel transport. Specifically, when considering two points $a$ and $b$, along with a tangent vector $v\in T_a\M$, the correction term employs $\exp_a(v+\exp_a^{-1}(b))$ instead of the direct parallel transport $\exp_b(\Gamma_a^b v)$. If our R-OOGD algorithm incorporates the correction term methodology, the algorithm will be expressed as follows,
\begin{align}\label{eq:corr}
    \begin{cases}
         x_{t+1} = \exp_{x_t}(-2\eta \nabla \ft(x_t)+\expinv{x_t}{\hat{x}_{t}}) \\
         \hat{x}_{t+1} = \exp_{x_{t}}(-\eta\nabla \ft(x_{t})+\expinv{x_{t}}{\hat x_{t}}).
    \end{cases}
\end{align}

However, incorporating the correction term into our R-OOGD algorithm does not yield an efficient online algorithm and is unable to guarantee a sublinear static regret. This is because the correction term introduces distortion in the inner product, and these distortions grow unboundedly over time. To be more precise, in order to analyze the static regret bound, we need to estimate the distortion in the inner product between the correction term and the real gradient term 
\begin{align*}
    A_t := \langle \expinv{x_t}{\hat{x}_{t}}- \Gamma_{x_{t-1}}^{x_t} \eta \nabla \f_{t-1}(x_{t-1}) , \expinv{x_t}{x} \rangle
\end{align*} at each iteration. However, as demonstrated in Appendix \ref{app:correct}, the distortion bound $A_t$ depends on the previous distortion $A_{t-1}$, which is governed by the recursive formula:
\begin{align}\label{eq:A_t}
A_{t} \le \bo\Big(( 5\eta G + 2A_{t-1})^2 (3\eta G+A_{t-1})\Big).
\end{align}

Therefore, the distortion accumulates iteratively and may eventually blow up. In contrast, our parallel transport method preserves the inner product, guaranteeing that the distortion term remains zero. Consequently, it is crucial to emphasize that parallel transporting past gradients is essential to ensure the effectiveness of the R-OOGD algorithm.

\subsubsection{Dynamic Regret for R-AOOGD}
Now we begin to analyze the the dynamic bound of the R-AOOGD algorithm.
\begin{theorem}\label{thm:aoogd}
 Suppose that Assumptions \ref{asm:start1}-\ref{asm:end1} hold.
Set $\mathcal{H} = \{ \eta_i = 2^{i-1} \sqrt{\frac{\sigma_0 D_0^2}{16\zeta_0^2 G^2 T}} \}$, $N = \Big\lceil \frac{1}{2}\log_2\big(\frac{\sigma_0 G^2 T}{D_0^2 L^2}\big) \Big\rceil+1$  and $\beta = \min\Big( \frac{1}{\sqrt{12D_0^4L^2+D_0^2G^2\zeta_0^2}} , \sqrt{\frac{2+\ln N}{3D_0^2(V_T+G^2)}}\Big)$. The R-AOOGD algorithm (Algorithms \ref{alg: R-AOOGDmeta} and \ref{alg:R-AOOGDexp}) incurs the dynamic regret
\begin{align*}
\RegD(u_1,\dots,u_T) \le \bo\Big(\frac{\zeta_0}{\sqrt{\sigma_0}}\sqrt{(1+V_T+P_T)(1+P_T)}\Big).
\end{align*}
\end{theorem}

The proof of Theorem \ref{thm:aoogd} can be found in Appendix \ref{app:reg}. In  Theorem \ref{thm:aoogd},An $\bo\Big(\sqrt{(1+V_T+P_T)(1+P_T)}\Big)$ regret bound is established for our meta-expert algorithm R-AOOGD which aligns with the findings in online optimization in Euclidean space \citep{zhao2020dynamic}. When considering the RADRv algorithm \citep{hu2023minimizing}, which is another Riemannian online meta-expert algorithm, we observe certain differences compared to our R-AOOGD algorithm. In the non-projection case, the RADRv algorithm at the meta level requires $N = \Big\lceil \frac{1}{2}\log_2\big(\frac{4\sigma_0 G^2 T}{D_0^2 L^2}\big) \Big\rceil+1$ expert algorithms, which is equal to or more than the number of experts required by our R-AOOGD algorithm. Additionally, at the expert level, the RADRv requires two gradients per iteration, while our R-AOOGD algorithm only requires one gradient. As a result, our R-AOOGD algorithm achieves the same order of regret bound of $\bo\Big(\sqrt{(1+V_T+P_T)(1+P_T)}\Big)$ while utilizing only half the number of gradients.

\section{Application: Nash Equilibrium Seeking in Riemannian Zero-sum Games}

In this section, we apply our R-OOGD to Riemannian zero-sum (RZS) games, where both players update their actions based on the R-OOGD dynamics. We then analyze the convergence rates of the resulting Riemannian Optimistic Gradient Descent Ascent method (R-OGDA). Specifically, we study the average-iterate convergence rate and the best-iterate convergence rate for g-convex-concave games, as well as the last-iterate convergence rate for g-strongly convex-strongly concave games.

\subsection{Formulation of Riemannian Zero-sum Games}
Riemannian zero-sum (RZS) games involve a competitive scenario between two players, denoted as $\mathtt X$ and $\mathtt Y$, i.e.,
\begin{align}\label{eq:RZS}
    \min_{x\in\M} \max_{y\in \N} \mathbf \f(x,y),
\end{align}
where player-$\mathtt X$ tries to find a strategy $x$ from a Riemannian manifold $\M$ to minimize the payoff function $\f(x,y)$, while player-$\mathtt Y$ tries find a strategy $y$ from a Riemannian manifold $\N$ to maximize the payoff function $\f(x,y)$.

A key concept in the RZS games is \textit{Nash Equilibrium} (NE). In the RZS game, an action pair $(x^*,y^*) \in \M \times \N$ is an NE if no player can improve individual payoff by only deviating his own action, i.e.,
\begin{align*}
    \max_{y\in \N} \f(x^*,y) = \f(x^*,y^*) = \min_{x \in \M} \f(x,y^*).
\end{align*}
Computing NEs in the Riemannian zero-sum games has direct applications in several learning tasks including minimum balanced cut, robust geometry-aware PCA, and robust Wasserstein barycenters \citep{khuzani2017stochastic,horev2016geometry,lin2020projection,zhang2022minimax}.

\subsection{Riemannian Optimistic Gradient Descent Ascent Algorithm}
We propose the Riemannian Optimistic Gradient Descent Ascent Algorithm (R-OGDA) as follows.
\begin{algorithm}[H]
    \centering
    \caption{Riemannian Optimistic Gradient Descent Ascent}
    \begin{algorithmic}
    \Require $(\M,\N$), step size $\eta$, 
    \State Initialize $(x_{-1},y_{-1})=(x_0,y_0)\in \M\times\N$.
    \For {$t$ = $2$ to $T-1$}
        \State Update $x_{t+1} = \exp_{x_t}(-2\eta \nabla_{x} \f(x_t,y_t) + \eta\Gamma^{x_t}_{x_{t-1}} \nabla_x \f(x_{t-1},y_{t-1}) )$
        \State Update $y_{t+1} = \exp_{y_t}(2\eta \nabla_{y} \f(x_t,y_t) - \eta\Gamma^{y_t}_{y_{t-1}} \nabla_y \f(x_{t-1},y_{t-1}) )$
        \State Average from geodesic 
        \begin{align}\label{eq: avg scheme}
        \begin{cases}
            \bar x_{t+1} = \exp_{\bar x_t}( \frac{1}{(t+1)} \exp^{-1}_{\bar x_t} x_{t+1} )\\
            \bar y_{t+1} = \exp_{\bar y_t}( \frac{1}{(t+1)}\exp^{-1}_{\bar y_t} y_{t+1} )
        \end{cases}
        \end{align}
    \EndFor
    \Ensure Sequence $(x_t,y_t)_{t=1}^T$, average $(\bar x_T, \bar y_T)$ .
    \end{algorithmic}\label{alg:R-OGDA}
\end{algorithm}

\subsection{Average-Iterate Analysis}
     We first analyze the convergence rate of averaged iterate $(\bar x_T,\bar y_T)$ by adpoting the following assumptions. To ease the notation, we denote $z_t=(x_t,y_t)$, $z^* = (x^*, y^*)$, and $\F(z_t) = [\nabla_{x} \f(z_t), -\nabla_{y} \f(z_t)]$.
    \begin{assumption}\label{asm:start2}
        The payoff function $f$ is g-$G$-Lipschitz and g-convex-concave on the manifold $\M \times \N$.
    \end{assumption}
    The g-convexity-concavity is helpful in analyzing Nash equilibriums. For any NE point $z^*=(x^*,y^*)$, the g-convexity-concavity of $\f$ implies that 
    \begin{align*}
        \langle -\F(z), \expinv{z}{z^*} \rangle \ge 0
    \end{align*} holds for all $z \in \M$. Moreover, if $\f$ is $\mu$-g-strongly-convex  strongly-concave, it further holds that for all $z \in \M$,
    \begin{align}\label{eq: strong-cc}
        \langle -\F(z), \expinv{z}{z^*} \rangle \ge\frac{\mu}{2} d^2(z, z^*).
    \end{align}
    \begin{assumption}
        The payoff function $f$ is g-$L$-smooth on $\M \times \N$, i.e., there is a constant $L>0$ such that
        \begin{align*}
            \|\nabla \f(x,y) -\Gamma_{(x^\prime,y^\prime)}^{(x,y)} \nabla \f(x^\prime,y^\prime)\|^2 \le L^2 (d^2(x,x^\prime) + d^2(y,y^\prime)), \quad \forall (x,y),(x^\prime,y^\prime) \in \M\times\N.
        \end{align*}
    \end{assumption}

    \begin{assumption}\label{asm:ne}
        There is a Nash equilibrium $z^* = (x^*,y^*)$ in the Riemannian zero-sum game \eqref{eq:RZS}.
    \end{assumption}
    Assumption \ref{asm:ne} can be directly derived from the g-strongly convex-strongly concave property of $\f$ in certain cases. In the case where the manifold $\M$ is compact or the level set $\{(x, y) \mid -\infty \le \f(x, y) \le \infty\}$ is bounded, we can rely on the Riemannian analog of Sion's minimax theorem \citep{zhang2022minimax} to establish the existence of a Nash equilibrium. Furthermore, if we assume g-strong-convexity strong-concavity of the function $\f$, by \eqref{eq: strong-cc}, Assumption \ref{asm:ne} guarantees the uniqueness of the resulting Nash equilibrium $z^*$.
    \begin{assumption}\label{asm:end2}
        All sectional curvatures of $\M$ and $\N$ are bounded below by a constant $\kappa$ and bounded above by a constant $K$.
    \end{assumption}

    We first show that it is possible to drop the boundedness assumption (Assumption \ref{asm:bound}) in the iteration of $x_t$ and $y_t$ in the Riemannian zero-sum (RZS) games \eqref{eq:RZS}, if prior knowledge of the distance $d(z_0,z^*)$ is available. The following Lemma \ref{lemma:bdd} ensures that the iteration of $x_t$ and $y_t$ is in a bounded set $\K$.

    \begin{lemma}\label{lemma:bdd}
        Suppose Assumptions \ref{asm:start2}-\ref{asm:end2} hold. Let the step size $\eta$ satisfying $\eta \le \min(\frac{\sigma_1}{\zeta_1 L}, \frac{D_1}{3G})$. Let $d(z_0,z^*) \le D_1 < \frac{\pi}{6\sqrt{K}}$, then by denoting $\zeta_1 = \zeta(\kappa,3D_1)$ and $\sigma_1 = \zeta(K,3D_1)$, we have $d(z_t,z^*) \le 2D_1$ for all iterations $t$.
    \end{lemma}
    Then we derive the average-iterate convergence rate for g-convex-concave RZS games in the following Theorem \ref{thm: avg-c}.
    \begin{theorem}\label{thm: avg-c}
         Under Assumptions \ref{asm:start2}-\ref{asm:end2} and the condition in Lemma \ref{lemma:bdd}, the averaged iterate $(\bar x_T,\bar y_T)$ of Algorithm \ref{alg:R-OGDA} with the step size $\eta \le \frac{\sigma_1}{2 \zeta_1 L}$ satisfies:
         \begin{align*}
              \max_{y\in\N}\f(\bar x_T, y ) - \min_{x\in\M}\f(x, \bar y_T ) \le \frac{D_1^2L+ \frac{D_1^2}{2\eta}}{T}
         \end{align*}
    \end{theorem}
    The proof of Theorem \ref{thm: avg-c} is in Appendix \ref{app:avg}. Theorem \ref{thm: avg-c} demonstrates that the averaged iterate $(\bar x_T,\bar y_T)$ is an $\bo(\frac{1}{T})$-NE in g-convex-concave RZS games. The result extends the corresponding results in Euclidean spaces \citep{mokhtari2020unified} and recovers the convergence rate of RCEG \citep{zhang2022minimax} by acquiring only one gradient in each iteration.
    
\subsection{Last-Iterate/Best-Iterate Analysis}
In this section, we focus on the last/best-iterate convergence of our R-OGDA algorithm. Dealing with the convergence of the R-OGDA algorithm in terms of the last-iterate or best-iterate is quite challenging. In Euclidean spaces, optimistic gradient descent/ascent algorithms benefit from the perspective of ``extrapolation from the past'' \citep{gidel2018a}. This means that by defining the immediate sequence
\begin{align*}
\hat{z}^E_{t+1} = z_t - \eta(\F(z_t) + \F(z_{t-1})), \tag*{(*)} \label{eq:eu1}
\end{align*}
the relationship holds \cite{gidel2018a}:
\begin{align*}
    \eta \F(z_{t}) = \hat{z}^E_t - \hat{z}^E_{t+1}. \tag*{(**)} \label{eq:eu2}
\end{align*} However, when we define the Riemannian counterparts of \ref{eq:eu1} $\hat{z}_{t+1} = \exp_{z_t}(\eta(-\F(z_t) + \F(z_{t-1})))$, it fails to hold that $\expinv{\hat z_{t+1}} { \hat z_{t} } = \eta \Gamma_{z_{t}}^{\hat{z}_{t+1}} \F(z_{t})$, which is the Riemannian version of \ref{eq:eu2}. In addition, measuring the distortion between vector $\expinv{\hat{z}_{t+1} }{\hat{z}_{t} }$ and $\eta \Gamma_{z_{t}}^{\hat{z}_{t+1}} \F(z_{t})$ is challenging.

One natural idea is to parallel transport the tangent vectors $\expinv{z_t}{z_{t-1}}$, $\expinv{z_t}{\hat z_t}$ and $\expinv{z_{t+1}}{\hat z_{t+1}}$ to the tangent space at $\hat z_{t+1}$, and then estimate the distortion between $G_t$ and $\Gamma_{z_{t}}^{\hat{z}{t+1}} \F(z_{t})$ using their vector sum. A commonly used technique in manifold settings, the comparison inequalities \citep{zhang2016first,alimisis2021momentum}, follows the same idea to estimate the Hessian-type distortion, i.e., 
\begin{align*}
    \sigma(K,D) \|\expinv{a}{b}\| \le \| \exp^{-1}_a b - \Gamma_c^a \exp^{-1}_c b \| \le \zeta(k,D) \|\expinv{a}{b}\|.
\end{align*} The crux of these inequalities is to use the gradient squared distance $\nabla d(b,x) = - \expinv{b}{x}$ on the manifold, and then estimate the error by exploiting the eigenvalues of the Hessian matrix of the distance function. Therefore, the comparison inequalities require the compared tangent vectors to share the same endpoint, which is not applicable in our case.

Furthermore, we have observed that the distortion analysis between $G_t$ and $\Gamma_{z_t}^{\hat{z}{t+1}} \F(z_t)$ brings in extra errors beyond the Hessian-type distortions. Parallel transporting vectors in the two distinct tangent spaces to a third tangent space can lead to additional distortion. Specifically, when dealing with three points $a,b,c$, and vectors $v\in T_{a}\M$ and $w\in T_{b}\M$, we have \begin{align*}
   \| \Gamma_a^c v - \Gamma_b^c w \| \neq \| \Gamma_a^b v -w \|. \end{align*}
    In particular, it holds 
    \begin{align*}
        \|\Gamma_a^c v - \Gamma_b^c w \| = \| \Gamma_a^b v -w \| + \| \Gamma_b^a \Gamma_c^b \Gamma_a^c v -w\| .
    \end{align*}
 On Riemannian manifolds, the latter term $\| \Gamma_b^a \Gamma_c^b \Gamma_a^c v -w\|$ does not vanish due to the presence of the holonomy effects, which indicates that parallel transport around closed loops fails to preserve the geometric data being transported. Therefore, estimating the distortion between $G_t$ and $\Gamma_{z_{t}}^{\hat{z}{t+1}} \F(z_{t})$ is faced with the change of holonomy distortion and becomes highly nontrivial.

To address the aforementioned difficulties, we propose a key technique inspired by the famous Gauss-Bonnet theorem \citep{chern1999lectures}, and estimate the holonomy distortion by the step size $\eta$ and sectional curvature $K_m = \max\{|K|,|\kappa|\}$.

\begin{lemma}\label{lemma:key}
    Suppose that $\f$ is g-$L$-smooth. If the step size $\eta \le \frac{1}{20L}$, them $G_{t+1} =  \expinv{\hat z_{t+1}} { \hat z_{t} }$ satisfies the following
    \begin{itemize}
        \item [(i)] $\|G_{t+1}\|^2 - \| \eta \F(z_{t}) \|^2 \le 64 K_m^2 \eta^6 \| \F(z_{t-1})\|^6  $;
        \item [(ii)] $\|\Gamma_{\hat z_{t+1}}^{z_{t}} G_{t+1} -\eta \F(z_{t}) \|\le  104 K_m\eta^3\|  \F(z_{t-1}) \|^3$.
    \end{itemize}
\end{lemma}
According to Lemma \ref{lemma:key}, the holonomy distortion turns out to be $\bo(K_m \eta^3)$, which enables us to obtain the convergence in Theorems \ref{thm: last-c} and \ref{thm: last-sc}. In the following theorems, we denote $\Upsilon = \frac{1}{5} \sigma_1 L + \frac{28}{5} (\zeta_1 - \sigma_1)L + 104(2D_1+\frac{1}{5})K_m G + 8 \sigma_1 K_m G $.
\begin{theorem}\label{thm: last-c}
Suppose Assumptions \ref{asm:start2}-\ref{asm:end2} and conditions in Lemma \ref{lemma:bdd} hold. Algorithm \ref{alg:R-OGDA} incurs the best-iterate convergence with the step size $\eta \le \min\{\frac{1}{20L},\frac{1}{8G}, \frac{\sigma_1}{2\Upsilon}\}$,
\begin{align*}
    \min_{t\le T} \| \nabla \f(z_t)\| \le \bo(\frac{1}{\sqrt{T}}),
\end{align*}
and moreover, we have $\lim_{t\to\infty}  \| \nabla \f(z_t)\| = 0$. 
\end{theorem}

\begin{theorem}\label{thm: last-sc}
Suppose Assumptions \ref{asm:start2}-\ref{asm:end2} and conditions in Lemma \ref{lemma:bdd} hold and $\f$ is $\mu$-g-strongly convex-strongly concave. Recall $z^*$ is the unique NE of the game \eqref{eq:RZS}. Algorithm \ref{alg:R-OGDA} with the step size\\ $\eta \le \min\{\frac{1}{20L},\frac{1}{8G}, \frac{\sigma_1}{\Upsilon +4\mu + 8\sigma_1 \mu}\}$ incurs
\begin{align*}
    d^2(z_t,z^*) \le \Big(\frac{1}{1 + \eta \mu/ 2 } \Big)^t \big(2(1+\frac{1}{\sigma_1}) d^2(z_1,z^*)\big).
\end{align*}
\end{theorem}

The proofs for Theorems \ref{thm: last-c} and \ref{thm: last-sc} can be found in Appendix \ref{app:last}. Theorem \ref{thm: last-c} demonstrates the $\bo(\frac{1}{\sqrt{T}})$ convergence rate for the best iterate of g-convex-concave games, which is the first proven result in Riemannian NE seeking algorithms for g-convex-concave games. Moreover, Theorem \ref{thm: last-sc} establishes a linear convergence rate for the last-iterate of the ROGDA algorithm, matching those in the RCEG algorithm and the second-order RHM algorithm \citep{han2022riemannian,NEURIPS2022_2ad9a1a6} in the g-strongly convex-strongly concave setting, while requiring only one first-order information in each iteration.

\section{Numerical Experiments}
 In this section, we presents several numerical experiments to validate our theoretical findings regarding Riemannian online optimization problems and Riemannian zero-sum games. We conduct experiments on both synthetic and real-world datasets, and compare the performance of our proposed algorithm with state-of-the-art methods in the literature. We implement our algorithm using the Pymanopt package \cite{JMLR:v15:boumal14a} and conduct all experiments in Python 3.8 on a machine with an AMD Ryzen5 processor clocked at 3.4 GHz and 16GB RAM. To ensure reproducibility of our results, we provide access to all source codes online\footnote{\url{https://github.com/RiemannianOCO/DynamicReg}}.

 \subsection{Online Fr\'echet Mean in the Hyperbolic Space}

The Fr\'echet mean problem, also known as finding the Riemannian centroid of a set of points on a manifold, has numerous applications in various fields, including diffusion tensor magnetic resonance imaging (DT-MRI) \citep{cheng2012efficient} and hyperbolic neural network \citep{liu2019hyperbolic}. We focus on the online version of the Fr\'echet mean problem, which aims to compute the average of $N$ time-variant points in a hyperbolic space. Hyperbolic space is a Riemannian manifold with constant negative sectional curvature -1, defined as 
\begin{align*}
    H^n = \{ x \in \R^{n+1} | \langle x, x \rangle_M = -1 \},
\end{align*} where the Minkowski dot product 
\begin{align*}
    \langle x,y \rangle_M = \sum_{i=1}^n x_i y_i - x_{n+1}y_{n+1}
\end{align*} defines the metric $\langle x,y \rangle_p = \langle x,y \rangle_M$. The loss function $\ft$ of the online Fr\'echet mean problem is given by 
\begin{align*}
    \ft(x_t) = \frac{1}{2N} \sum_{i=1}^{N} d^2(x_t,A_{t,i}) = \frac{1}{2N} \sum_{i=1}^{N} \cosh^{-1}(-\langle {x_t}, A_{t,i} \rangle_M)^2,
    \end{align*} where $\{ A_{t,1},A_{t,2},\dots,A_{t,N} \}$ are the time-variant points in the hyperbolic space.

   \paragraph{Experimental setting} We test the online Fr\'echet mean on synthetic datasets generated as follows: In each iteration $t$, we randomly sample the point $A_{i,t}$ from a ball centered at a point $P_t$ with a radius of $c=1$. The choice of $P_t$ is in following two ways to simulate non-stationary environments. 1) $P_t$ remains fixed in between $S$ time steps. After every $S$ rounds, we re-select $P_t$ randomly within a bounded set of diameter $D=1$ to simulate abrupt changes in the environment. 2) $P_t$ shifts a small distance of $0.1$ between each time step, and is re-selected after $S$ round, to simulate a slowly evolving environment. For our experiment, we set $T=10000$, $n=100$, $d=20$, $D=1,\kappa =1$, and $L=\zeta(\kappa,D)$.

    In addition, we compare our R-OOGD and R-AOOGD algorithms to other contenders in Riemannian online optimization. Specifically, we compare R-OOGD with the Riemannian online gradient descent (ROGD) algorithm \citep{wang2023online} and the EG-type expert algorithm R-OCEG \citep{hu2023minimizing}. We also compare our meta-expert algorithm R-AOOGD with the meta-expert RADRv algorithm \citep{hu2023minimizing}.

    \paragraph{Result} We examine the performance in terms of cumulative loss and present the result in Fig \ref{fig: fre}. First, we can see that OGD suffers from a high cumulative loss throughout the horizon. Conversely, our methods, as well as R-OCEG and RADRv, demonstrate satisfactory performance in terms of dynamic regret in both situations. As meta-expert algorithms, our R-AOOGD slightly outperforms the RARDv. Our R-OOGD algorithm performs comparably to RARDv-exp, but with only half number of the gradient information required. These facts validate the effectiveness of our algorithms and demonstrate our advantage in situations where gradient computation is time-consuming.

    \begin{figure}[htbp]\label{fig: fre}
		\centering 
		\subfigure[Stationary environment with abrupt changes]{
			\begin{minipage}[t]{0.4\linewidth}
				\centering
				\includegraphics[width=\linewidth]{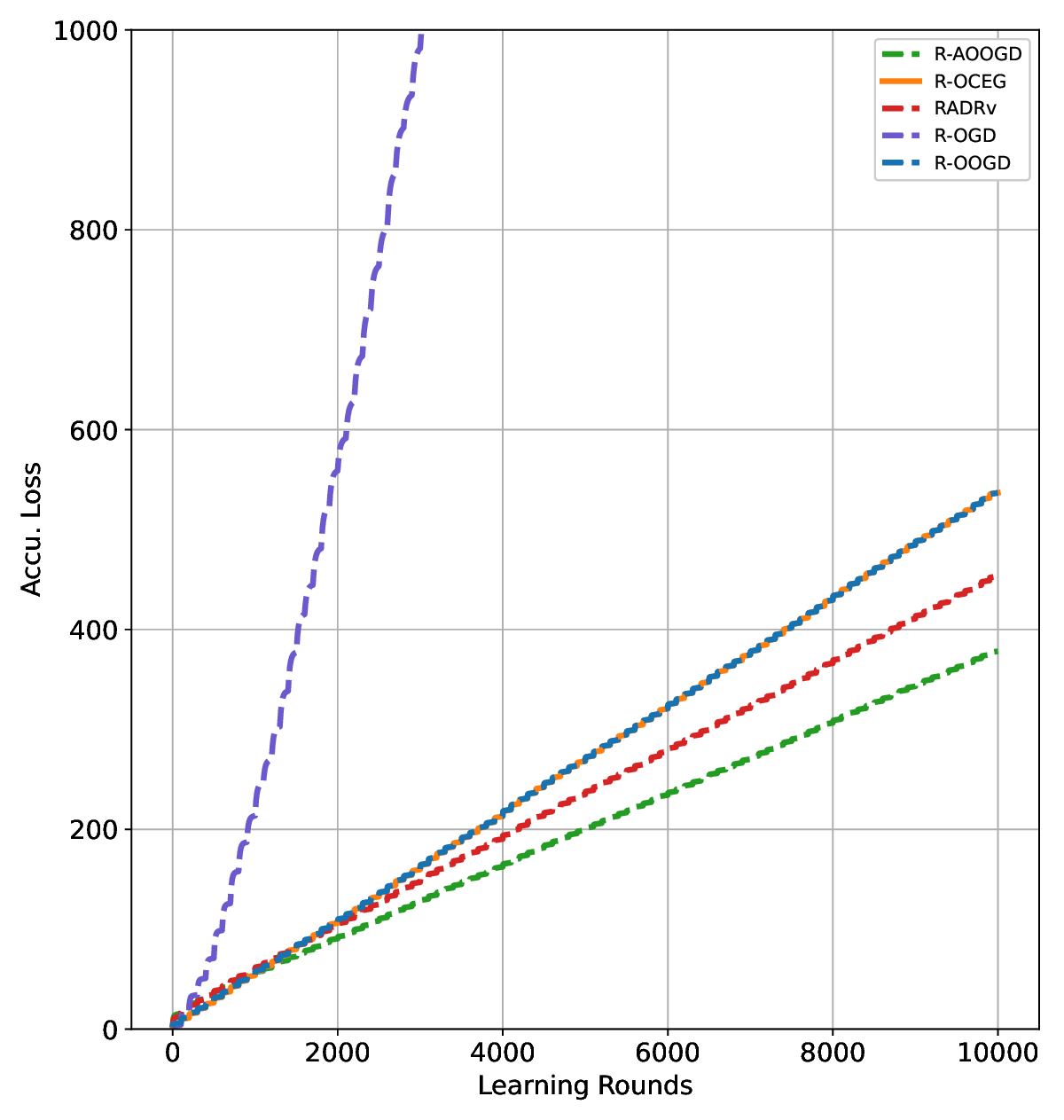}
			\end{minipage}%
        }
		\subfigure[Slowly evolving environment]{
			\begin{minipage}[t]{0.4\linewidth}
				\centering
				\includegraphics[width=\linewidth]{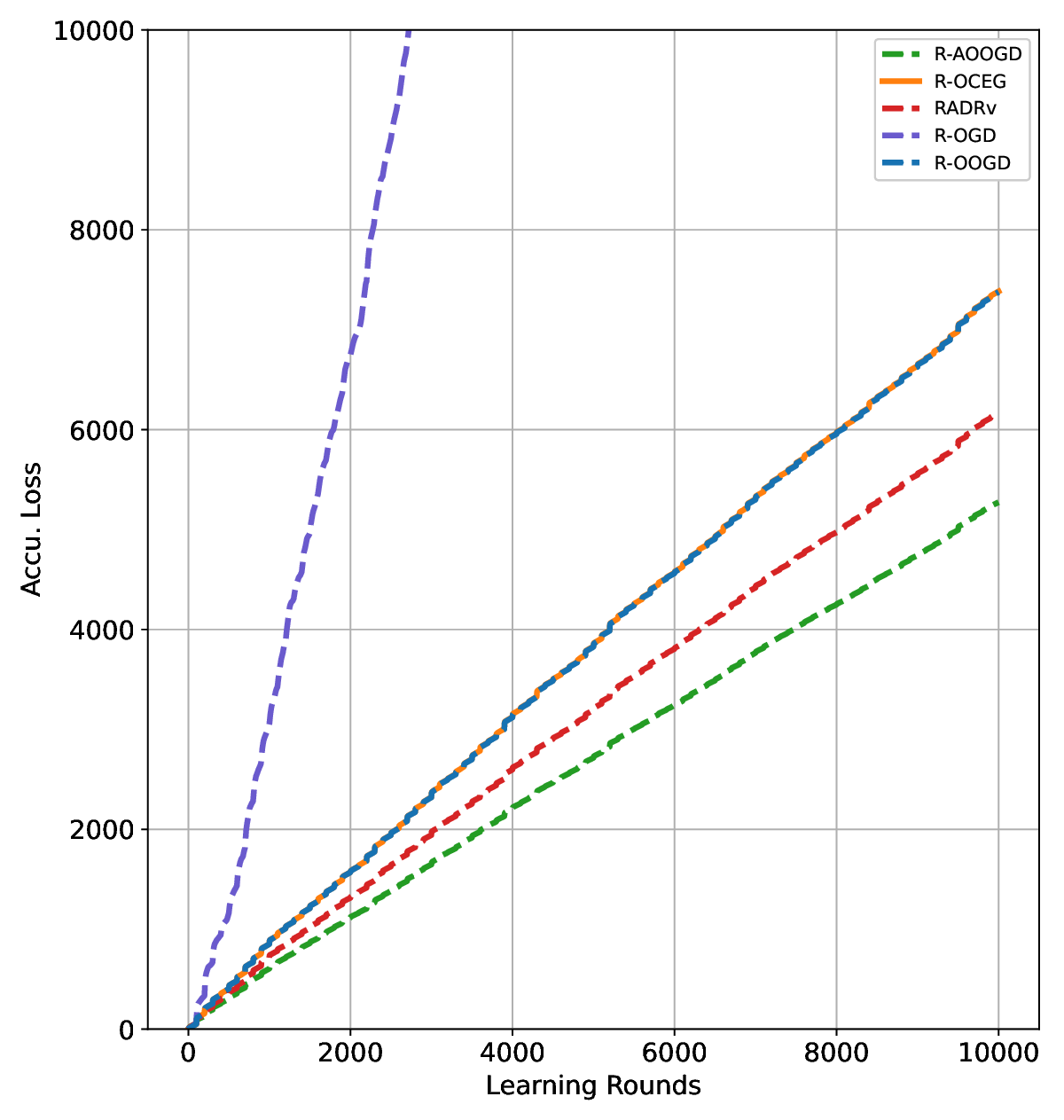}
			\end{minipage}%
        }
		\caption{Algorithm performance on hyperbolic Fr\'echet mean problems}
		\label{fig:FM}
	\end{figure} 

 \subsection{Online Geodesic Regression}\label{exp: ogr}
   Geodesic Regression is a statistical model that generalizes the Euclidean least-square regression by modeling data as
\begin{align*}
y = \exp_{\hat y_t}(\epsilon), \quad \hat y_t = \exp_{p}(x V),
\end{align*}
where $x\in \R$ is the feature, $y\in \M$ is the manifold-valued label, $V$ is a tangent vector at $P$, and $\epsilon$ is a Gaussian-distributed error. Geodesic Regression has many applications in medical imaging, object recognition, and linear system identification 
\citep{yang2016multivariate,shin2022robust,hong2014geodesic}.

We can also consider an online form of geodesic regression, where the model is trained sequentially. For each data point ${x_t, y_t}$, the online geodesic regression model minimizes the loss function
\begin{align*}
\ft(p_t,V_t) = \frac{1}{2} d^2(\hat y_t, y_t) = \frac{1}{2} d^2(\exp_{p_t}(x_t V_t), y_t),
\end{align*}
where $(p_t,V_t) \in T\M$ lies in the tangent bundle $T\M$ with the Sasaki metric \citep{muralidharan2012sasaki},
\begin{align*}
\big\langle (p^\prime_1,V^\prime_1),(p^\prime_2,V^\prime_2) \big\rangle = \langle p^\prime_1, p^\prime_2 \rangle_{\M} + \langle V^\prime_1, V^\prime_2 \rangle_{T_P\M}.
\end{align*}

\paragraph{Experimental setting} We now conduct our experiments on both synthetic and real-world dataset. The synthetic dataset is generated on a five-dimensional sphere $\mathbb S^5$. At each round, the feature $x_t$ is uniformly sampled from $[0,1]$, and the label $y_t$ is obtained as $y_t = \exp_{\hat y_t}(\epsilon_t),$ where $\hat y_t = \exp_{p_t}(x_t V_t)$ and $\epsilon_t$ is a random tangent vector with norm chosen uniformly from $[0,0.1]$. Similar to the previous experiment, we fix $p_t$ and $V_t$ for $S$ rounds and randomly select new $p_t$ and $V_t$ from the half sphere after every $S$ rounds. 

The real-world dataset used in this experiment is the \textit{corpus callosum dataset} from the Alzheimer's disease neuroimaging initiative, which was provided by \cite{cornea2017regression} and can also be accessed online \footnote{\url{http://www.bios.unc.edu/research/bias/software.html}}. The dataset includes information about 408 subjecst as well as shape of the subjects' corpus callosum obtained from mid-sagittal slices of magnetic resonance images (MRI).  The shape of an corpus callosum is described as a cloud point matrix $y_t \in \R^{50 \times 2}$. During the experiment, we aimed to analyze the relationship between the age of the subjects and the shape of the corpus callosum by utilizing geodesic regression. To achieve this, we preprocessed the shape information into a Grassmann manifold $ \mathcal{G}r(50,2)$ by computing the left-singular vectors of each $y_t$ as reported in a previous study by \cite{hong2014geodesic}. Then, we divided the $320$ data points into a training set and others to a testing set. Additionally, we duplicated training data points 5 times for efficient learning rounds.

\paragraph{Results} Figure \ref{fig: georeg} shows the accuracy and loss versus learning round for our algorithm. Additionally, we provide performance results on the test set of the corpus callosum shapes in Figure \ref{fig: test set}, which confirms the effectiveness of our R-OOGD algorithm and R-AOOGD algorithm in solving real-world problems. Our algorithm performs similarly or better than other state-of-the-art methods on both synthetic and real-world datasets, while requiring fewer gradient information. This finding aligns with our theoretical results.

It is noteworthy that we did not impose any special requirements on the boundedness of the training set in the real-world dataset. Therefore, our algorithm iterates over the entire Grassmann manifold with a diameter of $\frac{\sqrt{2}}{2}\pi>\frac{\pi}{2\sqrt{K}} = \frac{\pi}{2\sqrt{2}}$. This demonstrates the applicability of our algorithm when the diameter $D\ge \frac{\pi}{2\sqrt{K}}$. 

    \begin{figure}[htbp]\label{fig: georeg}
		\centering 
		\subfigure[Synthetic data]{
			\begin{minipage}[t]{0.45\linewidth}
				\centering
				\includegraphics[width=\linewidth]{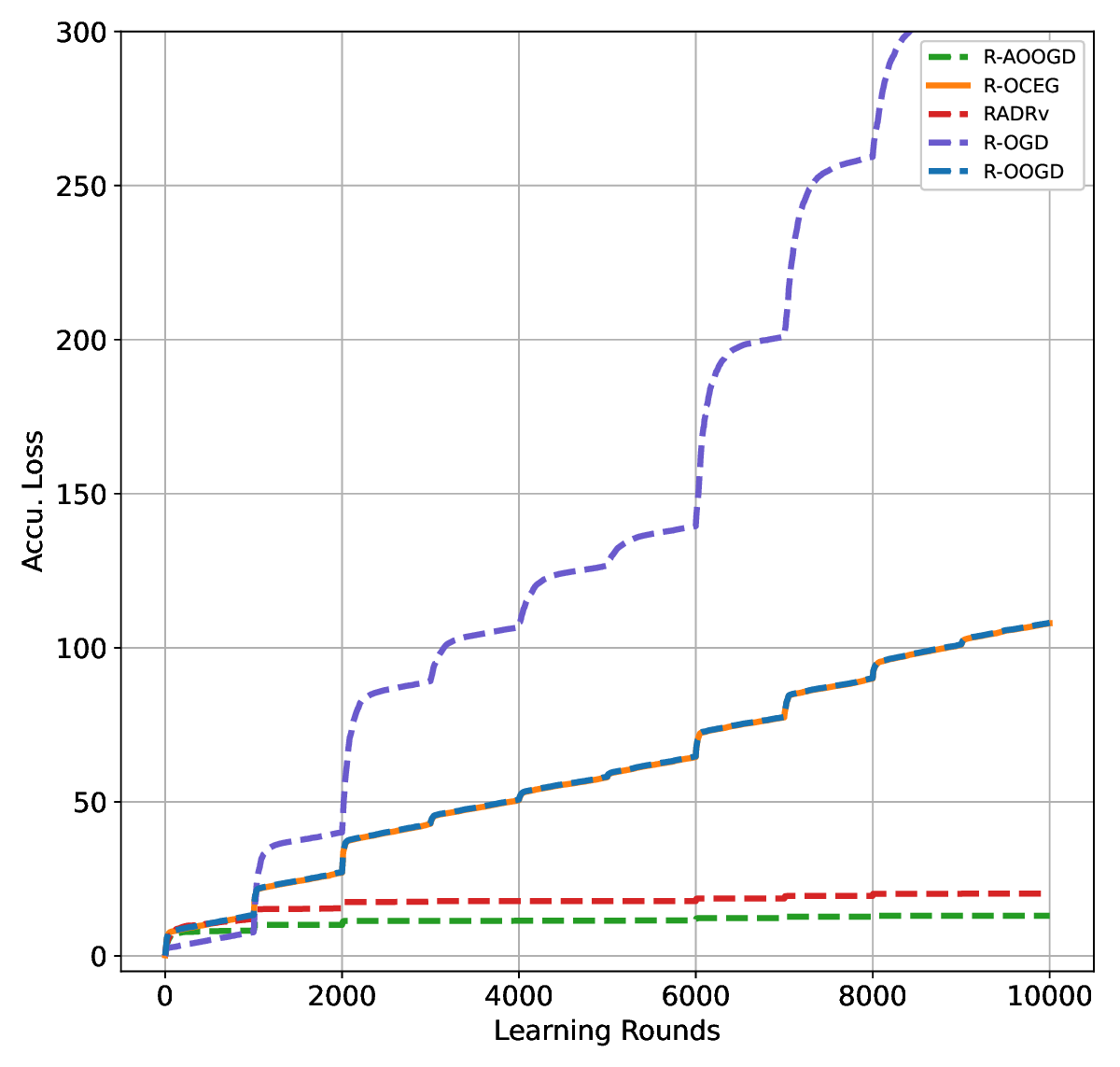}
			\end{minipage}%
        }
		\subfigure[Real-world data]{
			\begin{minipage}[t]{0.45\linewidth}
				\centering
				\includegraphics[width=\linewidth]{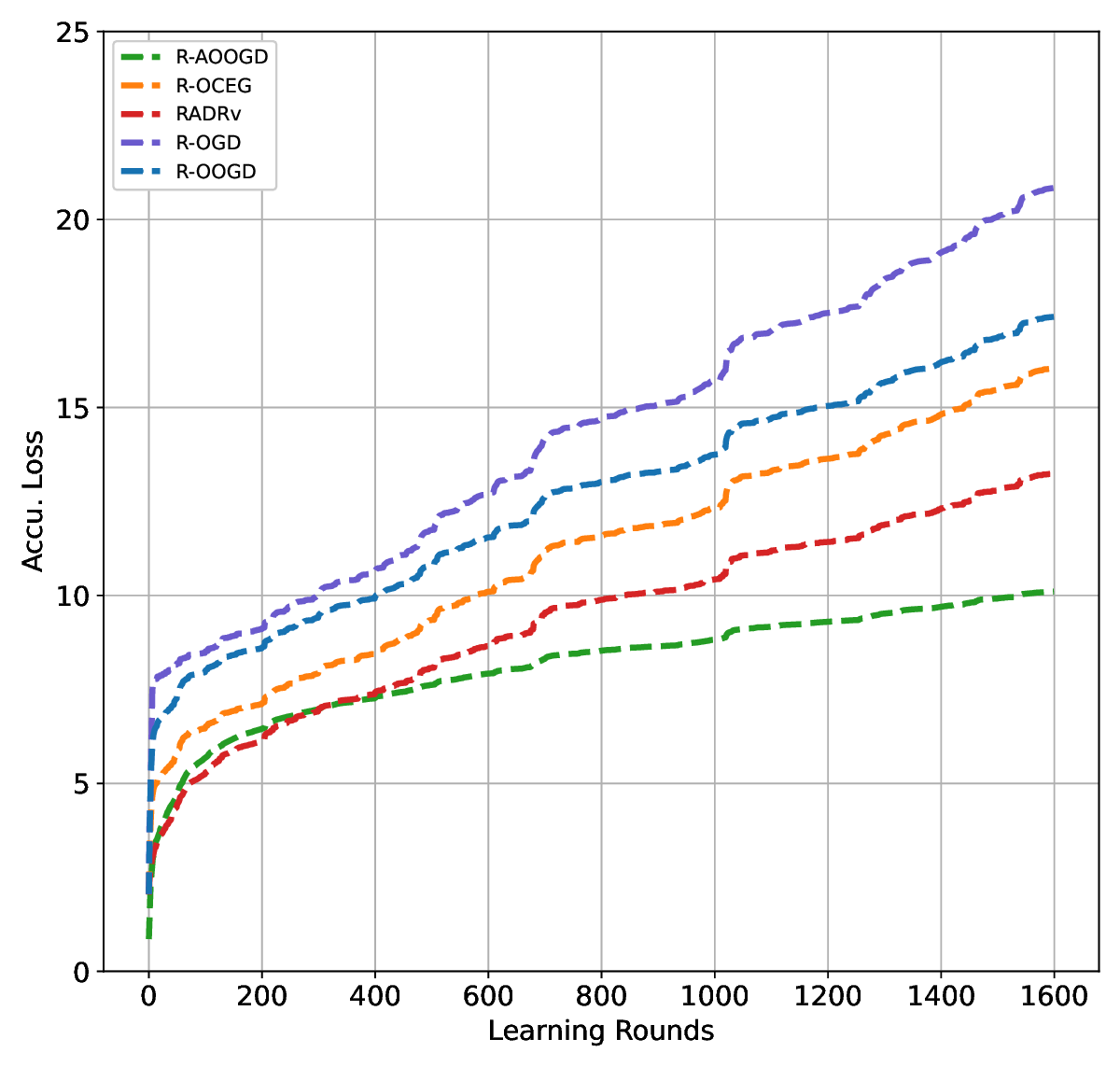}
			\end{minipage}%
        }
		\caption{Algorithm performance on geodesic regression}
	\end{figure} 

    \begin{figure}[htbp]\label{fig: test set}
		\centering 
		\subfigure[age=64]{
			\begin{minipage}[t]{0.2\linewidth}
				\centering
				\includegraphics[width=\linewidth]{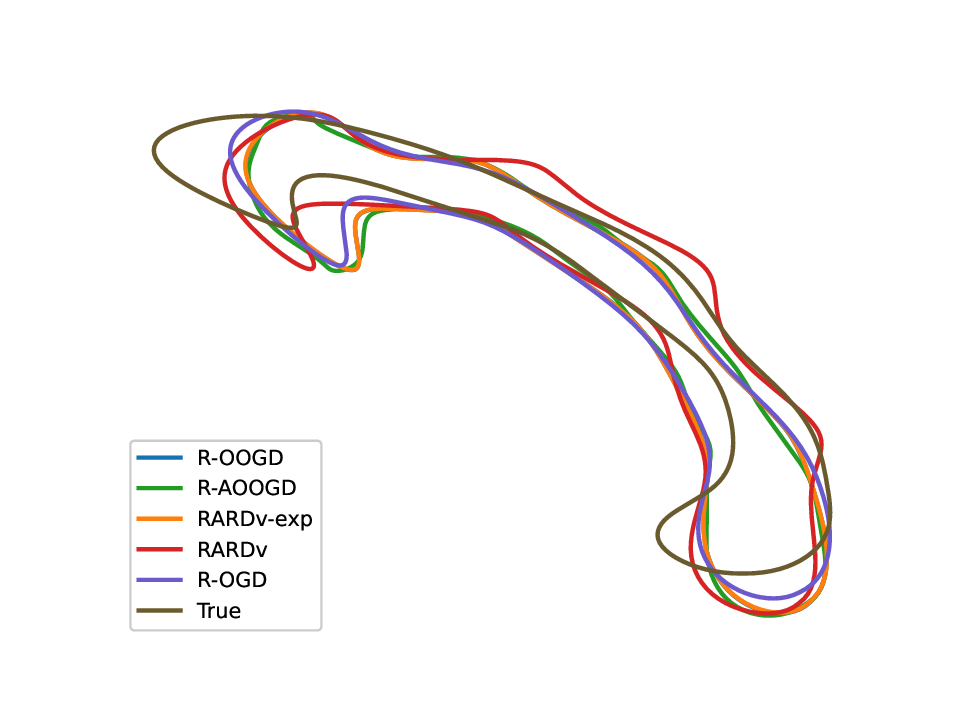}
			\end{minipage}%
        }\subfigure[age=66]{
			\begin{minipage}[t]{0.2\linewidth}
				\centering
				\includegraphics[width=\linewidth]{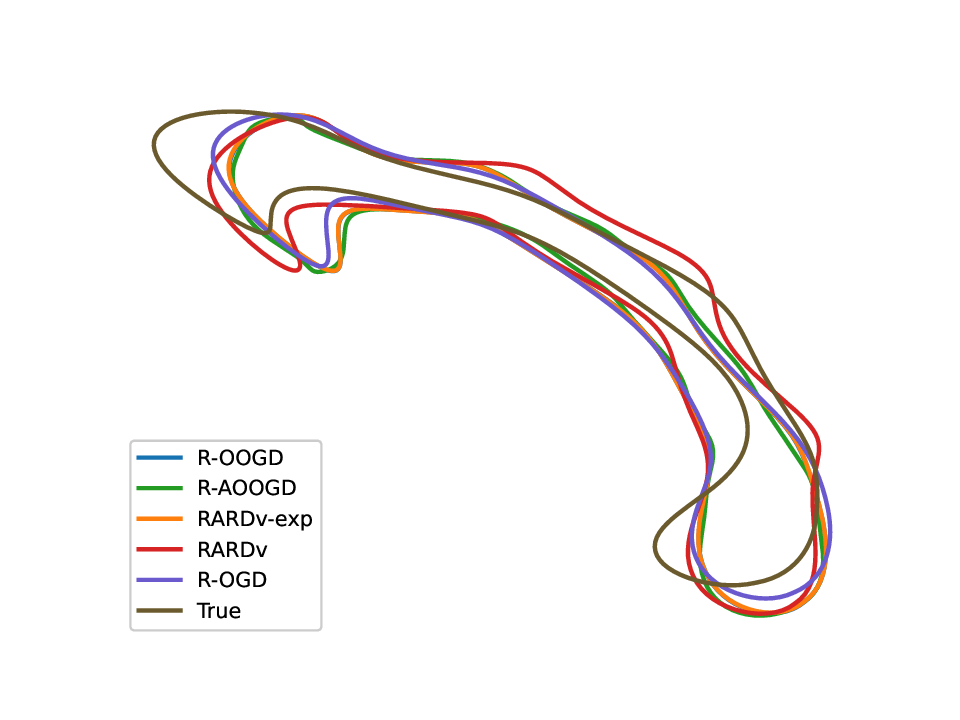}
			\end{minipage}%
        }\subfigure[age=69]{
			\begin{minipage}[t]{0.2\linewidth}
				\centering
				\includegraphics[width=\linewidth]{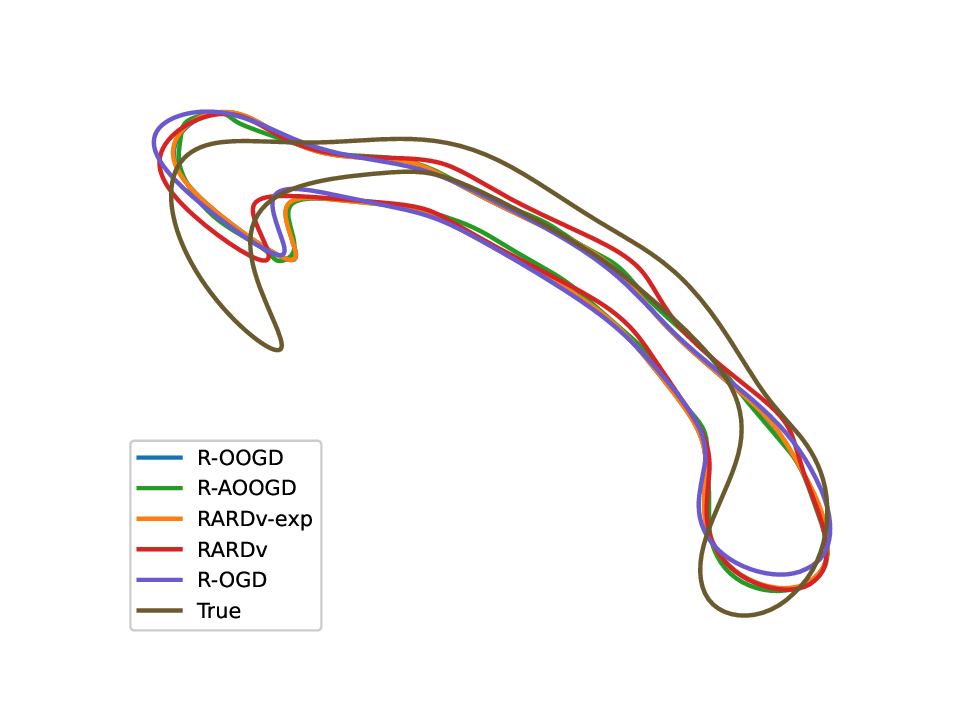}
			\end{minipage}%
        }\subfigure[age=70]{
			\begin{minipage}[t]{0.2\linewidth}
				\centering
				\includegraphics[width=\linewidth]{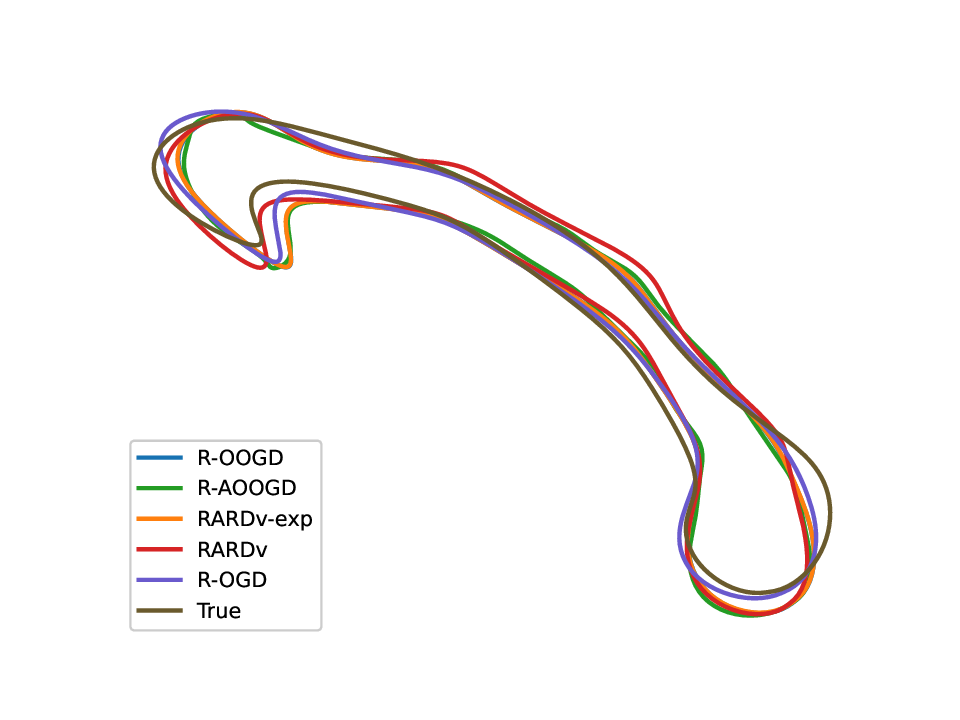}
			\end{minipage}}\\
            \subfigure[age=71]{
			\begin{minipage}[t]{0.2\linewidth}
				\centering
				\includegraphics[width=\linewidth]{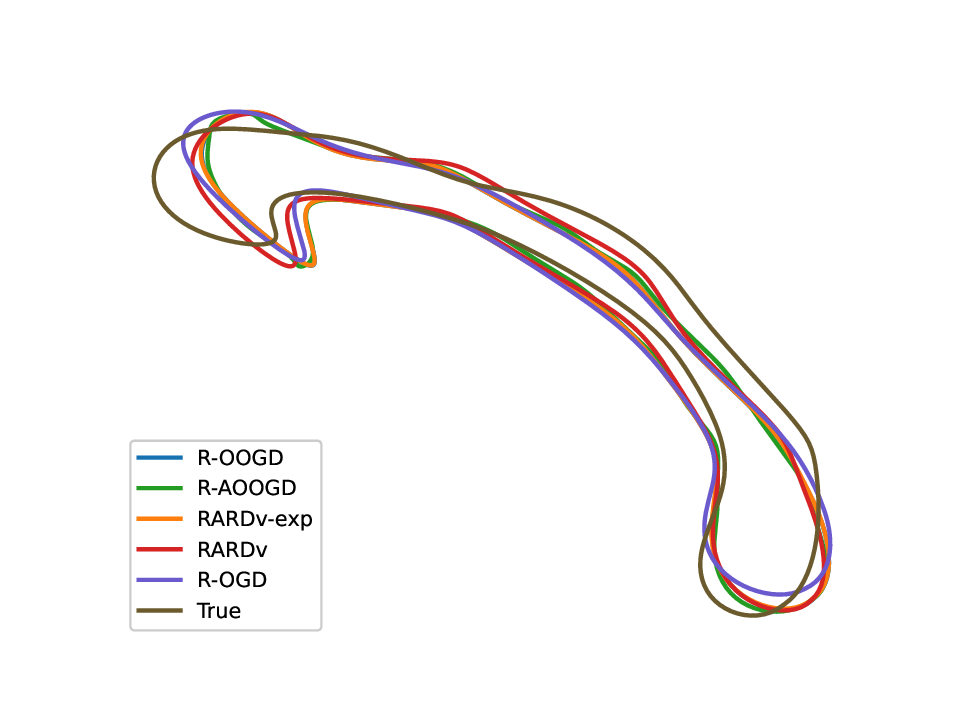}
			\end{minipage}
            }\subfigure[age=73]{
			\begin{minipage}[t]{0.2\linewidth}
				\centering
				\includegraphics[width=\linewidth]{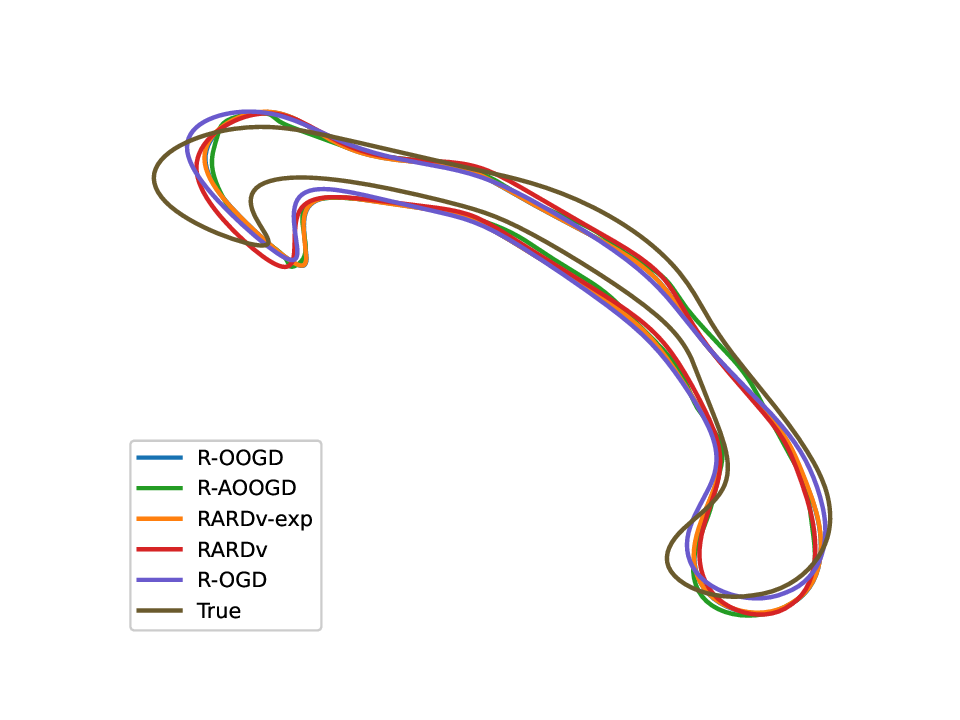}
			\end{minipage}
            }\subfigure[age=75]{
			\begin{minipage}[t]{0.2\linewidth}
				\centering
				\includegraphics[width=\linewidth]{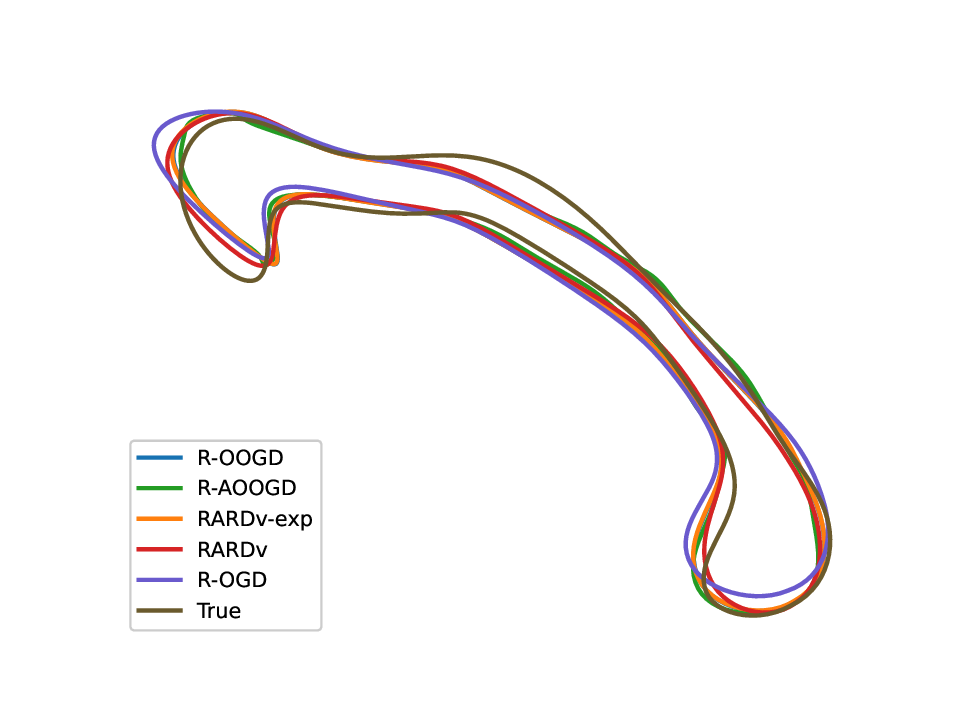}
			\end{minipage}
        }\subfigure[age=77]{
			\begin{minipage}[t]{0.2\linewidth}
				\centering
				\includegraphics[width=\linewidth]{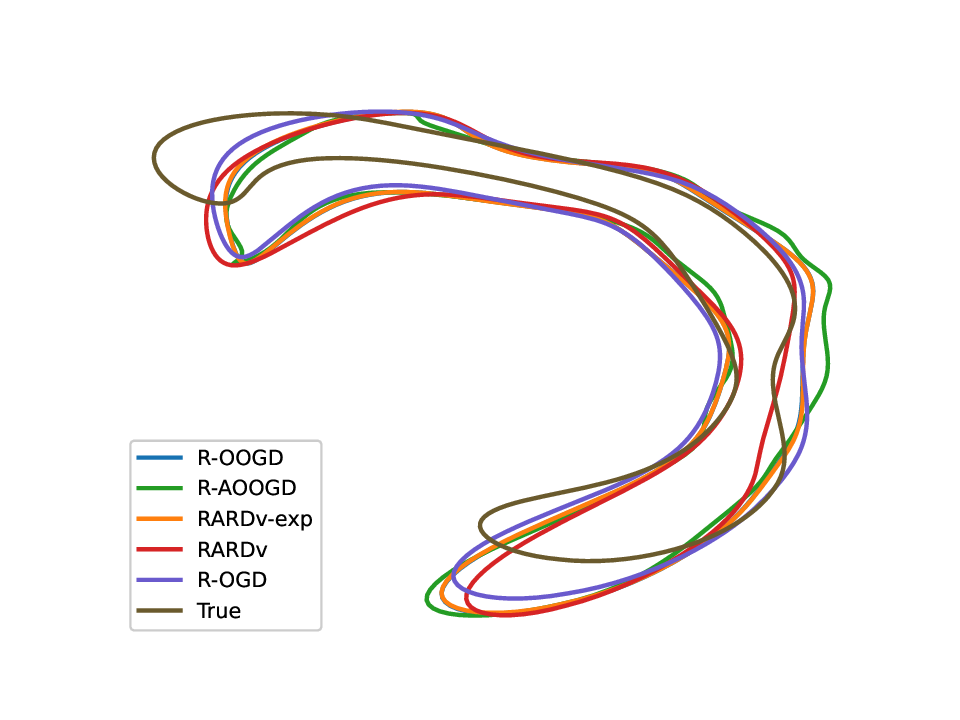}
			\end{minipage}%
        }\\
        \subfigure[age=80]{
			\begin{minipage}[t]{0.2\linewidth}
				\centering
				\includegraphics[width=\linewidth]{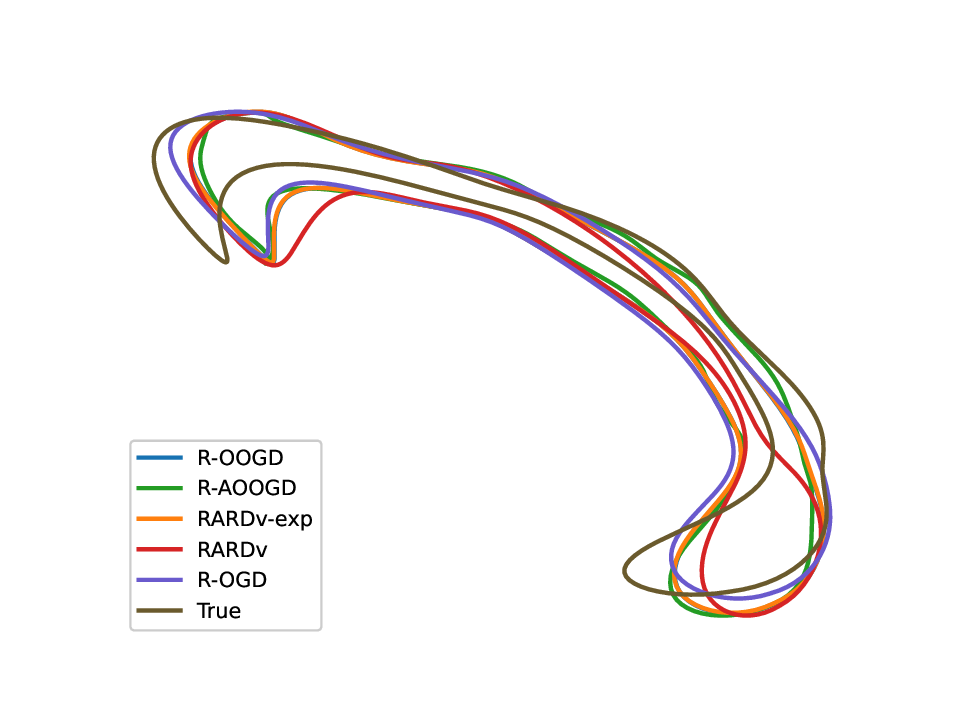}
			\end{minipage}
            }\subfigure[age=83]{
			\begin{minipage}[t]{0.2\linewidth}
				\centering
				\includegraphics[width=\linewidth]{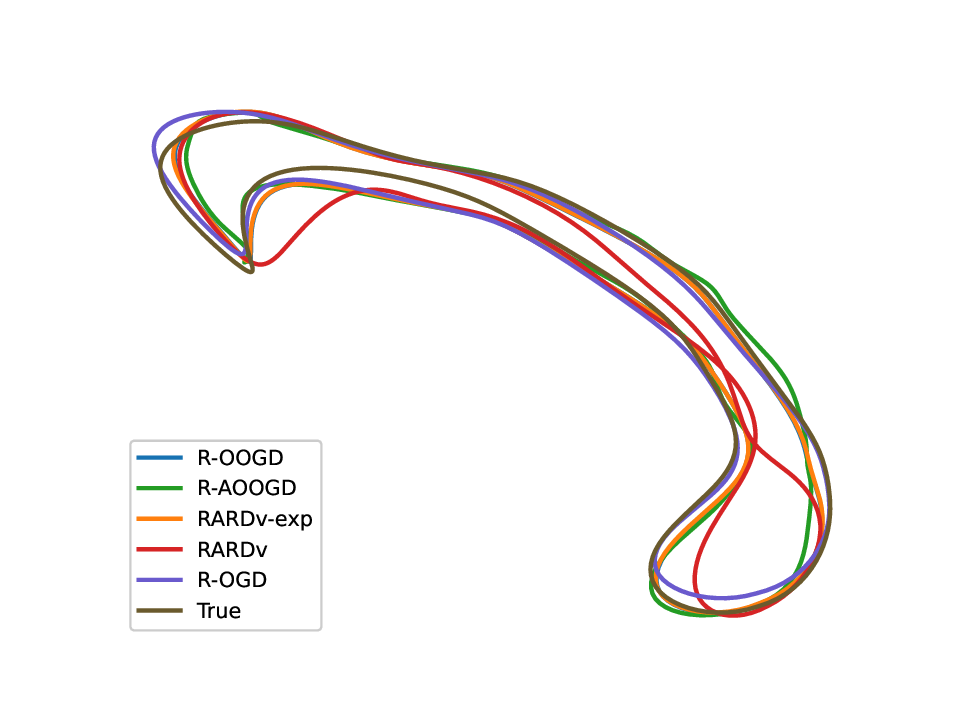}
			\end{minipage}
            }\subfigure[age=88]{
			\begin{minipage}[t]{0.2\linewidth}
				\centering
				\includegraphics[width=\linewidth]{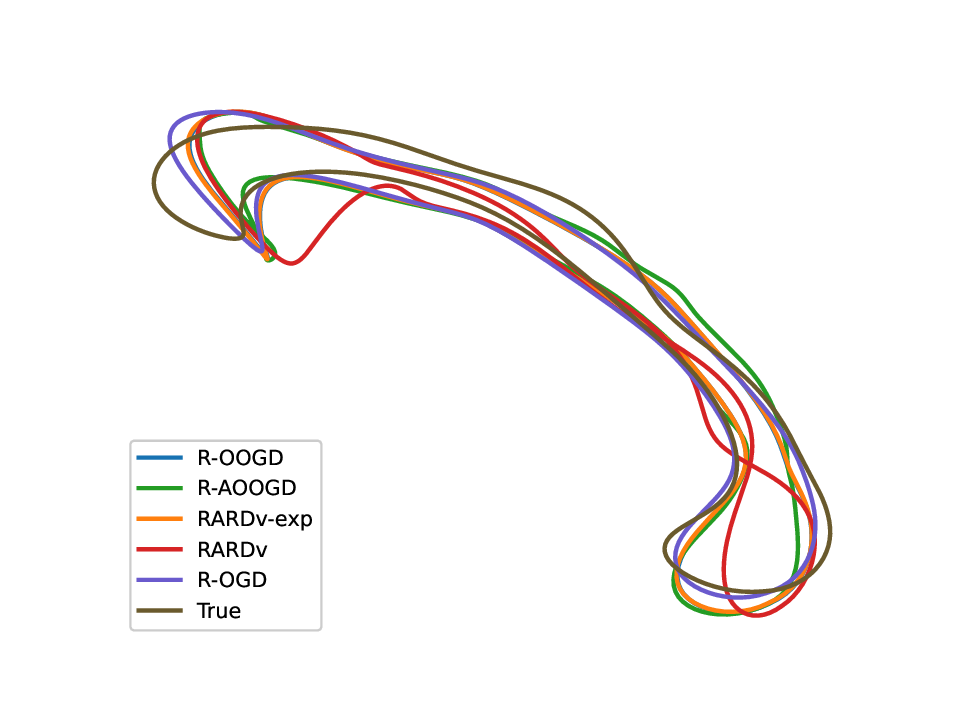}
			\end{minipage}
        }\subfigure[age=90]{
			\begin{minipage}[t]{0.2\linewidth}
				\centering
				\includegraphics[width=\linewidth]{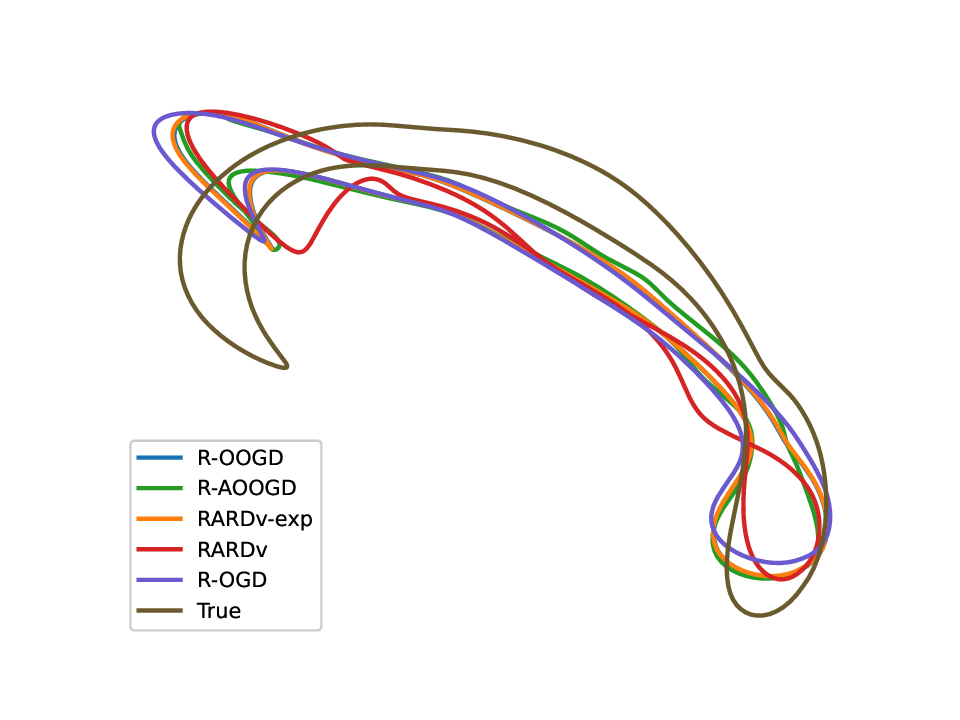}
			\end{minipage}%
        }
		\caption{Algorithm performance on testing set of corpus callosum shapes}
	\end{figure} 
\subsection{Quadratic Geodesic Lodget Game}
We now validate our theoretical finding in Riemannian zero-sum games. Consider the following toy-example RZS game
\begin{align}\label{eq:quad}
    \min_{X\in\mathcal{S}_d^{++}} \max_{Y\in\mathcal{S}_d^{++}}  c_1 (\log\det(X))^2 + c_2 \log\det(X) \log\det(Y) - c_1 (\log\det(Y))^2,
\end{align}
where $X$ and $Y$ take values on the symmetric positive definite (SPD) matrix manifold 
\begin{align*}
    S^{++}_d := \{ X\in \R^{d \times d}; X^T = X, X \succ 0 \}
\end{align*} with affine-invariant metric \begin{align*}
    \langle U,V \rangle_x = tr(X^{-1}UX^{-1}V).
\end{align*} Since the logdet function is geodesic linear on $S^{++}_d$ \citep{han2022riemannian}, the quadratic game \eqref{eq:quad} is g-convex-concave with the g-strong convexity coefficient $c_1$. The NEs of the game \eqref{eq:quad} are $(X^*,Y^*)$ where $\det(X^*) = \det(Y^*) = 1$.

\paragraph{Experimental setting} We test the R-OGDA in the case when $d=30$, $c_2 =1$, and $c_1 \in \{ 0,0.1,1 \}$. We check the R-OGDA with the step size $\eta=0.5$ for $c_1 \in \{0,0.1\}$ and $\eta = 0.2$ for $c_1=1$. We also compare our algorithm with the second-order Riemannian Hamitonian method (RHM) \citep{han2022riemannian}, the Riemannian corrected extragradient method (RCEG) \citep{zhang2022minimax}, and the Riemannian gradient descent ascent algorithm (R-GDA) \citep{NEURIPS2022_2ad9a1a6} with the best-tuned step size. 
\begin{figure}[tbp]
		\centering 
		\subfigure[$c_1=0,c_2=1$]{
			\begin{minipage}[t]{0.3\linewidth}
				\centering
				\includegraphics[width=\linewidth]{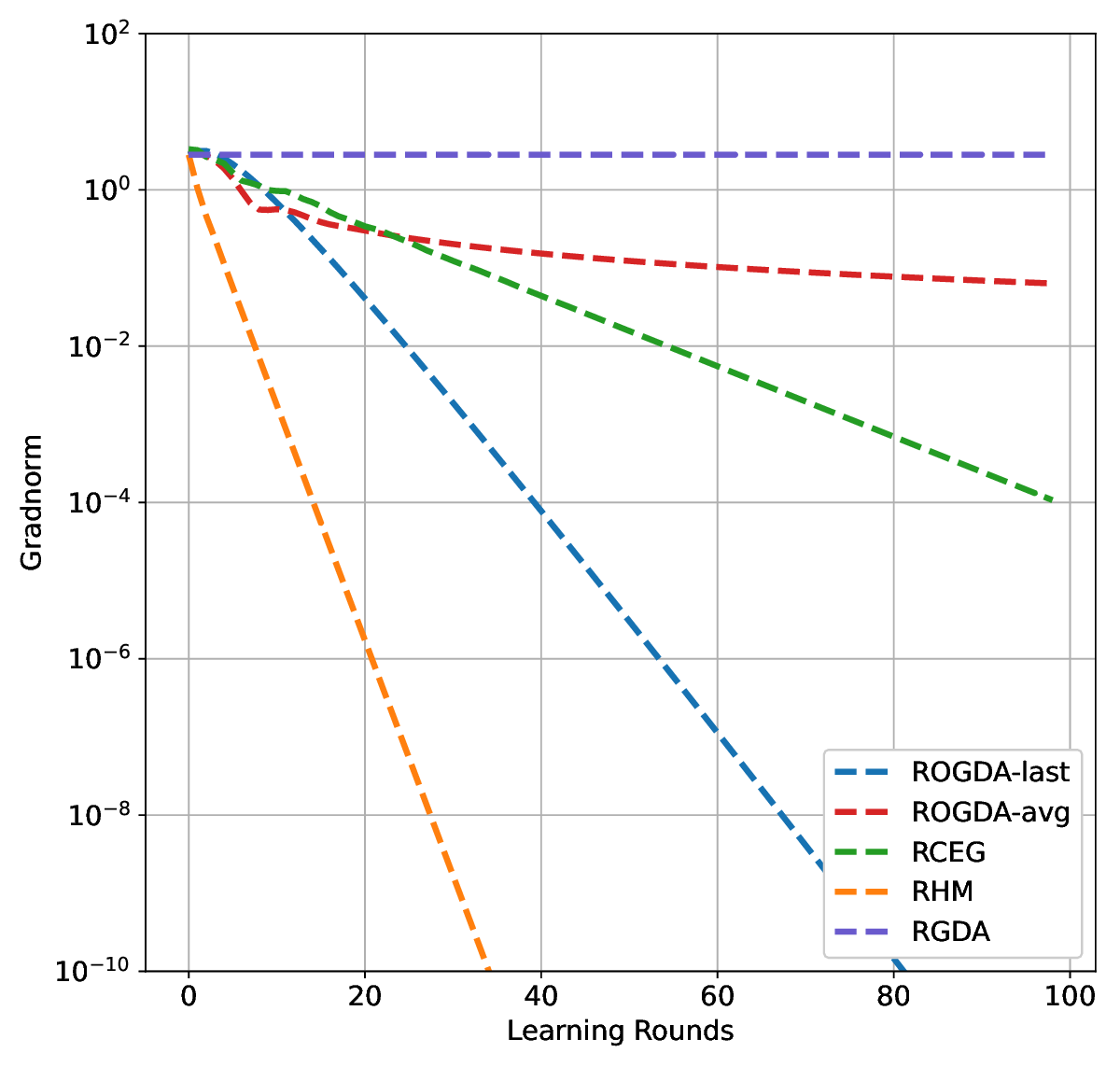}
			\end{minipage}%
		}%
		\subfigure[$c_1=0.1,c_2=1$]{
			\begin{minipage}[t]{0.3\linewidth}
				\centering
				\includegraphics[width=\linewidth]{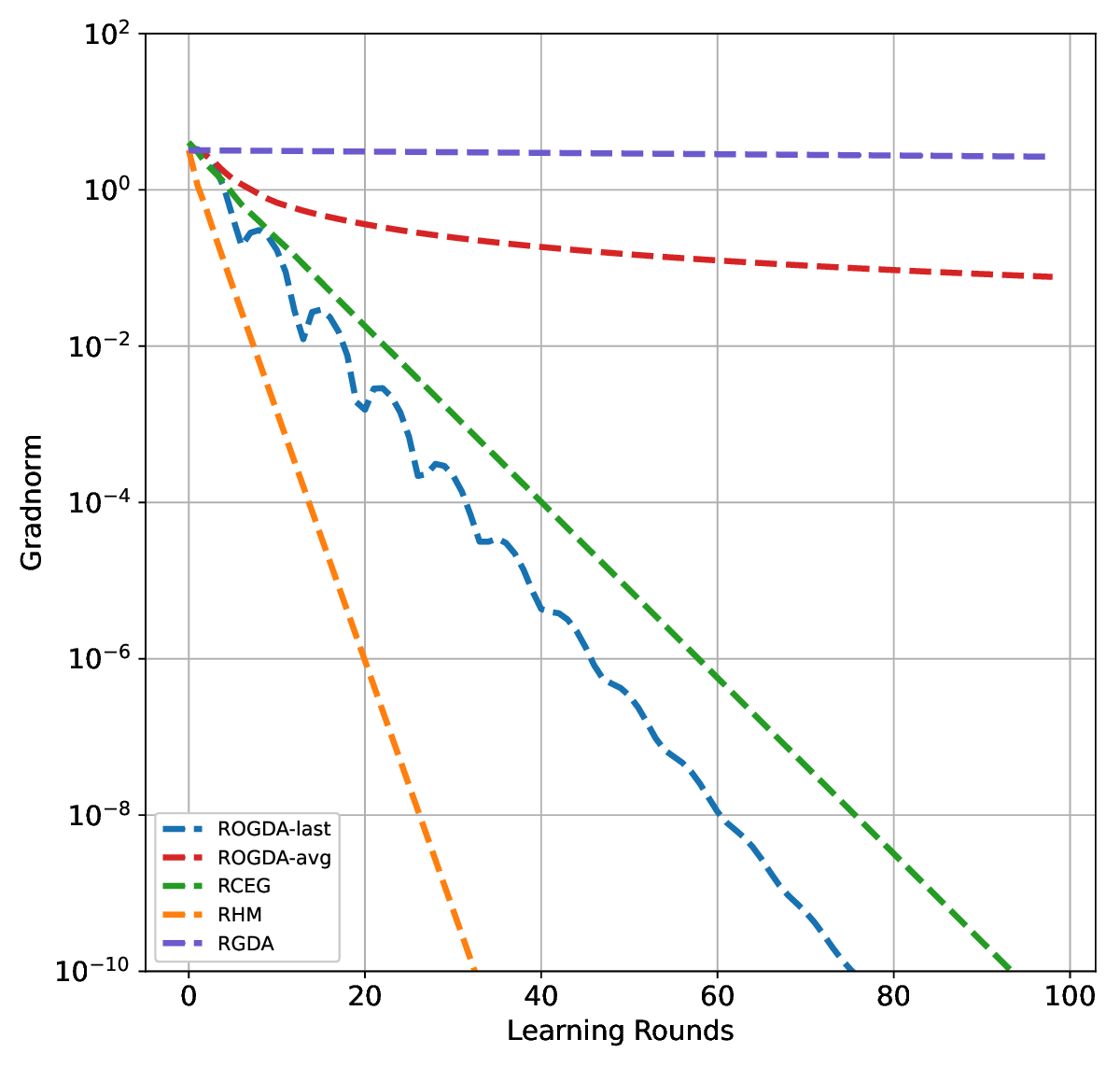}
			\end{minipage}%
			
		}
		\subfigure[$c_1=1,c_2=1$]{
			\begin{minipage}[t]{0.3\linewidth}
				\centering
				\includegraphics[width=\linewidth]{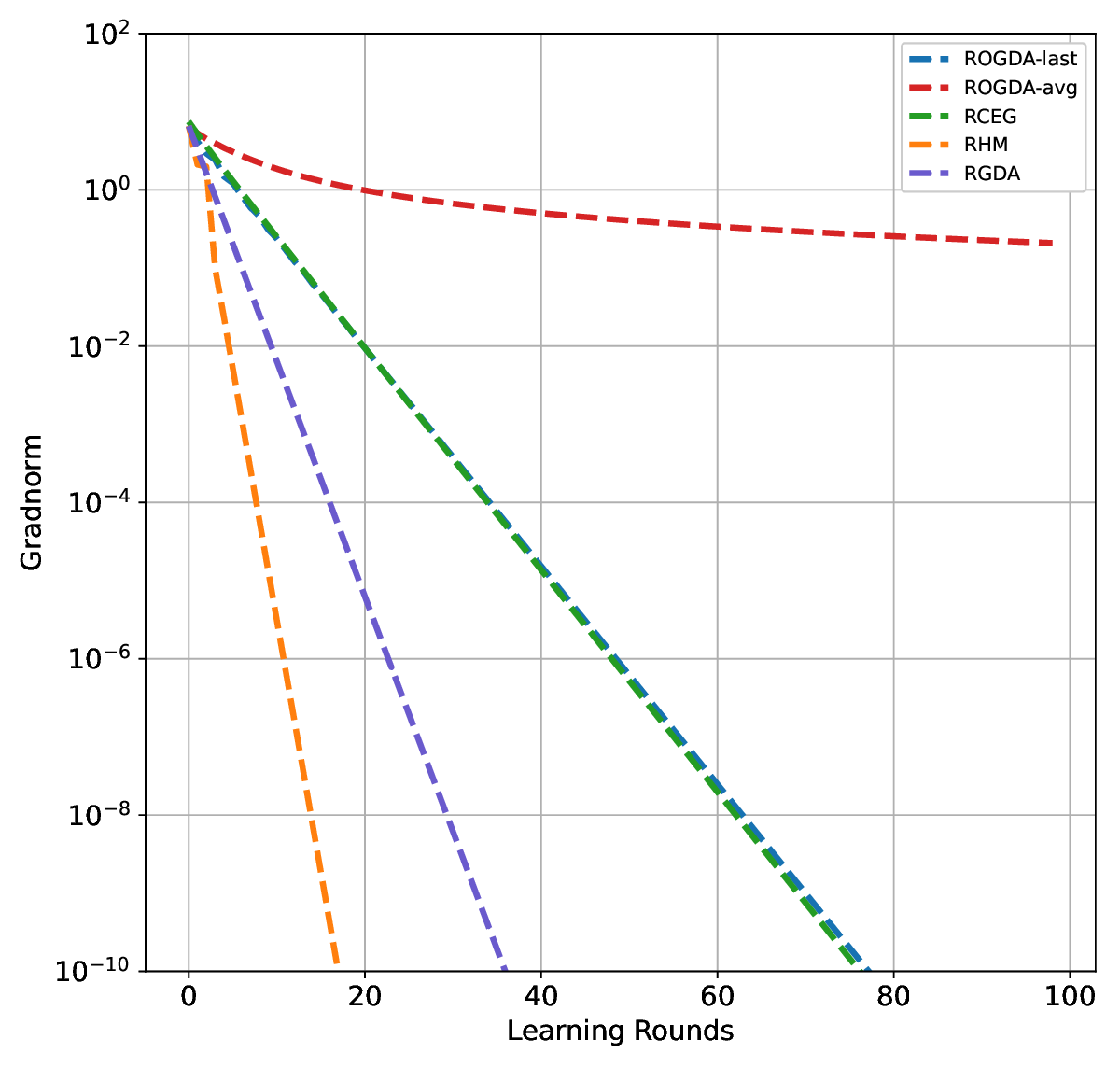}
			\end{minipage}
		}
		\caption{Quadratic geodesic lodget games}
		\label{fig:quad}
	\end{figure} 
\paragraph{Results} We plot the norm of gradient $\|\nabla f (x_t,y_t)\|$ versus learning round $t$ in Figure \ref{fig:quad}. From the results, all the algorithm converge to the NE. Among them, the second-order algorithm RHM performs the best. For first-order algorithms, our R-OGDA algorithm shows good performance in both g-convex-concave and g-strongly convex-strongly concave scenarios. In comparison, R-GDA performs well in g-strongly convex-strongly concave scenarios but does not converge in linear cases (i.e., $c=0$). Similarly, RCEG shows slower convergence in linear scenarios and gradually improves with increasing g-strong convexity of the function, showing comparable convergence to our R-OGDA  when $c=1$. 

\subsection{Robust Geometry-Aware PCA}
Robust geometry-aware principal component analysis (PCA) \citep{horev2016geometry} is a dimensional reduction tool for SPD matrices. Given a set of SPD matrices $\{ A_i \}_{i=1}^n$, the geometry-aware PCA aims to find a geometry mean $A$ with the maximal data variance, which can be formulated as finding the NE in the following RZS game
\begin{align}\label{eq:pca}
    \min_{A\in \mathcal S^{++}_d}\max_{X\in \mathbb S^d} X^T A X + \frac{\alpha}{n} \sum_i^n \| \log(A_i^{-\frac{1}{2}}AA_i^{-\frac{1}{2}}) \|_F,
\end{align}
where $\mathbb S^d$ is the $d$-dimensional unit sphere with the canonical metric $\langle U,V \rangle = U^T V$. The game \eqref{eq:pca} is g-strongly convex to $A$, but not g-concave to $X$. Hence, the game \eqref{eq:pca} is more challenging for our R-OGDA algorithm.

\paragraph{Experimental setting} We run the experiment with a synthetic dataset $\{ A_i \}_{i=1}^n$ and a real-world BCI dataset. The synthetic dataset is generated under conditions similar to those described in prior studies by \cite{zhang2022minimax} and \cite{han2022riemannian}, where the eigenvalues of $A_i$ are in $[0.2,4.5]$. The real-world dataset the \textit{BCI competition dataset IV}\footnote{\url{https://www.bbci.de/competition/iv/download/}} \citep{schlogl2005characterization}. The BCI competition dataset IV collected EEG signals from 59 channels electrodes of 5 subjects who executed left-hand, right-hand, foot and tongue movements. To preprocess the dataset, we follow the procedure outlined in \cite{horev2016geometry}, which involves selecting $200$ trials in the time interval from 0.5s to 2.5s, applying a band-pass filter to remove frequencies outside the 8--15Hz range and extracting covariant matrix into a $6\times6$ matrix. For numerical stability, we divide the covariant matrix by $200$. For the synthetic dataset, we set the number of samples $n=40$, the dimension $d=50$, the regularization parameter $\alpha = 1$, and the step size $\eta = 0.07$. For the competition BCI dataset IV, we set $n=200$, $d=6$, $\alpha = 1$ and $\eta = 0.05$.

\paragraph{Results} As shown in Figure \ref{fig:pca}, on the synthetic datase, the R-OGDA has the fastest convergence rate. For the real world dataset, R-GDA performs the best and the R-OGDA outperforms the RHM and RCEG. The performance in the g-convex-nonconcave RZS games illustrates the potential value of the R-OGDA.
\begin{figure}[tbp]
		\centering 
		\subfigure[synthetic dataset]{
			\begin{minipage}[t]{0.45\linewidth}
				\centering
				\includegraphics[width=\linewidth]{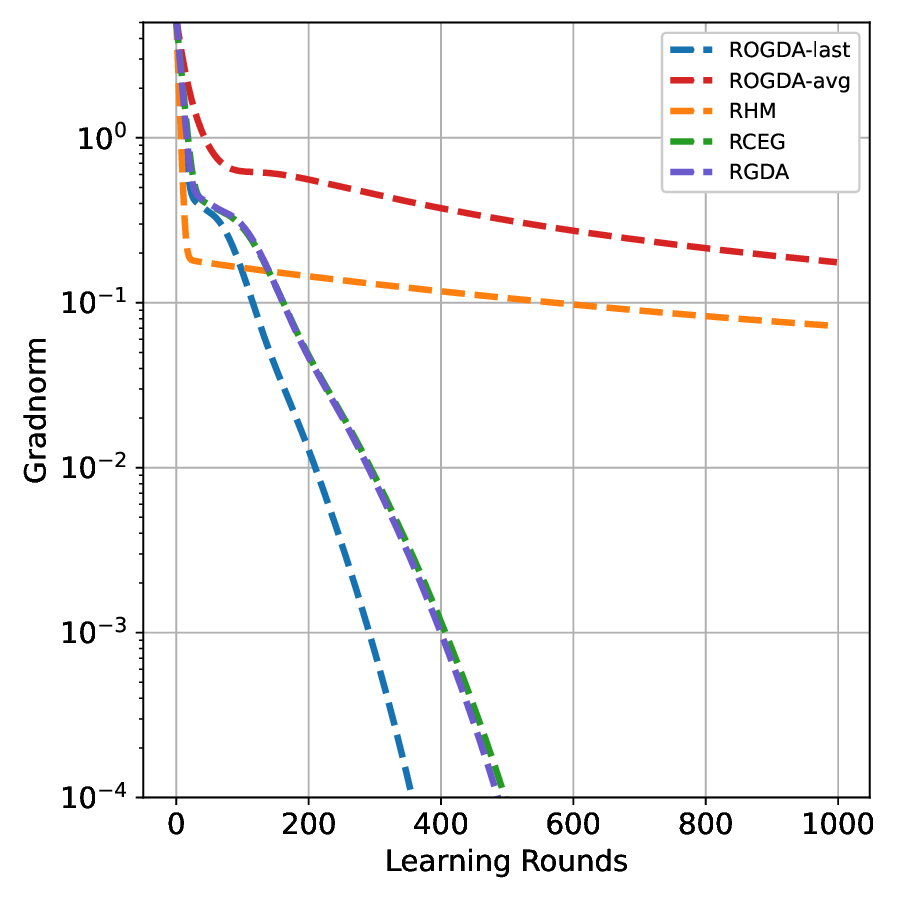}
			\end{minipage}%
		}%
		\subfigure[real-world dataset]{
			\begin{minipage}[t]{0.45\linewidth}
				\centering
				\includegraphics[width=\linewidth]{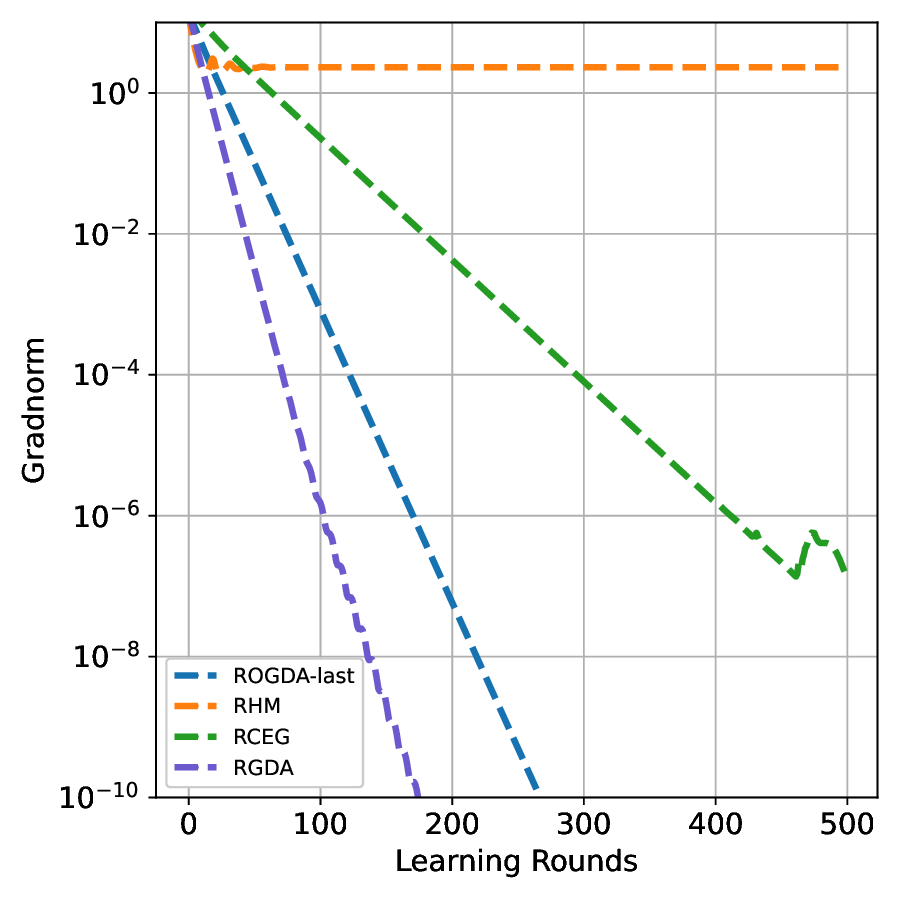}
			\end{minipage}%
			
		}
		\caption{Robust geometric-aware PCA}
		\label{fig:pca}
	\end{figure} 

\section{Conclusion}

In this paper, we have investigated the dynamic regret of Riemannian online optimization. We proposed R-OOGD and established a dynamic regret bound of $\bo(P_T\sqrt{1+V_T})$. Additionally, we introduced R-AOOGD, a meta-expertalgorithm that averages R-OOGD algorithms with different step sizes, which further improved the dynamic regret bound to $\bo(\sqrt{(1+P_T+V_T)(1+P_T)})$. We also applied the ROGD algorithm to Riemannian zero-sum games and achieved convergence rates of $\bo(\frac{1}{T})$ for average-iterate, $\bo(\frac{1}{\sqrt{T}})$ for best-iterate, and $\bo(e^{-\rho t})$ for last-iterate, for smooth g-convex-concave and g-strongly convex-strongly concave games. Our results demonstrate the impact of Riemannian geometry on algorithm performance, and all our regret bounds and convergence rates match the corresponding Euclidean results.

One possible future direction of our work is to consider Riemannian dynamic regret in the bandit feedback setting, where the learner only receives the function value $\ft(x_t)$ instead of the gradient $\nabla \ft(x_t)$. We plan to design Riemannian optimistic bandit algorithms by incorporating Riemannian optimization techniques into existing bandit optimization methods. This will make the resulting algorithms more effective in settings where obtaining gradients is difficult or even impossible, such as reinforcement learning on Riemannian manifolds.

\newpage

\appendix
\input{appendix3}

\clearpage

\vskip 0.2in
\bibliographystyle{plainnat}

\end{document}

%% file: appendix3.tex
\section{Definitions and Technical Lemmas}\label{app:lemma}
    In the appendix, we introduce some prerequisite about Riemannian geometry for further analysis. In the rest of the appendix, we denote $\X(\M)$ as the set of vector fields on the Riemannian manifold $\M$ and $C^\infty(\M)$ as the class of infinitely differentiable functions on $\M$. A vector field $X$ is equivalent to an operator on $C^\infty(\M)$ via the directional derivative $X(\f) := \lim_{t\to 0 } \frac{1}{t}( \f(\gamma(t)) - \f(t) )$, where $\f\in C^\infty(\M)$ and $\gamma$ is a curve such that $\gamma(0)=p$ and $\dot\gamma(0) = X(p)$.
    
    First, we recall some definitions on Riemannian manifolds.
    \begin{definition}[\citealp{lee2018introduction}]
        The gradient of a function $\f$ at the point $x$ is defined as the tangent vector $\nabla \f(x) \in T_x\M$ such that $\langle X(x) , \nabla \f(x) \rangle = X(\f)(x),\forall X\in \X(\M)$.
    \end{definition}
    
    \begin{definition}[\citealp{lee2018introduction}]
        The hessian of a function $f$ at the point $x$ is defined as the bilinear operator $\nabla^2 \f(x): T_x\M \times T_x\M \to \R$ such that 
        \begin{align}
            \nabla^2\f(X(x),Y(x))(x) = \langle  \nabla_X \nabla \f(x), Y(x)\rangle,\forall X,Y\in \X(\M).
        \end{align}
    \end{definition}
    
    \begin{definition}[\citealp{lee2018introduction}]
         A vector field $J$ along a geodesic $\gamma:[0,1]\to \M$ is a Jacobi field if it satisfies:
         \begin{align*}
             \langle \nabla_{\dot \gamma} \nabla_{\dot \gamma} J, W \rangle + R(\dot\gamma,J,\dot\gamma,W) = 0,\quad \forall W \in \X(\M).
         \end{align*}
    \end{definition}

    Then we recall some properties of the covariant derivative $\nabla_X Y$, the curvature tensor $R(X,Y,W,Z)$ and the Jacobi field.   
    \begin{lemma}[\citealp{carmo1992riemannian}]
        The covariant derivative $\nabla: \X(\M) \times \X(\M) \to \X(\M)$ satisfies the following properties:
        \begin{itemize}
            \item [(i)] $Z\langle X, Y \rangle = \langle \nabla_Z X , Y \rangle + \langle  X , \nabla_Z Y\rangle, \quad \forall X,Y,Z \in \X(\M)$;
            \item [(ii)] $ \nabla_X Y - \nabla_Y X = [X,Y] $, where $[X,Y] = XY-YX$ is also a vector field on $\M$.
        \end{itemize}
    \end{lemma}
    
    \begin{lemma}[\citealp{carmo1992riemannian}]
        If $\gamma:[0,1] \to \M$ is a curve on $M$ and $z\in T_{\gamma(0)}\M$. If we extend $z$ to a vector field $Z(t) \in T_{\gamma(t)}\M$ on $\gamma$ along the parallel transport $\Gamma_{\gamma(t)}$,  then we have $\nabla_{\dot \gamma(t)} Z(t) = 0$, $\forall t\in [0,1]$.
    \end{lemma}
    \begin{lemma}[\citealp{andrews2010ricci}]\label{lemma:sec}
        Denote the curvature tensor \begin{align*}
            \langle \nabla_X\nabla_Y Z- \nabla_Y\nabla_X Z -\nabla_{[X,Y]} Z, W\rangle = R(X,Y,W,Z),
        \end{align*} then the following statement hold.
        \begin{itemize}
            \item [(i)] $R(X,Y,W,Z)$ is multilinear over $C^\infty(\M)$ i.e.,
            \begin{align*}
                R(f_1X,f_2Y,f_3W,f_4Z) = f_1f_2f_3f_4R(X,Y,W,Z), \quad \forall f_1,f_2,f_3,f_4 \in C^\infty(\M).
            \end{align*}
            \item [(ii)] The curvature tensor is determined by the its values on $2-$dimensional surfaces \\$\mathcal K(X,Y) = R(X,Y,X,Y)$, i.e.,
            \begin{align*}
                6R(X,Y,W,Z) &= \mathcal K(X+W,Y+Z) - \mathcal K(X,Y+Z) + \mathcal K(W,X+Z) \\
                            & \quad + \mathcal K(Y+W,X+Z) - \mathcal K(Y,X+Z) + \mathcal K(W,X+Z)   \\
                            & \quad - \mathcal K(X+W,Y) + \mathcal K(X,Y) + \mathcal K(W,Y)\\
                            & \quad - \mathcal K(X+W,Z) + \mathcal K(X,Z) + \mathcal K(W,Z)\\
                            & \quad + \mathcal K(Y+W,X) + \mathcal K(Y,X) - \mathcal K(W,X)\\
                            & \quad + \mathcal K(Y+W,Z) - \mathcal K(Y,Z) - \mathcal K(W,Z).
            \end{align*}
        \end{itemize}
    \end{lemma}
    
    \begin{lemma}[Jacobi comparison theorem,\citealp{lee2018introduction}]\label{lemma:jacobi}
        Suppose $\M$ is a Riemannian manifold and $\gamma:[0,b]\to\M$ is a geodesic with $\|\dot\gamma(0)\| = 1$. $J$ is a Jacobi field along $\gamma$.
        \begin{itemize}
            \item [(i)] If all sectional curvatures of $\M$ are upper bounded by a constant $K$, then denote
            \begin{align*}
                \mathbf S(K,t) = 
                \begin{cases}
                        \frac{\sin(\sqrt{K}t)}{\sqrt{K}} & \quad K>0;\\
                         t                        & \quad K\le 0.
                \end{cases}
            \end{align*}
        Then,
        \begin{align*}
            \|J(\gamma(t))\| \ge \mathbf S(K,t) \| \nabla_{\dot\gamma} J(\gamma(0)) \|
        \end{align*}
        for all $t\in[0,b_1]$, where $b_1=b$ if $K\le0$ and $b_1 = \min(\frac{\pi}{\sqrt{K}},b)$ if $K\ge0$.
            \item [(ii)] If all sectional curvatures of $\M$ are lower bounded by a constant $\kappa$, then denote
            \begin{align*}
                \mathbf s(\kappa,t) = 
                \begin{cases}
                        \frac{\sinh(\sqrt{-\kappa}t)}{\sqrt{-\kappa}} & \quad \kappa<0;\\
                         t                        & \quad \kappa \ge 0.
                \end{cases}
            \end{align*}
        Then,
        \begin{align*}
            \|J(\gamma(t))\| \le \mathbf s(\kappa,t) \| \nabla_{\dot\gamma} J(\gamma(0)) \|
        \end{align*}
        for all $t\in[0,b]$.
        \end{itemize}
    \end{lemma}
    
    Furthermore, we consider two comparison inequalities, which served as the law of consine over Riemannian manifolds.
    \begin{lemma}[\citealp{zhang2016first}]
        Let $\M$ be a Riemannian manifold with all sectional curvatures lower bounded by $\kappa$. Denote 
        \begin{align*}
            \zeta(\kappa,D) = 
            \begin{cases}
                \frac{\sqrt{-\kappa}D}{\tanh(\sqrt{-\kappa}D)} & \kappa<0; \\
                1 &  \kappa \ge 0.
            \end{cases}
        \end{align*}
    Then for a geodesic triangle $\triangle ABC$, we have
    \begin{align*}
        2\langle \expinv{A}{C},  \expinv{A}{B} \rangle \le d^2(A,B)+ \zeta(\kappa,d(A,B)) d^2(A,C) -  d^2(B,C).
    \end{align*}
    \end{lemma}
    
    \begin{lemma}[\citealp{alimisis2021momentum}]\label{lemma:sigma}
        Let $\M$ be a Riemannian manifold with all sectional curvatures upper bounded by $K$. Denote 
        \begin{align*}
            \sigma(K,D) = 
            \begin{cases}
                \frac{\sqrt{K}D}{\tan(\sqrt{K}D)} & K>0; \\
                1 &  K\le 0,
            \end{cases}
            \quad \text{and} \quad D(K) = 
            \begin{cases}
                \infty & K \le 0; \\
                \frac{\pi}{2\sqrt{K}} &  K > 0.
            \end{cases}
    \end{align*}
    If a geodesic triangle $\triangle ABC$ has diameter less than $D(K)$, then we have
    \begin{align*}
        2\langle \expinv{A}{C},  \expinv{A}{B} \rangle \ge d^2(A,B)+ \sigma(K,d(A,B)) d^2(A,C) -  d^2(B,C).
    \end{align*}
    \end{lemma}

    In addtion, \cite{ahn2020nesterov} demonstrate bi-Lipschitzness of the exponential map on positive curved space.
    \begin{lemma}[\citealp{ahn2020nesterov}]\label{lemma:lipexp}
        Let $A,B,C$ be points on Riemannian manifold $\M$ with sectional curvatures upper bounded by $K \ge 0$. If $d(A,C) \le \frac{\pi}{2\sqrt{K}}$, then 
        \begin{align*}
            \| \expinv{A}{B}- \expinv{A}{C} \| \le \sqrt{(1+2K d^2(A,B))} d(B,C) \le \frac{2}{\sqrt{\sigma(K,d(A,B))}} d(B,C). 
        \end{align*}
    \end{lemma}
     Finally, we introduce a technique that bounds the metric distortion by parallel transport.
    \begin{lemma}[\citealp{alimisis2021momentum}]\label{lemma:H}
        Let $\M$ be a Riemannian manifold with sectional curvature lower bounded by $\kappa$ and upper bounded by $K$. If a geodesic triangle $\triangle ABC$ admits a diameter less than $D(K)$, then there exists a point $p$ in the edge $AC$ such that
        \begin{align*}
            \| \expinv{A}{B} - \Gamma_C^A \expinv{C}{B} \| = -\Gamma_p^A \big(\nabla^2 (-\frac{1}{2} d^2(C,p) ) \big) \Gamma_A^p \expinv{A}{C},
        \end{align*}
    where $\big(\nabla^2 (-\frac{1}{2} d^2(C,p) ) \big)$ is the hessian of the function $-\frac{1}{2} d^2(C,\cdot)$ at point $p$.
    \end{lemma}
    
    We denote $H_{A,p}^c$ as the operator $-\Gamma_p^A \big(\nabla^2 (-\frac{1}{2} d^2(C,p) ) \big) \Gamma_A^p$. From the hessian comparison theorem \citep{lee2018introduction,alimisis2020continuous}, we know that all the eigenvalues of $H_{A,p}^c$ are in the range $[\sigma(K,d(C,p)) , \zeta(\kappa,d(C,p)) ]$. Since 
    \begin{align*}
        \max\big\{{\zeta(\kappa,d(C,p))-1, 1-\sigma(K,d(C,p))}\big\} \le \max({|\kappa|,|K|}) d^2(C,p),
    \end{align*} we have the following corollary.
    \begin{corollary}\label{lemma:tau}
        Let $\M$ be a Riemannian manifold with sectional curvature lower bounded by $\kappa$ and upper bounded by $K$. If a geodesic triangle $\triangle ABC$ has diameter $D\le D(K)$, then there exists a point $p$ lying in the edge $AC$ such that
        \begin{itemize}
            \item [(i)]  $ \| \expinv{A}{B} - \Gamma_C^A \expinv{C}{B} \| \le \zeta(\kappa,D) \|\expinv{A}{C}\|$;
            \item [(ii)] $\| \expinv{A}{B} - \Gamma_C^A \expinv{C}{B} - \expinv{A}{C}  \| \le d^2(C,p) \max\{|\kappa|,|K|\}  \|\expinv{A}{C}\|$.
        \end{itemize}
    \end{corollary}

\section{Proofs in Section \ref{sec:oco}  }\label{app:reg}

    \paragraph{Proof of Theorem \ref{thm: reg-c}} To ease the notation, we define $v_t = \sup_{x\in\K} \| \f_t(x)-\f_{t-1}(x) \|^2$ and $\nabla_t = \nabla\ft(x_t)$. By the g-convexity of $\f_{t+1}$, we have
       \begin{align}\label{eq:gc}
           \f_{t+1}(x_{t+1}) - \f_{t+1}(u_{t+1}) \le \langle \expinv{x_{t+1}}{u_{t+1}}, - \nabla_{t+1} \rangle.
        \end{align}
        Applying Lemma \ref{lemma:sigma} in the geodesic triangle $\triangle x_{t+1}u_{t+1} x_t$ we have,
        \begin{align*}
            2\langle \expinv{x_{t+1}}{u_{t+1}} , \expinv{x_{t+1}}{x_t} \ge d^2(x_{t+1},u_{t+1}) - d^2(x_{t},u_{t+1}) + \sigma_0 d^2(x_{t+1},x_t).
        \end{align*}
        Notice from Algorithm \ref{alg:R-OOGD}, we have $x_{t+1} = \exp_{x_t} (-2\eta \nabla_t + \eta \Gamma_{x_{t-1}}^{x_t} \nabla_{t-1}) $. This give us 
        \begin{align*}
            \expinv{x_{t+1}}{x_t} = \eta ( \Gamma_{x_t}^{x_{t+1}} (2\nabla_t - \Gamma_{x_{t-1}}^{x_t} \nabla_{t-1}) ).
        \end{align*}
        Thus, we have
        \begin{align} \label{eq: ge0}
            \nonumber 0  \le   \langle \expinv{x_{t+1}}{u_{t+1}} ,  &\Gamma_{x_t}^{x_{t+1}} (2\nabla_t - \Gamma_{x_{t-1}}^{x_t} \nabla_{t-1}) \rangle - \\
            & \qquad \frac{1}{2\eta} (d^2(x_{t+1},u_{t+1}) - d^2(x_{t},u_{t+1})) - \frac{\sigma_0}{2\eta} d^2(x_{t+1},x_t).
        \end{align}
        Combining \eqref{eq:gc} and \eqref{eq: ge0}, we have
        \begin{align}\label{eq:ftsum}
             \nonumber\f_{t+1}(x_{t+1}) - \f_{t+1}(u_{t+1}) &\le  \langle \expinv{x_{t+1}}{u_{t+1}}, - \nabla_{t+1} \rangle+ \langle \expinv{x_{t+1}}{u_{t+1}} ,  \Gamma_{x_t}^{x_{t+1}} (2\nabla_t - \Gamma_{x_{t-1}}^{x_t} \nabla_{t-1}) \rangle \\
             \nonumber& \quad -\frac{1}{2\eta} (d^2(x_{t+1},u_{t+1}) - d^2(x_{t},u_{t+1})) - \frac{\sigma_0}{2\eta} d^2(x_{t+1},x_t)\\
             \nonumber& = \langle \expinv{x_{t+1}}{u_{t+1}}, - \nabla_{t+1} +  \Gamma_{x_t}^{x_{t+1}} \nabla_t \rangle\\
             \nonumber& \quad - \langle \expinv{x_{t+1}}{u_{t+1}} ,  \Gamma_{x_t}^{x_{t+1}} (\nabla_t - \Gamma_{x_{t-1}}^{x_t} \nabla_{t-1}) \rangle \\
             & \quad -\frac{1}{2\eta} (d^2(x_{t+1},u_{t+1}) - d^2(x_{t},u_{t+1})) - \frac{\sigma_0}{2\eta} d^2(x_{t+1},x_t).
        \end{align}
        Considering the term $- \langle \expinv{x_{t+1}}{u_{t+1}} ,  \Gamma_{x_t}^{x_{t+1}} (\nabla_t - \Gamma_{x_{t-1}}^{x_t} \nabla_{t-1}) \rangle - \frac{\sigma}{2\eta}d^2(x_t,x_{t+1})$, we have
        \begin{align}\label{eq:start}
            \nonumber& - \langle \expinv{x_{t+1}}{u_{t+1}} ,  \Gamma_{x_t}^{x_{t+1}} (\nabla_t - \Gamma_{x_{t-1}}^{x_t} \nabla_{t-1}) \rangle - \frac{\sigma}{2\eta}d^2(x_t,x_{t+1})  \\
            \nonumber& \le  \langle -\expinv{x_{t+1}}{u_{t+1}} + \Gamma_{x_t}^{x_{t+1}}\expinv{x_t}{u_{t+1}} ,  \Gamma_{x_t}^{x_{t+1}} (\nabla_t - \Gamma_{x_{t-1}}^{x_t} \nabla_{t-1}) \rangle \\
            \nonumber& \quad + \langle -\expinv{x_t}{u_{t+1}} + \expinv{x_t}{u_t} ,  \nabla_t - \Gamma_{x_{t-1}}^{x_t} \nabla_{t-1} \rangle \\
            & \quad - \langle \expinv{x_t}{u_t} ,  \nabla_t - \Gamma_{x_{t-1}}^{x_t} \nabla_{t-1} \rangle - \frac{\sigma_0}{2\eta} d^2(x_{t+1},x_t).
        \end{align}
        By Lemma \ref{lemma:tau}, we have 
        \begin{align}
            \nonumber\langle &-\expinv{x_{t+1}}{u_{t+1}} + \Gamma_{x_t}^{x_{t+1}}\expinv{x_t}{u_{t+1}} ,  \Gamma_{x_t}^{x_{t+1}} (\nabla_t - \Gamma_{x_{t-1}}^{x_t} \nabla_{t-1}) \rangle \\
             \nonumber&\le \zeta_0 d(x_{t+1},x_t)\| \nabla_t - \Gamma_{x_{t-1}}^{x_t} \nabla_{t-1} \| \\
             &\le \frac{\zeta_0^2}{\sigma_0} \eta \|\nabla_t - \Gamma_{x_{t-1}}^{x_t} \nabla_{t-1} \|^2 + \frac{\sigma_0}{4\eta}d^2(x_t,x_{t+1}).
        \end{align}
        The latter inequality is due to the Young's inequality $\langle a,b\rangle \le \frac{\alpha\|a\|^2}{2} + \frac{\|b\|^2}{2\alpha}$. Also, By Lemma \ref{lemma:lipexp}, we have
        \begin{align}
             \nonumber\langle & -\expinv{x_t}{u_{t+1}} + \expinv{x_t}{u_t} ,  \nabla_t - \Gamma_{x_{t-1}}^{x_t} \nabla_{t-1} \rangle \\
             \nonumber&\le \frac{2}{\sqrt \sigma_0}d(u_t,u_{t+1}) \|\nabla_t - \Gamma_{x_{t-1}}^{x_t} \nabla_{t-1} \|\\
             \nonumber&\le \frac{1}{\sigma_0} \eta  \|\nabla_t - \Gamma_{x_{t-1}}^{x_t} \nabla_{t-1} \|^2 + \frac{d^2(u_{t+1},u_t)}{\eta}\\
             & \le \frac{1}{\sigma_0} \eta  \|\nabla_t - \Gamma_{x_{t-1}}^{x_t} \nabla_{t-1} \|^2 + \frac{d(u_{t+1},u_t)D}{\eta}.
        \end{align}
        Since $\ft$ is g-$L$-smooth, we have
        \begin{align} \label{eq:end}
            \nonumber\|\nabla_t - \Gamma_{x_{t-1}}^{x_t} \nabla_{t-1} \|^2 &\le  2 \| \nabla\ft(x_t) - \nabla\f_{t-1}(x_t) \|^2 + 2\| \nabla\f_{t-1}(x_t) - \Gamma_{x_{t-1}}^{x_t} \nabla_{t-1} \|^2 \\
            & \le 2v_t +  2L^2d^2(x_t,x_{t-1}).
        \end{align}
        Putting \eqref{eq:start}-\eqref{eq:end} together, we have
        \begin{align}\label{eq:tobesum}
            \nonumber\f_{t+1}(x_{t+1}) - \f_{t+1}(u_{t+1}) \le & \langle \expinv{x_{t+1}}{u_{t+1}} , -\nabla_{t+1} + \Gamma_{x_t}^{x_{t+1}}\nabla_t \rangle - \langle \expinv{x_{t}}{u_{t}} , -\nabla_{t} + \Gamma_{x_{t-1}}^{x_{t}}\nabla_{t-1} \rangle\\
            \nonumber & \quad +\frac{1}{2\eta}(d^2(x_{t},u_{t+1}) - d^2(x_{t+1},u_{t+1})) \\
            \nonumber& \quad + \frac{2(1+\zeta_0^2)}{\sigma_0} \eta v_t  +\frac{2L^2(1+\zeta_0^2)}{\sigma_0} \eta d(x_t,x_{t-1}) \\
            & \quad + \frac{d(u_t,u_{t+1})D}{\eta} - \frac{\sigma_0}{4\eta}d^2(x_t,x_{t+1}).
        \end{align}
        Summing \eqref{eq:tobesum} from $t=0$ to $T-1$ together, we have
        \begin{align*}
            \RegD(u_1,u_2,\dots,u_T) &= \sum_{t=0}^{T-1} \f_{t+1}(x_{t+1}) - \f_{t+1}(u_{t+1}) \\
            & \le  \sum_{t=0}^{T-1} \langle \expinv{x_{t+1}}{u_{t+1}} , -\nabla_{t+1} + \Gamma_{x_t}^{x_{t+1}}\nabla_t \rangle \\
            & \quad - \sum_{t=0}^{T-1} \langle \expinv{x_{t}}{u_{t}} , -\nabla_{t} + \Gamma_{x_{t-1}}^{x_{t}}\nabla_{t-1} \rangle \\
            & \quad + \sum_{t=0}^{T-1} \frac{1}{2\eta}d^2(x_{t},u_{t+1}) - \sum_{t=0}^{T-1} \frac{1}{2\eta} d^2(x_{t+1},u_{t+1})\\
            & \quad + \sum_{t=0}^{T-1} \frac{2L^2(1+\zeta_0^2)}{\sigma_0}\eta d(x_t,x_{t-1}) - \sum_{t=0}^{T-1} \frac{\sigma_0}{4\eta}d^2(x_t,x_{t+1}).
        \end{align*}
        Rearranging the summation, we can obtain
        \begin{align*}
             \RegD(u_1,u_2,\dots,u_T)& \le \frac{2(1+\zeta_0^2)}{\sigma_0}\eta V_T  +  \frac{D P_T}{\eta}\\
            & \le  \sum_{t=1}^{T} \langle \expinv{x_{t}}{u_{t}} , -\nabla_{t} + \Gamma_{x_{t-1}}^{x_{t}}\nabla_{t-1} \rangle -  \sum_{t=0}^{T-1} \langle \expinv{x_{t}}{u_{t}} , -\nabla_{t} + \Gamma_{x_{t-1}}^{x_{t}}\nabla_{t-1} \rangle \\
            & \quad +  \sum_{t=0}^{T-1} \frac{1}{2\eta}d^2(x_{t},u_{t+1}) - \sum_{t=1}^{T} \frac{1}{2\eta} d^2(x_{t},u_{t})\\
            & \quad + \sum_{t=0}^{T-1} \frac{4L^2\zeta_0^2}{\sigma_0}\eta d(x_t,x_{t-1}) - \sum_{t=1}^{T} \frac{\sigma_0}{4\eta}d^2(x_t,x_{t-1}) \\
            & \quad + \frac{4\zeta_0^2}{\sigma_0} \eta V_T  +  \frac{D P_T}{\eta}.\\
        \end{align*}
        Since $\eta \le \frac{\sigma_0}{4\zeta_0L}$, we have $\frac{4L^2\zeta_0^2\eta}{\sigma_0} \le  \frac{\sigma_0}{4\eta} $. Therefore, we have
        \begin{align*}
             \RegD(u_1,u_2,\dots,u_T) &= \sum_{t=0}^{T-1} \f_{t+1}(x_{t+1}) - \f_{t+1}(u_{t+1}) \\
            & \le  \langle \expinv{x_{T}}{u_{T}} , -\nabla_{T} + \Gamma_{x_{T-1}}^{x_{T}}\nabla_{T-1} \rangle -  \langle \expinv{x_{0}}{u_{0}} , -\nabla_{0} + \Gamma_{x_{-1}}^{x_{0}}\nabla_{-1} \rangle \\
            & \quad +  \frac{1}{2\eta}d^2(x_{0},u_{1}) - \frac{1}{2\eta} d^2(x_{T},u_{T})\\
            & \quad +  \frac{4L^2\zeta_0^2}{\sigma_0}\eta d(x_0,x_{-1}) - \frac{\sigma_0}{4\eta}d^2(x_T,x_{T-1}) \\
            & \quad + \frac{4\zeta_0^2}{\sigma_0} \eta V_T  +  \frac{2D P_T}{\eta}.\\
        \end{align*}
        As $x_{-1}=x_0=x_1$, we can see that $ \nabla_{0} = \nabla_{-1} $ and $d(x_0,x_{-1}) = 0$. In this way, we can see that 
        \begin{align*}
            \RegD(u_1,u_2,\dots,u_T) &= \sum_{t=0}^{T-1} \f_{t+1}(x_{t+1}) - \f_{t+1}(u_{t+1}) \\
            & \le  \langle \expinv{x_{T}}{u_{T}} , -\nabla_{T} + \Gamma_{x_{T-1}}^{x_{T}}\nabla_{T-1} \rangle +  \frac{1}{2\eta}d^2(x_{0},u_{1}) + \frac{4\zeta_0^2}{\sigma_0} \eta V_T  +  \frac{2D P_T}{\eta}\\
            & \le 2DG + \frac{D^2}{2\eta} + \frac{4\zeta_0^2}{\sigma_0} \eta V_T  +  \frac{2D P_T}{\eta}\\
            & \le \frac{D^2 + 2D P_T}{\eta} +  \frac{4\zeta_0^2}{\sigma_0}   \eta (V_T + G^2),
        \end{align*}
        which completes our proof. \hfill$\blacksquare$
    \paragraph{Proof of Theorem \ref{thm:aoogd}}
    We follow the idea to treat the dynamic regret by the meta-regret and expert-regret as in the work by \cite{hu2023minimizing}. For any $ i\le N $, it holds that
    \begin{align*}
        \sum_{t=1}^T \f_t(x_t) - \f_t(u_t) = \underbrace{\sum_{t=1}^T \f_t(x_t) - \f_t(x_{i,t})}_{\mathtt{meta-regret}} +  \underbrace{\sum_{t=1}^T \f_t(x_{t,i}) - \f_t(u_t)}_{\mathtt {expert-regret}}.
    \end{align*}
    Based on Theorem 2 of \cite{hu2023minimizing}, we obtain
    \begin{align*}
        \mathtt{meta}{\rm{-}}\mathtt{regret} &= \sum_{t=1}^T \f_t(x_t) - \f_t(x_{i,t})\\
        & \le \frac{2+\ln N}{\beta} + 3D_0^2\beta(V_T+G^2) + \\
        & \quad +\sum_{t=2}^T\big( 3\beta(D_0^4L^2+D_0^2G^2\zeta_0^2)- \frac{1}{4\beta} \big)\| w_t - w_{t-1}\|_1^2\\
        & \le \max\Big( 2\sqrt{3D_0^2(V_T+G^2)(2+\ln N)}, 2(2+\ln N)\sqrt{12(D_0^4L^2+D_0^2G^2\zeta_0^2)} \Big).
    \end{align*}
    
    Moreover, according to the dynamic regret in Theorem \ref{thm: reg-c}, for any index $i$ with step size $\eta_i$, we have:
    \begin{align*}
        \mathtt{expert}{\rm{-}}\mathtt{regret} &\le \sum_{t=1}^T \f_t(x_{t,i}) - \f_t(u_t)\\
        &\le \frac{D_0^2+2D_0P_T}{\eta_i} + \frac{4\zeta_0^2}{\sigma_0}\eta_i(V_T+G^2).\\
    \end{align*}
    Our step size pool $\mathcal{H} = \Big\{ \eta_i = 2^{i-1} \sqrt{\frac{D_0^2}{16\zeta_0^2 G^2 T}} \Big\}$ ensures that
    \begin{align*}
        \begin{cases}
            \min_{\mathcal{H}} = \sqrt{\frac{D_0^2}{16\zeta_0^2 G^2 T}} \le \sqrt{\frac{D_0^2+2D_0P_T}{4\zeta_0^2 (G^2+V_T)}},\\
            \max_{\mathcal{H}} \ge \frac{\sigma_0}{\zeta_0L}.
        \end{cases}
    \end{align*}
    So for the optimal step size $\eta^* = \min(\sqrt{\frac{D_0^2+P_T}{4\zeta_0^2 (G^2+V_T)}}, \frac{\sigma_0}{\zeta_0L})$, there exists $i^* \le N$ such that $\eta_{i^*} \le \eta^* \le 2 \eta_{i^*}$. Taking $i=i^*$, we have
        \begin{align*}
        \mathtt{expert}{\rm{-}}\mathtt{regret} &\le \sum_{t=1}^T \f_t(x_{t,i}) - \f_t(u_t)\\
        &\le \frac{D_0^2+2D_0P_T}{\eta_{i^*}} + \frac{4\zeta_0^2}{\sigma_0}\eta_{i^*}(V_T+G^2)\\
        &\le \frac{4\zeta_0^2}{\sigma_0}(V_T+G^2) \sqrt{\frac{D_0^2+2D_0P_T}{4\zeta_0^2 (G^2+V_T)}} \\
        &\quad + (D_0^2+2D_0P_T) (\sqrt{\frac{4\zeta_0^2 (G^2+V_T)}{D_0^2+2D_0P_T}} + \frac{4\zeta_0 L}{\sigma_0})\\
        &\le \frac{4\zeta_0}{\sqrt{\sigma_0}}\sqrt{(D_0^2+2D_0P_T)((V_T+G^2) )} + (D_0^2+2D_0P_T) \frac{4\zeta_0 L}{\sigma_0}.
    \end{align*}
    Combining meta-regret and expert-regret together, we finally get
    \begin{align*}
        \sum_{t=1}^T \f_t(x_t) - \f_t(u_t) &\le \max\Big( 2\sqrt{3D_0^2(V_T+G^2)(2+\ln N)}, 2(2+\ln N)\sqrt{12(D_0^4L^2+D_0^2G^2\zeta_0^2)} \Big) \\
        & \quad + \frac{4\zeta_0}{\sqrt{\sigma_0}}\sqrt{(D_0^2+2D_0P_T)((V_T+G^2) )} + (D_0^2+2D_0P_T) \frac{4\zeta_0 L}{\sigma_0} \\
        & = \bo(\sqrt{((V_T+1)\ln N)\ln N})+ \frac{\zeta_0}{\sqrt{\sigma_0}}bo(\sqrt{(1+P_T+V_T)(1+P_T)})\\
        & = \frac{\zeta_0}{\sqrt{\sigma_0}}\bo(\sqrt{(1+P_T+V_T)(1+P_T)}),
    \end{align*}
    which completes our proof. \hfill$\blacksquare$
    
\section{Computing the Corrected R-OOGD}\label{app:correct}       
 We first recall the corrected version of ROGD
\begin{align*}
    \begin{cases}
         x_{t+1} = \exp_{x_t}(-2\eta \nabla_t+\expinv{x_t}{\hat{x}_{t}}) \\
         \hat{x}_t = \exp_{x_{t-1}}(-\eta\nabla_{t-1}+\expinv{x_{t-1}}{\hat x_{t-1}}).
    \end{cases}
\end{align*}
To analyse the regret bound of the corrected version of R-OOGD, we follow \eqref{eq: ge0} and get
\begin{align*}
    0 \le \langle  \expinv{x_{t+1}}{x}, 
    \Gamma_{x_{t+1}}^{x_t} (-2\eta \nabla_t+\expinv{x_t}{\hat{x}_{t}}) \rangle + \frac{1}{2\eta} (d^2(x_t,x) -d^2(x_{t+1},x)) - \frac{\sigma_0}{2\eta}d^2(x_{t+1},x_t).
\end{align*}
Thus, by g-convexity we have,
\begin{align}\label{eq:add1}
    \nonumber \f_{t+1}(x_{t+1}) - \f_{t+1}(x) & \le \langle \expinv{x_{t+1}}{x}, -\nabla_{t+1}\rangle \\
    \nonumber& \le \langle \expinv{x_{t+1}}{x}, -\nabla_{t+1}+ \Gamma_{x_t}^{x_{t+1}} (-2\eta \nabla_t+\expinv{x_t}{\hat{x}_{t}}) \rangle \\
    \nonumber& \quad + \frac{1}{2\eta} (d^2(x_t,x) -d^2(x_{t+1},x))+\frac{\sigma_0}{2\eta}d^2(x_{t+1},x_t)\\
    \nonumber& = \langle \expinv{x_{t+1}}{x}, \Gamma_{x_t}^{x_{t+1}}(\nabla_t-\Gamma_{x_{t-1}}^{x_t}\nabla_{t-1}) \rangle \\
    \nonumber&- \langle \expinv{x_{t+1}}{x}, \nabla_{t+1}-\Gamma_{x_t}^{x_{t+1}} \nabla_t \rangle \\
    \nonumber& \quad + \frac{1}{2\eta} (d^2(x_t,x) -d^2(x_{t+1},x)) - \frac{\sigma_0}{2\eta}d^2(x_{t+1},x_t)\\
    & \quad + \langle \expinv{x_{t+1}}{x}, -\Gamma_{x_{t}}^{x_{t+1}} \big( \Gamma_{x_{t-1}}^{x_t} \eta \nabla_{t-1}+  \expinv{x_t}{\hat x_{t}} \big)\rangle.
\end{align}
The expression \eqref{eq:add1} follows from the proof of Theorem \ref{thm: reg-c} except the term 
\begin{align*}
     \langle \expinv{x_{t+1}}{x}, -\Gamma_{x_{t}}^{x_{t+1}} \big( \Gamma_{x_{t-1}}^{x_t} \eta \nabla_{t-1}+  \expinv{x_t}{\hat x_{t}} \big)\rangle.
\end{align*} 
Applying Corollary \ref{lemma:tau} in the geodesic triangle $\triangle x_{t-1}x_t\hat x_t$, we have
\begin{align*}
 & \quad \|  \expinv{x_{t-1}}{\hat{x}_t} - \expinv{x_{t-1}}{x_t} - \Gamma_{x_t}^{x_{t-1}} \expinv{x_t}{\hat x_{t}} \| \\
 &= \|2\eta\nabla_{t-1} - \expinv{x_{t-1}}{\hat x_{t-1}} - \eta\nabla_{t-1} + \expinv{x_{t-1}}{\hat x_{t-1}} - \Gamma_{x_t}^{x_{t-1}} \expinv{x_t}{\hat x_{t}}\|\\
 & \le d^2(\hat x_t,p) K_m \|\expinv{x_{t-1}}{x_t}\|,
\end{align*}
where $p$ lies in the geodesic $x_{t-1}x_t$. Since 
\begin{align*}
    d(\hat x_t,p) &\le d(\hat x_t,x_{t-1})+ d(x_{t-1},x_t) \\
                  &\le \| -2\eta\nabla_{t-1} + \eta\nabla_{t-2} -\eta\nabla_{t-2} +\expinv{x_{t-1}}{\hat x_{t-1}}\| \\
                  & \quad + \| -\eta\nabla_{t-1} + \eta\nabla_{t-2} -\eta\nabla_{t-2} +\expinv{x_{t-1}}{\hat x_{t-1}}\|\\
                  &\le 5\eta G + 2A_{t-1},
\end{align*} and
\begin{align*}
    \|\expinv{x_{t-1}}{x_t}\| \le 3\eta G+A_{t-1},
\end{align*}
we have:
\begin{align}\label{eq:ffffff}
    A_t \le K_m ( 5\eta G + 2A_{t-1})^2 (3\eta G+A_{t-1}),
\end{align}
indicating that the distortion in iteration $t-1$ evolves and accumulates in the distortion in iteration $t$. In the worst case scenario where $\eta G = 0.1$ and $K_m=1$, the equality always holds. We observe that $A_t \to \infty$, which implies that the corrected ROGD fails to achieve sublinear static regret.

\section{Proof of Theorem \ref{thm: avg-c}}\label{app:avg}
We first prove Lemma \ref{lemma:bdd}.
\paragraph{Proof of Lemma \ref{lemma:bdd}}
We prove by induction on $z_t = (x_t,y_t)$. The base case $d(z_0,z^*) \le D_1 \le 2D_1$ is straightforwardly hold. Then we assume that there exists a $K_0$ such that all $d(z_t,z^*) \le  2D_1$ holds for all $t \leq K_0$. Then we carry out induction step on $K_0+1$. The distance $d(z_{K_0+1},z^*)$ can be first bounded as
\begin{align*}
    d(z_{K_0+1},z^*) \le d(z_{K_0},z^*) + d(z_{K_0},z_{K_0+1}) \le 2D_1 + 3\eta G \le 3D_1.
\end{align*}

So, by setting $D_0=3D_1$, $\ft(x) = \f(x,y_t)$ and plugging in $u_1=u_2=\dots=u_n = x^*$ in the analysis of Theorem~\ref{thm: reg-c}, the R-OOGD for player-$\mathtt X$ holds for all $t\le K_0$
\begin{align}\label{eq:regx}
    \nonumber &\langle \expinv{x_{t+1}}{x^*}, -\nabla_x \f(x_{t+1},y_{t+1}) \rangle \\
    \nonumber &\le \langle \expinv{x_{t+1}}{x^*}, \big( \nabla_x \f(x_{t+1},y_{t+1}) - \Gamma_{x_t}^{x_{t+1}} \nabla_y\f(x_t,y_t) - \big( \nabla_x \f(x_{t},y_{t}) \big) - \Gamma_{x_{t-1}}^{x_{t}} \f(x_{t-1},y_{t-1}) \big) \rangle\\
    \nonumber& \quad + \frac{1}{2\eta} (d^2(x_t,x^*) - d^2(x_{t+1},x^*))-\frac{\sigma_1}{2\eta} d^2(x_t,x_{t+1})\\
    \nonumber& \le \langle \expinv{x_{t+1}}{x^*}, \nabla_x \f(x_{t+1},y_{t+1}) - \Gamma_{x_t}^{x_{t+1}} \f(x_t,y_t) \rangle \\
    \nonumber& \quad -  \langle \expinv{x_{t}}{x^*}, \nabla_x \f(x_{t},y_{t}) - \Gamma_{x_{t-1}}^{x_{t}} \f(x_{t-1},y_{t-1}) \rangle + \frac{1}{2\eta} (d^2(x_t,x^*) - d^2(x_{t+1},x^*))\\
    & \quad  -\frac{\sigma_1}{2\eta} d^2(x_t,x_{t+1}) + \zeta_1 \| \nabla_x \f(x_{t},y_{t}) - \Gamma_{x_{t-1}}^{x_{t}} \f(x_{t-1},y_{t-1}) \| d(x_t,x_{t+1}).
\end{align}
Similarly, the player-$\mathtt Y$ holds for all $t\le K_0$
\begin{align}\label{eq:regy}
    \nonumber &\langle \expinv{y_{t+1}}{y^*}, \nabla_y \f(x_{t+1},y_{t+1}) \rangle \\
    \nonumber& \le \langle \expinv{y_{t+1}}{y^*}, \nabla_y \f(x_{t+1},y_{t+1}) - \Gamma_{y_t}^{y_{t+1}} \nabla_y \f(x_t,y_t) \rangle \\
    \nonumber& \quad - \langle \expinv{y_{t}}{y^*}, \nabla_y \f(x_{t},y_{t}) - \Gamma_{y_{t-1}}^{y_{t}} \nabla_y \f(x_{t-1},y_{t-1}) \rangle + \frac{1}{2\eta} (d^2(y_t,y^*) - d^2(y_{t+1},y^*))\\
    &\quad  -\frac{\sigma_1}{2\eta} d^2(y_t,y_{t+1})  + \zeta_1 \| \nabla_y \f(x_{t},y_{t}) - \Gamma_{y_{t-1}}^{y_{t}} \nabla_y \f(x_{t-1},y_{t-1}) \| d(y_t,y_{t+1}).
\end{align}
Adding \eqref{eq:regx} and \eqref{eq:regy} together, for all $t\le K_0+1$,  we have 
  \begin{align}\label{eq:sumreg}
      \nonumber0 &\le  \langle \expinv{z_{t+1}}{z^*} , - \F(z_{t+1}) \rangle \\
      \nonumber& \le \langle \expinv{z_{t+1}}{z^*} , - \F(z_{t+1}) + \Gamma_{z_{t}}^{z_{t+1}} \F(z_t) \rangle - \langle \expinv{z_{t}}{z^*} , - \F(z_{t}) + \Gamma_{z_{t-1}}^{z_{t}} \F(z_{t-1}) \rangle \\
      \nonumber&\quad + \frac{1}{2\eta} (d^2(x_t,x^*) + d^2(y_t,y^*) ) - \frac{1}{2\eta} (d^2(x_{t+1},x^*) + d^2(y_{t+1},y^*) )\\
      \nonumber&\quad + \zeta_1 \| \nabla_y \f(x_{t},y_{t}) - \Gamma_{y_{t-1}}^{y_{t}} \nabla_y \f(x_{t-1},y_{t-1}) \| d(y_t,y_{t+1}) \\
      \nonumber&\quad + \zeta_1 \| \nabla_x \f(x_{t},y_{t}) - \Gamma_{x_{t-1}}^{x_{t}} \f(x_{t-1},y_{t-1}) \| d(x_t,x_{t+1})  \\
      &\quad -\frac{\sigma_1}{2\eta} d^2(x_t,x_{t+1}) -\frac{\sigma_1}{2\eta} d^2(y_t,y_{t+1}).
  \end{align}
  Taking Young's inequality with  
  \begin{align*}
  \begin{cases}
      a = ( d(x_t,x_{t+1}),  d(y_t,y_{t+1}) ) \\
      b =(\| \nabla_x \f(x_{t},y_{t}) - \Gamma_{x_{t-1}}^{x_{t}} \f(x_{t-1},y_{t-1}) \| , \| \nabla_y \f(x_{t},y_{t}) - \Gamma_{y_{t-1}}^{y_{t}} \nabla_y\f(x_{t-1},y_{t-1}) \| ), \\
      \alpha = L,
  \end{cases}
  \end{align*} we have
    \begin{align}\label{eq:smooth}
        \nonumber a \cdot b & \le  \frac{L}{2} (d^2(x_t,x_{t+1})+  d^2(y_t,y_{t+1})) \\
        \nonumber &\quad+\frac{1}{2L}(\| \nabla_x \f(x_{t},y_{t}) - \Gamma_{x_{t-1}}^{x_{t}} \f(x_{t-1},y_{t-1}) \|^2 +  \| \nabla_y \f(x_{t},y_{t}) - \Gamma_{y_{t-1}}^{y_{t}} \nabla_y \f(x_{t-1},y_{t-1}) \|^2).
        \\ & \le \frac{L}{2} d^2(z_t,z_{t+1}) + \frac{1}{2L} L^2 d^2(z_{t-1},z_t) =  \frac{L}{2} (d^2(z_t,z_{t+1}) + d^2(z_{t-1},z_t)).
    \end{align}
  Plugging \eqref{eq:smooth} into \eqref{eq:sumreg} yields
  \begin{align}
         \nonumber0 &\le  \langle \expinv{z_{t+1}}{z^*} , - \F(z_{t+1}) \rangle \\
       \nonumber & \le \langle \expinv{z_{t+1}}{z^*} , - \F(z_{t+1}) + \Gamma_{z_{t}}^{z_{t+1}} \F(z_t) \rangle - \langle \expinv{z_{t}}{z^*} , - \F(z_{t}) +   \nonumber\Gamma_{z_{t-1}}^{z_{t}} \F(z_{t-1}) \rangle \\
      &\quad + \frac{1}{2\eta} (d^2(z_t,z^*) - d^2(z_{t+1},z^*)) + \frac{\zeta_1 L}{2} (d^2(z_t,z_{t+1}) + d^2(z_{t-1},z_t))  -\frac{\sigma_1}{2\eta} d^2(z_t,z_{t+1}).
  \end{align}
  Since $\eta \le \frac{\zeta_1}{2\sigma_1L}$, we have $- \frac{\sigma_1}{2\eta} + \frac{\zeta_1 L}{2} \le - \frac{\zeta_1 L}{2}$, which gives us 
    \begin{align}\label{eq:sumreg2}
         \nonumber0 &\le  \langle \expinv{z_{t+1}}{z^*} , - \F(z_{t+1}) \rangle \\
       \nonumber & \le \langle \expinv{z_{t+1}}{z^*} , - \F(z_{t+1}) + \Gamma_{z_{t}}^{z_{t+1}} \F(z_t) \rangle - \langle \expinv{z_{t}}{z^*} , - \F(z_{t}) +   \nonumber\Gamma_{z_{t-1}}^{z_{t}} \F(z_{t-1}) \rangle \\
      &\quad + \frac{1}{2\eta} (d^2(z_t,z^*) - d^2(z_{t+1},z^*)) + \frac{\zeta_1 L}{2} ( - d^2(z_t,z_{t+1}) + d^2(z_{t-1},z_t)).
  \end{align}
 By summing \eqref{eq:sumreg} from $t=0$ to $K_0$, we observe that
  \begin{align*}
      0  & \le  \sum_{t=0}^{K_0} \langle \expinv{z_{t+1}}{z^*} , - \F(z_{t+1}) \rangle \\
         & \le   \langle \expinv{z_{K_0+1}}{z^*} , - \F(z_{K_0+1}) + \Gamma_{z_{K_0}}^{z_{K_0+1}} \F(z_{K_0}) \rangle - \langle \expinv{z_{0}}{z^*} , - \F(z_{-1}) +   \nonumber\Gamma_{z_{-1}}^{z_{0}} \F(z_{-1}) \rangle \\
         & \quad  +\frac{1}{2\eta} d^2(z_0,z^*) - \frac{1}{2\eta} d^2(z_{K_0+1},z^*) - \frac{\zeta_1L}{2} d^2(z_{K_0},z_{K_0+1}) + \frac{\zeta_1L}{2} d^2(z_0,z_{-1}).
  \end{align*}
    Furthermore, based on the fact that $z_0 = z_{-1}$, we can express that
      \begin{align*}
      0  & \le  \sum_{t=0}^{K_0} \langle \expinv{z_{t+1}}{z^*} , - \F(z_{t+1}) \rangle \\
         & \le   \langle \expinv{z_{K_0+1}}{z^*} , - \F(z_{K_0+1}) + \Gamma_{z_{K_0}}^{z_{K_0+1}} \F(z_{K_0}) \rangle \\
         & + \frac{1}{2\eta} d^2(z_0,z^*) - \frac{1}{2\eta} d^2(z_{K_0+1},z^*) - \frac{\zeta_1L}{2} d^2(z_{K_0},z_{K_0+1}).
  \end{align*}
  Using g-$L$-smoothness again, we have
    \begin{align*}
      0  & \le  \sum_{t=0}^{K_0} \langle \expinv{z_{t+1}}{z^*} , - \F(z_{t+1}) \rangle \\
         & \le  Ld(z_{K_0+1},z^*)d(z_{K_0+1},z_{K_0}) + \frac{1}{2\eta} d^2(z_0,z^*) - \frac{1}{2\eta} d^2(z_{K_0+1},z^*) - \frac{\zeta_1L}{2} d^2(z_{K_0},z_{K_0+1}) \\
         & \le \frac{L}{2}(d^2(z_{K_0+1},z^*) + d^2(z_{K_0+1},z_{K_0})) + \frac{1}{2\eta} d^2(z_0,z^*) \\
         & \quad - \frac{1}{2\eta} d^2(z_{K_0+1},z^*) - \frac{\zeta_1L}{2} d^2(z_{K_0},z_{K_0+1})  \\
         &\le \frac{L}{2}d^2(z_{K_0+1},z^*) + \frac{1}{2\eta} d^2(z_0,z^*) - \frac{1}{2\eta} d^2(z_{K_0+1},z^*).
  \end{align*}
  The last inequality is due to $\zeta_1 \ge 1$, and the inequality gives us 
  \begin{align*}
      d^2(z_{K_0+1},z^*) \le \frac{1}{1-\eta L} d^2(z_0,z^*) \le \frac{1}{1-\frac{\sigma_1}{2\zeta_1}} d^2(z_0,z^*) \le 2 d^2(z_0,z^*) = 2D_1^2.
  \end{align*}
  Thus, $d^2(z_t,z^*) \le 2D_1^2$ holds for $t = K_0+1$, By mathematical induction, the claim $d^2(z_t,z^*) \le 2D_1$  holds for all $t \geq 0$, which completes our proof. \hfill$\blacksquare$
  
  Now we shift our focus on proof of Theorem \ref{thm: avg-c}.
\paragraph{Proof of Theorem \ref{thm: avg-c}}
    By the g-convexity-concavity, for any $(x,y)\in \M\times\N $, there holds
    \begin{align*}
    \begin{cases}
         \sum_{t=1}^T \f(x_{t},y) - \sum_{t=1}^T \f(x_t,y_t)  \le  \sum_{t=1}^T \langle \nabla_y \f(x_t,y_t), \expinv{y_t}{y} \rangle \\
         \sum_{t=1}^T \f(x_{t},y_{t})- \sum_{t=1}^T \f(x,y_{t})  \le  \sum_{t=1}^T \langle -\nabla_x \f(x_t,y_t), \expinv{x_t}{x} \rangle. 
    \end{cases}   
    \end{align*}
    In this way, there holds
    \begin{align*}
        \sum_{t=1}^T \f(x_{t},y) - \sum_{t=1}^T \f(x,y_{t}) &\le \sum_{t=1}^T \langle \nabla_y \f(x_t,y_t), \expinv{y_t}{y} \rangle + \sum_{t=1}^T \langle -\nabla_x \f(x_t,y_t), \expinv{x_t}{x} \rangle\\
        & =  \sum_{t=1}^T  \langle -\F(z_t), \expinv{z_t}{z} \rangle.
    \end{align*}
     From the proof of Lemma \ref{lemma:bdd}, we can obtain
     \begin{align*}
         \sum_{t=1}^T \f(x_{t},y) - \sum_{t=1}^T \f(x,y_{t}) & \le \sum_{t=1}^T  \langle -\F(z_t), \expinv{z_t}{z} \rangle\\
         & \le \langle \expinv{z_{T}}{z} , - \F(z_{T}) + \Gamma_{z_{T-1}}^{z_{T}} \F(z_{T-1}) \rangle + \frac{1}{2\eta} d^2(z_0,z^*) \\
         & \quad - \frac{1}{2\eta} d^2(z_{K_0+1},z^*) - \frac{\zeta_1L}{2} d^2(z_{K_0},z_{K_0+1})\\
         & \le 2D_1G+\frac{1}{\eta}D_1^2.
     \end{align*}
    To complete the proof, it remains to show that
    \begin{align*}
        \f(\bar x_T,y) - \f(x,\bar y_T) \le \frac{1}{T}\sum_{t=1}^T \f(x_{t},y) -  \frac{1}{T}\sum_{t=1}^T \f(x,y_{t}),
    \end{align*}
    which can be proved by induction 
    \begin{align*}
        \nonumber \f(\bar x_T,y) &= \f( \exp_{\bar x_t}( \frac{1}{T} \exp^{-1}_{\bar x_{T-1}} x_{T} ),y)\\
        \nonumber&\le \frac{1}{T} \f(x_T,y) + \frac{T-1}{T} \f(\bar x_{T-1},y)\\
        \nonumber&\le \frac{1}{T} \f(x_T,y) + \frac{T-1}{T} \frac{1}{T-1} \f(x_{T-1},y) + \frac{T-1}{T} \frac{T-2}{T-1} \f(\bar x_{T-2},y)\\
        \nonumber&\le \frac{1}{T} \f(x_T,y) + \frac{1}{T} \f(x_{T-1},y) + \cdots + \frac{1}{T} \f(x_2,y) + \frac{1}{T} \f(\bar x_1,y)\\
        & = \frac{1}{T} \sum_{t=1}^T \f(x_t,y),
    \end{align*}
    and
    \begin{align*}
        \nonumber \f(x, \bar y_T) &= \f( x,\exp_{\bar y_t}( \frac{1}{T} \exp^{-1}_{\bar y_{T-1}} y_{T} ))\\
        \nonumber&\ge \frac{1}{T} \f(x,y_T) + \frac{T-1}{T} \f( x,\bar y_{T-1})\\
        \nonumber&\ge \frac{1}{T} \f(x,y_T) + \frac{1}{T} \f(x,y_{T-1}) + \cdots + \frac{1}{T} \f(x,y_2) + \frac{1}{T} \f(x, \bar y_1)\\
        & = \frac{1}{T} \sum_{t=1}^T \f(x,y_t).
    \end{align*}
    Then Theorem \ref{thm: avg-c} has been established. $\hfill \blacksquare$
\section{Proof of Lemma \ref{lemma:key}}\label{app:key}
We first propose some lemmas that are useful in proving Lemma \ref{lemma:key}.
\begin{lemma}[A variant of Gauss-Bonnet theorem, \citealp{lee2018introduction,chern1999lectures}]\label{lemma:rotate}
    Suppose $M$ is a manifold with sectional curvature in $[\kappa,K]$ and $\Xi (s,t):[0,1]\times[0.1] \to \M$ is a rectangle map. $\Gamma_\gamma$ is the parallel transport around the boundary curve $\gamma$ that $\gamma = \Xi(t,0) \cup \Xi(1,s) \cup \Xi (t,1) \cup \Xi(0,s) $. Denote vector fields $S(\Xi(s,t)) = \Xi_{*}\frac{\partial}{\partial s}(s,t)$, $T(\Xi(s,t)) = \Xi_{*}\frac{\partial}{\partial t}(s,t)$ and $K_m = \max(|\kappa|,|K|)$. Then we have
    \begin{align*}
        \| \Gamma_\gamma z - z \| \le 12 K_m\|z\| \int_0^1 \int_0^1 \|T\| \|S\| ds dt, \forall z\in T_{\Xi(0,0)}\M.
    \end{align*}
\end{lemma}
\begin{proof}
    We first extend $z$ to a vector field $Z(s_0,t_0) = Z(\Xi(s_0,t_0))$ by first parallel transporting $z$ along the curve $\Xi(0,t), (0\le t \le t_0)$ and then  parallel transporting  along the curve $\Xi(s,t_0), (0\le s \le s_0)$. It shows that
    \begin{align*}
        \begin{cases}
            \nabla_S Z(s,t) = 0\\
            \nabla_T Z(0,t) = 0
        \end{cases}
        \forall (s,t) \in [0,1]\times[0,1].
    \end{align*}
    
    For an arbitrary vector $w\in T_{\Xi(0,0)}\M$, we also extend it to $W(s_0,t_0)$ by first parallel transporting along the curve $\Xi(s,0), (0\le s\le 1)$, then along the curve $\Xi(1,t), (0\le t\le t_0)$, and along the curve $\Xi(s,t_0), (1\ge s\ge s_0)$. We can also have
    \begin{align*}
        \begin{cases}
            \nabla_S W(s,t) = 0\\
            \nabla_T W(1,t) = 0
        \end{cases}
        \forall (s,t) \in [0,1]\times[0,1].
    \end{align*}
    We denote two curves that $\xi_1 = \Xi(s,0),(0\le s\le 1)$ and $\xi_2 = \Xi(s,0) \cup \Xi (1,t), (0\le s,t\le 1)$. By the above notation, we find
    \begin{align*}
        \langle \Gamma_\gamma z - z , w \rangle &= \langle \Gamma_\gamma z , w \rangle - \langle z,w\rangle \\
        & = \langle \Gamma_{\xi_2} \Gamma_\gamma z , \Gamma_{\xi_2} w \rangle - \langle \Gamma_{\xi_1} z ,\Gamma_{\xi_2} w \rangle.
    \end{align*}
    From the way we extend $Z$ and $W$, we know that $\Gamma_{\xi_2} \Gamma_\gamma z = Z(1,1)$, $\Gamma_{\xi_2} w = W(1,1)$, $ \Gamma_{\xi_1} z = Z(1,0)$, and $\Gamma_{\xi_2} w = W(1,0)$, thus we have
    \begin{align*}
        \langle \Gamma_\gamma z - z , w \rangle &= \langle Z(1,1),W(1,1)\rangle - \langle Z(1,0) - W(1,0)\rangle\\
        & = \int_0^1 \frac{\partial}{\partial t} \langle Z(1,t),W(1,t)\rangle dt \\
        & = \int_0^1 T \langle Z(1,t),W(1,t)\rangle dt \\
        & = \int_0^1 \langle \nabla_T Z(1,t),W(1,t)\rangle + \langle Z(1,t), \nabla_T  W(1,t)\rangle dt.
    \end{align*}
    Due to the fact that $\nabla_T W(1,t) = 0$, we have
    \begin{align*}
        \langle \Gamma_\gamma z - z , w \rangle &= \int_0^1 \langle \nabla_T Z(1,t),W(1,t)\rangle dt\\
        & = \int_0^1 \Big( \langle \nabla_T Z(0,t),W(0,t) \rangle + \int_0^1  \partial_s \langle \nabla_T Z(s,t),W(s,t) \rangle ds \big)dt\\
        & = \int_0^1 \int_0^1 \langle \frac{\partial}{\partial s} \langle \nabla_T Z(s,t),W(s,t) \rangle ds dt\\
        & = \int_0^1 \int_0^1   S \langle \nabla_T Z(s,t),W(s,t) \rangle ds dt\\
        & = \int_0^1 \int_0^1    \langle \nabla_S \nabla_T Z(s,t),W(s,t) \rangle + \langle  \nabla_T Z(s,t),\nabla_S W(s,t) \rangle ds dt\\
        & = \int_0^1 \int_0^1    \langle \nabla_S \nabla_T Z(s,t),W(s,t) \rangle  ds dt.
    \end{align*}
    The last equality is from the fact that $\nabla_S W(s,t) = 0$.
    Since the curvature has the form
    \begin{align*}
        R(S,T,Z,W) = \langle \nabla_S \nabla_T Z(s,t),W(s,t) \rangle + \langle \nabla_T \nabla_S  Z(s,t),W(s,t) \rangle + \langle \nabla_{[S,T]}  Z(s,t),W(s,t) \rangle
    \end{align*}
    and we have $\nabla_S Z(s,t)  = 0$, $[S,T] = 0$, it holds that 
    \begin{align*}
        \langle \Gamma_\gamma z - z , w \rangle & = \int_0^1 \int_0^1    R(S,T,Z,W) ds dt\\
        & =  \int_0^1 \int_0^1    R(\frac{S}{\|S\|},\frac{T}{\|T\|},\frac{Z}{\|Z\|},\frac{W}{\|W\|})\|S\|\|T\|\|W\|\|Z\| ds dt.
    \end{align*}
    By Lemma \ref{lemma:sec} and $R(X,Y,X,Y) \le K_m\|X\|^2\|Y\|^2$, we have
    \begin{align*}
         R(\frac{S}{\|S\|},\frac{T}{\|T\|},\frac{Z}{\|Z\|},\frac{W}{\|W\|}) \le 12 K_m.
    \end{align*}
    Hence,
    \begin{align*}
        \langle \Gamma_\gamma z - z , w \rangle &\le 12 K_m \int_0^1 \int_0^1 \|S\|\|T\|\|W\|\|Z\| ds dt\\
        & =  12 K_m \|z\|\|w\|\int_0^1 \int_0^1 \|S\|\|T\| ds dt,
    \end{align*}
    which completes our proof since $w$ is arbitary.
\end{proof}    
    \begin{lemma}\label{lemma:rauch}
        Suppose $M$ is a manifold with sectional curvature in $[\kappa,K]$ and $\gamma:[0,b] \to \M$ is a geodesic with $\| \dot \gamma(0) \| =1$ ($b\le \frac{\pi}{\sqrt{K}}$ if $K>0$). If $J$ is a Jacobi field along $\gamma$ with $\|J(\gamma(0))\|=\alpha_1$ and $\|J(\gamma(b))\|=\alpha_2$, then we have
        \begin{align*}
            \|J(\gamma(t))\| \le \frac{ \mathbf s(\kappa,b)}{ \mathbf S(K,b)} (\alpha_1 + \alpha_2).
        \end{align*}

    \end{lemma}
    \begin{proof}
        We split $J(\gamma(t))=J_1(\gamma(t))+J_2(\gamma(t))$, where $J_1$ is a Jacobi field such that $J_1(\gamma(0)) = 0$ and $J_1(\gamma(b)) = J(\gamma(b))$, and where $J_2$ is a Jacobi field such that $J_2(\gamma(b)) = 0$ and $J_2(\gamma(0)) = J(\gamma(0))$.
        Applying the Jacobi comparison theorem (Lemma \ref{lemma:jacobi}), we have
        \begin{align} \label{eq:S1}
            \begin{cases}
               \mathbf S(K,b) \| \nabla_{\dot\gamma} J_1(0) \| \le \| J_1(\gamma(b)) \| = \alpha_2 \\
               \mathbf S(K,b) \| \nabla_{\dot\gamma} J_2(b) \| \le \| J_2(\gamma(0)) \| = \alpha_1 ,
            \end{cases}
        \end{align}
        and
        \begin{align} \label{eq:s1}
            \begin{cases}
             \| J_1(\gamma(t)) \|  \le \mathbf s(\kappa,t) \| \nabla_{\dot\gamma} J_1(0) \| \le   \mathbf s(\kappa,b) \| \nabla_{\dot\gamma} J_1(0) \| \\
             \| J_2(\gamma(t)) \|  \le \mathbf s(\kappa,(b-t)) \| \nabla_{\dot\gamma} J_2(b) \| \le   \mathbf s(\kappa,b) \| \nabla_{\dot\gamma} J_2(b)\|.
            \end{cases}
        \end{align}
        Combining \eqref{eq:S1} and \eqref{eq:s1} we have
        \begin{align*}
            \begin{cases}
             \| J_1(\gamma(t)) \| \le  \frac{ \mathbf s(\kappa,b)}{ \mathbf S(K,b)} \alpha_1  \\
             \| J_2(\gamma(t)) \| \le  \frac{ \mathbf s(\kappa,b)}{ \mathbf S(K,b)} \alpha_2,
            \end{cases}
        \end{align*}
        which gives us
        \begin{align*}
            \|J(\gamma(t))\| \le  \| J_1(\gamma(t)) \| +  \| J_2(\gamma(t)) \| \le \frac{ \mathbf s(\kappa,b)}{ \mathbf S(K,b)} (\alpha_1 + \alpha_2).
        \end{align*}
        This completes the proof.
    \end{proof}

    \begin{lemma}\label{lemma:sinh}
    Denote $K_m = \max(|\kappa|,|K|)$. If $0 \le b \le \frac{1}{\sqrt {K_m}}$, then we have
        $\frac{ \mathbf s(\kappa,b)}{ \mathbf S(K,b)} \le 3. $
    \end{lemma}
    \begin{proof}
        It suffice to proof the lemma in the case $K\ge 0 $ and $\kappa \le 0$. We first consider the case where $K > 0 $ and $\kappa < 0$, where 
        \begin{align}\label{eq:res0}
            \frac{ \mathbf s(\kappa,b)}{ \mathbf S(K,b)} = \frac{ \sqrt{K} \sinh(\sqrt{-k}b) }{\sqrt{-\kappa} \sin(\sqrt{K}b)}.
        \end{align}
        
        We show that $\cosh(ax) \le (1+a^2x^2)$ for $ 0 \le x\le \frac{1}{a}$. Let
        \begin{align*}
            f(x) = \cosh(ax) - 1- a^2x^2.
        \end{align*}
        We have 
        \begin{align*}
            \begin{cases}
                f(0) = 0\\
                f^\prime(x) = a\sinh(ax) - 2a^2x, f^\prime(0) = 0\\
                f^{\prime\prime}(x) = a^2(\cosh(ax) - 2).
            \end{cases}
        \end{align*}
        Since $\cosh(ax) - 2 \le 0$ for $0\le x\le \frac{1}{a}$, we have $f(x)\le 0$ for $0\le x\le \frac{1}{a}$.
        
        Then we show that $ \frac{\sinh(a x)}{\sin(c x)} \le \frac{a}{c} + \frac{a(a^2+c^2)x^2}{c} $ for $0 \le x \le \frac{1}{a}$. By defining 
        \begin{align*}
            g(x) = \sinh(ax) - ( \frac{a}{c} + \frac{a(a^2+c^2)x^2}{c} )\sin(c x),
        \end{align*}we have
         \begin{align*}
            \begin{cases}
                g(0) = 0\\
                g^\prime(x) = a\cosh(ax) -(a+ax^2(a^2+c^2)) \cos cx - (\frac{2a(a^2+c^2)x}{c})) \sin cx \\
            \end{cases}
        \end{align*}
        
    As $0\le x\le \frac{1}{a}$, we have
    \begin{align*}
        g^\prime(x) &\le a(1+a^2x^2) -(a+ax^2(a^2+c^2)) \cos cx - (\frac{2a(a^2+c^2)x}{c})) \sin cx\\
                    &\le a(1+a^2x^2)(1-\cos cx) - 2 a c x \sin cx\\
                    &\le  a(1+a^2x^2)(1-\cos^2 cx) - 2 a \sin^2 cx\\
                    & = a(1+a^2x^2) \sin^2 cx - 2a\sin^2 cx\\
                    & \le a(-1+a^2x^2)\sin^2 cx\\
                    & \le 0
    \end{align*}
    And thus, we have
    \begin{align*}
        \sinh(ax) - ( \frac{a}{c} + \frac{a(a^2+c^2)x^2}{c} )\sin(c x) \le 0,\quad 0\le x\le  \frac{1}{a}
    \end{align*}
    which is equivalent to 
    \begin{align}\label{eq:res1}
        \frac{\sinh(a x)}{\sin(c x)} \le \frac{a}{c} + \frac{a(a^2+c^2)x^2}{c},\quad 0\le x\le  \frac{1}{a}
    \end{align}
    Putting \eqref{eq:res0} into \eqref{eq:res1} with $a = \sqrt{-\kappa}$ and $c = \sqrt{K}$, we have, for $0 \le b \le \frac{1}{\sqrt {K_m}}$
    \begin{align*}
        \frac{ \mathbf s(\kappa,b)}{ \mathbf S(K,b)} &= \frac{ \sqrt{K} \sinh(\sqrt{-k}b) }{\sqrt{-\kappa} \sin(\sqrt{K}b)}\\
        & \le 1+(\kappa+K)b^2.\\
        &\le 1+2K_m\frac{1}{K_m}\\
        &\le 3,
    \end{align*}
    which completes the proof for the where case $K>0$ and $\kappa<0$. Since $ \frac{ \mathbf s(\kappa,b)}{ \mathbf S(K,b)}$ is continuous on $K$ and $\kappa$, we have proved the case with $K=0$ or $\kappa = 0$ by continuity. 
    \end{proof}
    
    \begin{lemma}\label{lemma:norm}
        In Algorithm \ref{alg:R-OGDA}, if $\eta \le \frac{1}{8L}$, it satisfies
        \begin{align}
            \frac{1}{2} \le \frac{\|\F(z_{t+1})\|}{\|\F(z_t)\|} \le \frac{3}{2}.
        \end{align}
    \end{lemma}
    The proof is as same as that the prior work by \cite{chavdarova2021last}, so we omit the proof.
    
    Then we begin our proof of Lemma \ref{lemma:key}.
    \paragraph{Proof of Lemma \ref{lemma:key}}
        We first estimate the length $d(z_{t-1},z_t)$ by,
        \begin{align*}
            d(z_{t-1},z_t) &= \| 2 \eta \F(z_{t-1}) - \Gamma_{z_{t-2}}^{z_{t-1}}\eta(\F(z_{t-2})) \| \\
            & \le \eta\|\F(z_{t-1})\| + \eta \|  \F(z_{t-1}) - \Gamma_{z_{t-2}}^{z_{t-1}}(\F(z_{t-2})) \| \\
            & \le \eta\|\F(z_{t-1})\| + \eta L d(z_{t-1},z_{t-2}) \\
            & =\eta\|\F(z_{t-1})\| + \eta L \| 2\F(z_{t-2}) - \Gamma_{z_{t-3}}^{z_{t-2}} \F(z_{t-3}) \| \\
            & \le \eta\|\F(z_{t-1})\| + \eta L (\|2\F(z_{t-2}) \| + \| \F(z_{t-3}) \|).
        \end{align*}
    By Lemma \ref{lemma:norm}, we have $\|\F(z_{t-2}) \| \le 2\|\F(z_{t-1})\| $ and $\|\F(z_{t-3})\| \le 4\|\F(z_{t-1})\|  $. So, we have
    \begin{align}
            d(z_{t-1},z_t) \le (1+8L\eta)\eta \| \F(z_{t-1})\|.
    \end{align}
    Similarly, $d(\hat z_{t+1}, z_t)$ and $d(\hat z_{t}, z_{t-1})$ can be bounded by 
    \begin{align*}
        d(\hat z_{t+1}, z_t) &= \| \eta \F(z_{t}) - \Gamma_{z_{t-1}}^{z_t} \F(z_{t-1})  \| \\
        & \le L\eta d(z_t,z_{t-1}) \\
        & \le L\eta^2 \| 2 \eta \F(z_{t-1}) - \Gamma_{z_{t-2}}^{z_{t-1}}\eta(\F(z_{t-2})) \|\\
        & \le L\eta^2 (2 \|\F(z_{t-1})\| +\|\F(z_{t-2})\| )\\
        & \le 4L\eta^2\|\F(z_{t-1})\|,
    \end{align*}
    and
    \begin{align*}
        d(\hat z_{t}, z_{t-1}) &= \eta \|  \F(z_{t-1}) - \Gamma_{z_{t-2}}^{z_{t-1}}(\F(z_{t-2}))  \|\\
        &\le L\eta^2  \|2\F(z_{t-2}) - \Gamma_{z_{t-3}}^{z_{t-2}}(\F(z_{t-3})) \|\\
        &\le 8 L\eta^2\|\F(z_{t-1})\|.
    \end{align*}
        
    Consequently, the $l_t: =  d(z_{t-1},z_t) +  d(\hat z_{t+1}, z_t) + d(\hat z_{t}, z_{t-1})$ has the bound
    \begin{align*}
        l_t &\le (1+8L\eta)\eta \| \F(z_{t-1})\| +  4L\eta^2\|\F(z_{t-1})\| + 8 L\eta^2\|\F(z_{t-1})\|\\
        & \le (1+20L\eta)\eta \|\F(z_{t-1})\|.
    \end{align*}
    Since $\eta \le \frac{1}{20L}$, we have
    \begin{align}\label{eq:dt}
        l_t \le 2\eta \|\F(z_{t-1})\|.
    \end{align}
    
    Now we examine (i). Applying Lemma \ref{lemma:tau} in the geodesic triangle $\triangle z_{t-1}\hat z_{t+1}z_{t}$ yields
    \begin{align}\label{eq:hh1}
        \expinv{z_{t-1}}{\hat z_{t+1}} = H_{z_{t-1},p_1}^{\hat z_{t+1}} \expinv{z_{t-1}}{z_t} + \Gamma_{z_t}^{z_{t-1}} \expinv{z_t}{\hat z_{t+1}},
    \end{align}
    where $p_1$ lies in the geodesic between $z_{t-1}$ and $z_t$. Applying Lemma \ref{lemma:tau} in the geodesic triangle $\triangle z_{t-1}\hat z_{t+1}\hat z_{t}$ yields
    \begin{align}\label{eq:hh2}
        \expinv{z_{t-1}}{\hat z_{t+1}} = H_{z_{t-1},p_2}^{\hat z_{t+1}} \expinv{z_{t-1}}{\hat z_t} + \Gamma_{\hat z_{t}}^{z_{t-1}} \expinv{\hat z_t}{\hat z_{t+1}}.
    \end{align}
     where $p_2$ lies in the geodesic between $z_{t-1}$ and $\hat z_t$.

    Combing \eqref{eq:hh1} and \eqref{eq:hh2}, we have
    \begin{align}\label{eq:hh3}
    \nonumber H_{z_{t-1},p_2}^{\hat z_{t+1}}& (-\eta \F(z_{t-1}) + \eta \Gamma_{z_{t-2}}^{z_{t-1}} \F(z_{t-2})  ) - \Gamma_{\hat z_{t}}^{z_{t-1}} \Gamma_{\hat z_{t+1}}^{\hat z_{t}} G_{t+1}\\
    &= H_{z_{t-1},p_1}^{\hat z_{t+1}} (-2\eta \F(z_{t-1}) + \eta \Gamma_{z_{t-2}}^{z_{t-1}} \F(z_{t-2})  ) + \Gamma_{z_t}^{z_{t-1}} (-\eta \F(z_t) + \Gamma_{z_{t-1}}^{z_{t}}\eta \F(z_{t-1})).
    \end{align}
    Rearranging \eqref{eq:hh3}, we have
    \begin{align*}
        \Gamma_{\hat z_{t}}^{z_{t-1}} \Gamma_{\hat z_{t+1}}^{\hat z_{t}} G_{t+1} - \Gamma_{z_t}^{z_{t-1}} (\eta \F(z_t) ) & = (H_{z_{t-1},p_2}^{\hat z_{t+1}}-Id) (-\eta \F(z_{t-1}) + \eta \Gamma_{z_{t-2}}^{z_{t-1}} \F(z_{t-2})  ) \\
        & \quad- (H_{z_{t-1},p_1}^{\hat z_{t+1}} - Id) (-2\eta \F(z_{t-1}) + \eta \Gamma_{z_{t-2}}^{z_{t-1}} \F(z_{t-2})  ).
    \end{align*}
    By Corollary \ref{lemma:tau}, we have
    \begin{align*}
        \|\Gamma_{\hat z_{t}}^{z_{t-1}} \Gamma_{\hat z_{t+1}}^{\hat z_{t}} G_{t+1} - \eta \Gamma_{z_t}^{z_{t-1}} \F(z_t)  \| &\le K_m d^2(\hat z_{t+1},p_1) d(z_{t-1},\hat z_t) +  K_m d^2(\hat z_{t+1},p_2) d(z_{t-1}, z_t).
    \end{align*}
    Since $p_1$ lies in the geodesic between $z_{t-1}$ and $z_t$, we have $d(\hat z_{t+1},p_1) \le d(\hat z_{t+1},z_t) + d(\hat z_{t},z_{t-1}) \le l_t$. Also, we have $d(\hat z_{t+1},p_2) \le l_t$. Then, we have
    \begin{align*}
        \|\Gamma_{\hat z_{t}}^{z_{t-1}} \Gamma_{\hat z_{t+1}}^{\hat z_{t}} G_{t+1} - \eta \Gamma_{z_t}^{z_{t-1}} \F(z_t) \| &\le  K_m l_t^2(d(z_{t-1},\hat z_t)+d(z_{t-1}, z_t)) \\
        &\le  K_m l_t^3 = K_m 8 \eta^3 \|\F(z_{t-1})\|^3.
    \end{align*}
    Then,
    \begin{align*}
        \|G_{t+1}\|^2 - \|\eta \F(z_t)\|^2 &= \|\Gamma_{\hat z_{t}}^{z_{t-1}} \Gamma_{\hat z_{t+1}}^{\hat z_{t}} G_{t+1} \|^2- \|\eta \Gamma_{z_t}^{z_{t-1}} \F(z_t)  \|^2\\
        &\le  \|\Gamma_{\hat z_{t}}^{z_{t-1}} \Gamma_{\hat z_{t+1}}^{\hat z_{t}} G_{t+1} - \eta \Gamma_{z_t}^{z_{t-1}} \F(z_t)  \|^2 = 64K_m\eta^6 \|\F(z_{t-1})\|^6.
    \end{align*}
    
    Next we examine (ii). It is shown that
    \begin{align*}
        \| \Gamma_{\hat z_{t+1}}^{z_t} G_{t+1}  -\eta \F(z_t)\| &\le \| \Gamma_{\hat z_{t+1}}^{z_t} G_{t+1} -\Gamma_{\hat z_{t+1}}^{z_{t}} \Gamma_{\hat z_t}^{\hat z_{t+1}} \Gamma_{z_{t-1}}^{\hat z_{t}}\Gamma_{z_t}^{z_{t-1}} \eta \F(z_t)\|\\
        & \quad + \| \Gamma_{\hat z_{t+1}}^{z_{t}} \Gamma_{\hat z_t}^{\hat z_{t+1}} \Gamma_{z_{t-1}}^{\hat z_{t}}\Gamma_{z_t}^{z_{t-1}} \eta \F(z_t) - \eta \F(z_t)\|\\
         &= \| \Gamma_{\hat z_{t}}^{z_{t-1}} \Gamma_{\hat z_{t+1}}^{\hat z_{t}} G_{t+1} - \eta \Gamma_{z_t}^{z_{t-1}} \F(z_t)\|\\
        & \quad + \| \Gamma_{\hat z_{t+1}}^{z_{t}} \Gamma_{\hat z_t}^{\hat z_{t+1}} \Gamma_{z_{t-1}}^{\hat z_{t}}\Gamma_{z_t}^{z_{t-1}} \eta \F(z_t) - \eta \F(z_t)\|\\
        &\le  K_m 8 \eta^3 \|\F(z_{t-1})\|^3 + \| \Gamma_{\hat z_{t+1}}^{z_{t}} \Gamma_{\hat z_t}^{\hat z_{t+1}} \Gamma_{z_{t-1}}^{\hat z_{t}}\Gamma_{z_t}^{z_{t-1}} \eta \F(z_t) -\eta \F(z_t)\|.
    \end{align*}
    We turn our focus on the geodesic rectangle $z_t z_{t-1} \hat z_t \hat z_{t+1}$. Denote $\gamma_1(s):[0,1]\to\M$ as the geodesic from $z_{t-1}$ to $\hat z_t$ and $\gamma_1(s):[0,1]\to\M$ as the geodesic from $z_{t}$ to $\hat z_{t+1}$. We define a rectangle map $\Xi:[0,1]\times[0,1] \to \M$ such that
    \begin{align*}
        \Xi(s,t) = \exp_{\gamma_1(s)} (t \expinv{\gamma_1(s)}{\gamma_2(s)}).
    \end{align*}
    The boundary curve of $\Xi$ is the geodesic rectangle $z_t z_{t-1} \hat z_t \hat z_{t+1}$. Denote $S = \Xi_* (\frac{\partial}{\partial s})$ and $T = \Xi_* (\frac{\partial}{\partial t})$. From Lemma \ref{lemma:rotate}, we have
    \begin{align}\label{eq:rnorm}
         \| \Gamma_{\hat z_{t+1}}^{z_{t}} \Gamma_{\hat z_t}^{\hat z_{t+1}} \Gamma_{z_{t-1}}^{\hat z_{t}}\Gamma_{z_t}^{z_{t-1}} \eta \F(z_t) -\eta \F(z_t)\| \le 12K_m \eta \| \F(z_t)\| \int_0^1\int_0^1 \|S\| \|T\| ds dt.
    \end{align}
    Notice that each $t$-curve is a geodesic from $ \gamma_1(s)$ to $\gamma_2(s)$, we have 
    \begin{align}\label{eq:T}
     \nonumber \|T(s,t)\| &= (d(\gamma_1(s),\gamma_2(s)))\\
     \nonumber &\le (d(\gamma_1(s),z_{t-1}) + d(z_{t-1}, z_{t}) + d(\gamma_2(s),z_{t})) \\
     \nonumber &\le (d(\hat z_t,z_{t-1}) + d(z_{t-1}, z_{t}) + d(\hat z_{t+1},z_{t})) \\
               &\le l_t = 2\eta\|\F(z_{t-1})\|.
    \end{align}
    Moreover, the vector field $S$ is a Jacobi field along every $t$-curve with $\|S(s,0)\| = d(\hat z_t,z_{t-1})$ and $\|S(s,1)\| = d(z_t,\hat z_{t+1})$. By Lemma \ref{lemma:rauch}, we have
    \begin{align*}
        \|S(s,t)\| \le \frac{ \mathbf s(\kappa,\|T(s,t)\|)}{ \mathbf S(K, \|T(s,t)\|) } (d(\hat z_t,z_{t-1}) + d(z_t,\hat z_{t+1})).
    \end{align*}
    Since $\eta \le \frac{1}{2\sqrt{K_m}G}$, we have $\|T(s,t)\| \le  2\eta\|\F(z_{t-1})\| \le \frac{1}{\sqrt{K_m}}$, and thus by Lemma \ref{lemma:sinh}
    \begin{align*}
         \|S(s,t)\| &\le \frac{ \mathbf s(\kappa,\|T(s,t)\|)}{ \mathbf S(K, \|T(s,t)\|) } (d(\hat z_t,z_{t-1}) + d(z_t,\hat z_{t+1}))\\
         &\le 3(d(\hat z_t,z_{t-1}) + d(z_t,\hat z_{t+1}))
    \end{align*}
    With $d(\hat z_t,z_{t-1}) \le 4L\eta^2\|\F(z_{t-1})\| $ , $d(z_t,\hat z_{t+1}) \le 8L\eta^2\|\F(z_{t-1})\|  $ and $\eta\le \frac{1}{20L}$, we have
    \begin{align}\label{eq:S}
         \|S(s,t)\| \le \frac{9}{5} \eta \|\F(z_{t-1})\| \le 2 \eta \|\F(z_{t-1})\|
    \end{align}
    Taking \eqref{eq:T} and \eqref{eq:S} in \eqref{eq:rnorm}, we have
    \begin{align*}
         \| \Gamma_{\hat z_{t+1}}^{z_{t}} \Gamma_{\hat z_t}^{\hat z_{t+1}} \Gamma_{z_{t-1}}^{\hat z_{t}}\Gamma_{z_t}^{z_{t-1}} \eta \F(z_t) -\eta \F(z_t)\| &\le 12K_m \eta \| \F(z_t)\| \cdot 4 \eta^2 \|\F(z_{t-1})\|^2 \\
         &\le 12K_m \eta  8 \eta^2 \|\F(z_{t-1})\|^3 = 96K_m \eta^3 \|\F(z_{t-1})\|^3.
    \end{align*}
    Thus, we have
    \begin{align*}
        \| \Gamma_{\hat z_{t+1}}^{z_t} G_{t+1}  -\eta \F(z_t)\| &\le  K_m 8 \eta^3 \|\F(z_{t-1})\|^3 + \| \Gamma_{\hat z_{t+1}}^{z_{t}} \Gamma_{\hat z_t}^{\hat z_{t+1}} \Gamma_{z_{t-1}}^{\hat z_{t}}\Gamma_{z_t}^{z_{t-1}} \eta \F(z_t) -\eta \F(z_t)\|\\
        & \le  K_m 8 \eta^3 \|\F(z_{t-1})\|^3  +  96K_m \eta^3 \|\F(z_{t-1})\|^3\\
        & = 104K_m\eta^3 \|\F(z_{t-1})\|^3,
    \end{align*}
    which completes our proof.\hfill$\blacksquare$
    
\section{Proof of Theorems \ref{thm: last-c} and \ref{thm: last-sc} }\label{app:last}

\paragraph{Proof of Theorem \ref{thm: last-c}} Define $\hat z_{t+1} = \exp_{z_t}(\eta \F(z_{t}) + \Gamma_{z_{t-1}}^{z_t} \F(z_{t-1}))$. We consider the following Lyapunov function
\begin{align*}
    \phi_{t+1} &:=d(\hat z_{t+1},z^*) + \sigma_1 d^2(z_{t+1},\hat z_{t+2})\\
                &= d(\hat z_{t+1},z^*) + \sigma_1 \eta^2 \| \nabla \F(z_{t+1}) - \Gamma_{z_t}^{z_{t+1}} \F(z_t) \|^2.
\end{align*}

The difference between $\phi_{t+2}$ and $\phi_{t+1}$ is 
\begin{align}\label{eq:diff}
    \nonumber\phi_{t+2} - \phi_{t+1}  & = d^2(\hat z_{t+2},z^*) -  d^2(\hat z_{t+1},z^*) + 2 \sigma_1 \langle \eta \F(z_{t+1}), \Gamma_{z_t}^{z_{t+1}} \eta \F(z_t) \rangle\\
       & \quad + \sigma_1 \eta^2 (\|\F(z_{t+2}) \|^2 - \|\F(z_t)\|^2 - 2 \langle \F(z_{t+2}), \Gamma_{z_{t+1}}^{z_{t+2}}\F(z_{t+1}) \rangle .
\end{align}
In the geodesic $\triangle\hat z_{t+1}\hat z_{t+2}z^*$, Lemma \ref{lemma:bdd} and Lemma \ref{lemma:sigma} give
\begin{align}\label{eq:1}
    d^2(\hat z_{t+2},z^*) - d^2(\hat z_{t+1},z^*) \le 2\langle G_{t+2}, \expinv{\hat z_{t+2}}{z^*} \rangle - \sigma_1\|G_{t+2}\|^2,
\end{align}
where $G_{t+2} = \expinv{\hat z_{t+2}}{\hat z_{t+1}}.$ Substituting \eqref{eq:1} into \eqref{eq:diff}, we have
\begin{align*}
    \nonumber\phi_{t+2} - \phi_{t+1}  & = 2\langle G_{t+2}, \expinv{\hat z_{t+2}}{z^*} \rangle - \sigma_1\|G_{t+2}\|^2 + 2 \sigma_1 \langle \eta \F(z_{t+1}), \Gamma_{z_t}^{z_{t+1}} \eta \F(z_t) \rangle\\
    \nonumber           & \quad + \sigma_1 \eta^2 (\|\F(z_{t+2}) \|^2 - \|\F(z_t)\|^2 - 2 \langle \F(z_{t+2}), \Gamma_{z_{t+1}}^{z_{t+2}}\F(z_{t+1})) \rangle. \\
\end{align*}
With a matter of algebraic calculations, we have
\begin{align}\label{eq:final1la}
     \nonumber \phi_{t+2} - \phi_{t+1}  &\le 2\langle G_{t+2}, \expinv{\hat z_{t+2}}{z^*} \rangle + 2\sigma_1 \langle \eta \F(z_{t+1}) , -\eta \F(z_{t+1}) + \eta \Gamma_{z_{t}}^{z_{t+1}} \F(z_{t}) \rangle\\
     \nonumber &\quad + \eta \sigma_1 \big(\|\F(z_{t+2})\|^2 - \|\F(z_{t})\|^2 + \|\F(z_{t+1})\|^2 - 2 \langle \eta \F(z_{t+2}) ,\eta \Gamma_{z_{t+1}}^{z_{t+2}} \F(z_{t+1}) \rangle  \big)\\
    &\quad + \sigma_1 (\|\eta \F(z_{t+1})\|^2 - \|G_{t+2}\|^2).
\end{align}

We define
\begin{align*}
    \begin{cases}
        A:= 2\langle G_{t+2}, \expinv{\hat z_{t+2}}{z^*} \rangle + 2\sigma_1 \langle \eta \F(z_{t+1}) , -\eta \F(z_{t+1}) + \eta \Gamma_{z_{t}}^{z_{t+1}} \F(z_{t}) \rangle;\\
        B:= \eta \sigma_1 \big(\|\F(z_{t+2})\|^2 - \|\F(z_{t})\|^2 + \|\F(z_{t+1})\|^2 - 2 \langle \eta \F(z_{t+2}) ,\eta \Gamma_{z_{t+1}}^{z_{t+2}} \F(z_{t+1}) \rangle  \big);\\
        C:= \sigma_1 (\|\eta \F(z_{t+1})\|^2 - \|G_{t+2}\|^2).
    \end{cases}
\end{align*}
and analyze them term by term.

We rewrite $A$ as
\begin{align}\label{eq:1last}
    \nonumber A = &2\langle \Gamma_{\hat z_{t+2}}^{z_{t+1}} G_{t+2}, \Gamma_{\hat z_{t+2}}^{z_{t+1}}\expinv{\hat z_{t+2}}{z^*} \rangle + 2\sigma_1 \langle \eta \F(z_{t+1}) , -\eta \F(z_{t+1}) + \eta \Gamma_{z_{t}}^{z_{t+1}} \F(z_{t}) \rangle\\
    \nonumber&= 2 \langle \Gamma_{\hat z_{t+2}}^{z_{t+1}}\expinv{\hat z_{t+2}}{z^*} , \Gamma_{\hat z_{t+2}}^{z_{t+1}} G_{t+2} - \eta \F(z_{t+1}) \rangle \\
    \nonumber& \quad + 2 \langle  \Gamma_{\hat z_{t+2}}^{ z_{t+1}}\expinv{\hat z_{t+2}}{z^*} - \expinv{z_{t+1}}{z^*}, \eta \F(z_{t+1}) \rangle\\
     &\quad + 2 \langle \expinv{z_{t+1}}{z^*} ,  \eta \F(z_{t+1}) \rangle +  2\sigma_1 \langle \eta \F(z_{t+1}) , -\eta \F(z_{t+1}) + \eta \Gamma_{z_{t}}^{z_{t+1}} \F(z_{t}) \rangle.
\end{align}

According to Lemma \ref{lemma:H}, there exists a point $p$ in the geodesic between $z_{t+1}$ and $\hat z_{t+2}$ such that
\begin{align}\label{eq:2last}
    \expinv{z_{t+1}}{z^*} - \Gamma_{\hat z_{t+2}}^{z_{t+1}} \expinv{\hat z_{t+2}}{z^*} = H_{z_{t+1},p}^{z^*} (-\eta \F(z_{t+1}) + \eta \Gamma_{z_t}^{z_{t+1}} \F(z_t)).
\end{align}
Combining \eqref{eq:1last} and \eqref{eq:2last}, we have
\begin{align*}
    A &=  2 \langle \Gamma_{\hat z_{t+2}}^{\hat z_{t+1}} G_{t+2} - \eta \F(z_{t+1}), \Gamma_{\hat z_{t+2}}^{\hat z_{t+1}}\expinv{\hat z_{t+2}}{z^*} \rangle \\
    \nonumber& \quad + 2 \langle  \eta \F(z_{t+1}), (-H_{z_{t+1},p}^{z^*} + \sigma_1 Id) [ -\eta \F(z_{t+1})] + \eta \Gamma_{z_{t}}^{z_{t+1}} \F(z_{t})\rangle\\
     &\quad + 2 \langle  \eta \F(z_{t+1}), \expinv{z_{t+1}}{z^*} \rangle.
\end{align*}

Since the eigenvalue of $-H_{z_{t+1},p}^{z^*}$ lies in $[-\zeta_1,-\sigma_1]$, we have
\begin{align*}
    \langle  \eta \F(z_{t+1}), (-H_{z_{t+1},p}^{z^*} + \sigma_1)& -\eta \F(z_{t+1}) + \eta \Gamma_{z_{t}}^{z_{t+1}} \F(z_{t})\rangle \\ &\le (\zeta_1-\sigma_1) \| \eta \F(z_{t+1}) - \eta \Gamma_{z_t}^{z_{t+1}} \F(z_t) \| \| \eta \F(z_{t+1})  \| \\
    & \le (\zeta_1-\sigma_1)L\eta^2 d(z_{t+1},z_t) \|\F(z_{t+1})\|.
\end{align*}
In Lemma \ref{lemma:key},we know that $d(z_{t+1},z_t) \le \eta(1+8L\eta)\|\F(z_t)\| $. Putting them together, we have
\begin{align*}
    \langle  \eta \F(z_{t+1}), (-H_{z_{t+1},p}^{z^*} + \sigma_1)& -\eta \F(z_{t+1}) + \eta \Gamma_{z_{t}}^{z_{t+1}} \F(z_{t})\rangle 
    \\& \le (\zeta_1-\sigma_1)L\eta^3 (1+8L\eta) \|\F(z_t)\|\|\F(z_{t+1})\| \\
    &\le 2(\zeta_1-\sigma_1)L\eta^3 (1+8L\eta) \|\F(z_t)\|^2.
\end{align*}
Moreover, by Lemma \ref{lemma:key}, we have
\begin{align*}
    \langle \Gamma_{\hat z_{t+2}}^{\hat z_{t+1}} G_{t+2} - \eta \F(z_{t+1}), \Gamma_{\hat z_{t+2}}^{\hat z_{t+1}}\expinv{\hat z_{t+2}}{z^*} \rangle &\le d(\hat z_{t+2},z^*) 104 \eta^3 \|\F(z_t)\|^3\\
    &\le (D+4L\eta^2\|\F(z_t)\|)104K_m  \eta^3 \|\F(z_t)\|^3.
\end{align*}
Hence, we have
\begin{align*}
    A &\le  2(D+4L\eta^2\|\F(z_t)\|)104 \eta^3 \|\F(z_t)\|^3+2(\zeta_1-\sigma_1)L\eta^3(1+8L\eta)2 \|\F(z_t)\|^2\\
      & \quad  + 2 \langle  \eta \F(z_{t+1}), \expinv{z_{t+1}}{z^*} \rangle.
\end{align*}
Since $\eta \le \frac{1}{2G}$ and $\eta\le\frac{1}{20L}$, we can find
\begin{align}\label{eq:A}
   A \le 104 K_m (2D+\frac{1}{5}) \eta^3 \|\F(z_t)\|^3+\frac{28}{5}(\zeta_1-\sigma_1)L\eta^3\|\F(z_t)\|^2 + 2 \langle  \eta \F(z_{t+1}), \expinv{z_{t+1}}{z^*} \rangle.
\end{align}

Next we analyze $B$ and $C$ as follows.
\begin{align}
    \nonumber B &=  \eta^2 \sigma_1 (\| F{z_{t+2}} - \Gamma_{t+1}^{t+2} \F(z_{t+1}) \|^2 - \|\F(z_t)\|^2)\\
    \nonumber & \le \eta^2\sigma_1 (L^2d^2(z_{t+2},z_{t+1})- \|\F(z_t)\|^2)\\
    \nonumber & \le \eta^2\sigma_1(4L^2\eta^2\|\F(z_t)\|^2- \|\F(z_t)\|^2)\\
   \label{eq:B}  & \le \eta^2\sigma_1(\frac{1}{5}L\eta\|\F(z_t)\|^2- \|\F(z_t)\|^2), \\
\label{eq:C} C & \le \sigma_1 64 \eta^6 \|\F(z_t)\|^6 K_m \le 8 \sigma_1  K_m \eta^3 \|\F(z_t)\|^3.
\end{align}

Finally, taking \eqref{eq:A}, \eqref{eq:B} and \eqref{eq:C} in \eqref{eq:final1la}, we have
\begin{align}\label{eq:final2la}
    \nonumber &\phi_{t+2} -\phi_{t+1} \le\\
    \nonumber &\quad\eta^2\|\F(z_t)\|^2 \Big( 104 K_m (2D+\frac{1}{5}) \eta \|\F(z_t)\| +  \frac{28}{5}(\zeta_1-\sigma_1)L\eta +\sigma_1  K_m \eta \|\F(z_t)\|+  \frac{\sigma_1L}{5}\eta - \sigma_1 \Big) \\ 
    &\quad + 2 \langle  \eta \F(z_{t+1}), \expinv{z_{t+1}}{z^*} \rangle.
\end{align}
Because $f$ is g-convex-concave, we have
\begin{align*}
     2 \langle  \eta \F(z_{t+1}), \expinv{z_{t+1}}{z^*} \rangle \le 0,
\end{align*}
By g-$G$-Lipschitz, we have
\begin{align*}
    \phi_{t+2} -\phi_{t+1} \le \eta^2\|\F(z_t)\|^2 \Big( (104 K_m (2D+\frac{1}{5}) G +  \frac{28}{5}(\zeta_1-\sigma_1)L +\sigma_1  K_m G+  \frac{\sigma_1L}{5})\eta - \sigma_1 \Big).
\end{align*}
Since 
\begin{align*}
    \eta \le \frac{\sigma_1}{2(104 K_m (2D+\frac{1}{5})  G +  \frac{28}{5}(\zeta_1-\sigma_1)L +\sigma_1  K_m G+  \frac{\sigma_1L}{5})},
\end{align*}
we have
\begin{align}\label{eq:final-c}
    \phi_{t+2} -\phi_{t+1} \le -\frac{\sigma_1}{2} \eta^2 \| \F(z_t) \|^2.
\end{align}
Summing \eqref{eq:final-c} from $t=0,1,2\dots$ yields
\begin{align*}
    \sum_{t=1}^\infty \|\F(z_t)\|^2 \le  \frac{2}{\sigma_1 \eta^2}\phi_1 \le \frac{2 D^2}{\sigma_1 \eta^2},
\end{align*}
which immediately indicates that
\begin{align*}
     \min_{t\le T} \| \nabla \f(z_t)\| \le \frac{2D}{\eta\sqrt{\sigma_1T}},
\end{align*}
and 
\begin{align*}
    \lim_{t\to\infty}  \| \nabla \f(z_t)\| = 0. 
\end{align*} 

In this way, we have completed our proof. $\hfill \blacksquare$

\paragraph{Proof of Theorem \ref{thm: last-c}} 
The g-strongly convexity-strongly concavity implies
\begin{align*}
    2 \langle  \eta \F(z_{t+1}), \expinv{z_{t+1}}{z^*} \rangle \le -\mu \eta d^2(z_{t+1},z^*).
\end{align*}
By Young's inequality, we have
\begin{align}\label{eq:scfinal1}
   \nonumber 2 \langle  \eta \F(z_{t+1}), \expinv{z_{t+1}}{z^*} \rangle &\le -\mu \eta d^2(z_{t+1},z^*)\\
    & \le  \mu \eta d^2(z_{t+1},\hat z_{t+2}) - \frac{\mu \eta}{2} d^2(\hat z_{t+2},z^*).
\end{align}
Putting \eqref{eq:scfinal1} into \eqref{eq:final2la} and $\Upsilon = 104 K_m (2D+\frac{1}{5})  G +  \frac{28}{5}(\zeta_1-\sigma_1)L +\sigma_1  K_m G+  \frac{\sigma_1L}{5}$, we write
\begin{align}\label{eq:sc1}
    \nonumber\phi_{t+2} - \phi_{t+1} \le & \eta^2\|\F(z_t)\|^2 (\Upsilon \eta - \sigma_1) + \mu\eta d^2(z_{t+1},\hat z_{t+2}) - \frac{\mu \eta}{2} d^2(\hat z_{t+2},z^*)\\
    \nonumber& = \eta^2\|\F(z_t)\|^2 (\Upsilon \eta - \sigma_1) + \mu\eta d^2(z_{t+1},\hat z_{t+2}) - \frac{\mu \eta}{2} \phi_{t+2} \\
    &\quad + \frac{\mu \eta}{2} \sigma_1 \eta^2 \| \F(z_{t+2}) - \Gamma_{z_{t+1}}^{z_{t+2}} \F(z_{t+1})\|^2 .
\end{align}
From the proof Lemma \ref{lemma:key}, we have 
\begin{align}\label{eq:sc2}
    \begin{cases}
        d(z_{t+1},\hat z_{t+2}) \le 2 \eta \|\F(z_t)\|\\
        \eta\| \F(z_{t+2}) - \Gamma_{z_{t+1}}^{z_{t+2}} \F(z_{t+1})\| = d(z_{t+2},\hat z_{t+3}) \le 2 \eta \|\F(z_{t+1})\|\le 4 \eta \|\F(z_t)\|.
    \end{cases}
\end{align}
Taking \eqref{eq:sc2} in \eqref{eq:sc1}, we have
\begin{align*}
    (1+\frac{\mu\eta}{2})\phi_{t+2} - \phi_{t+1} \le \eta^2\|\F(z_t)\|^2 (\Upsilon \eta + 4\mu\eta + 8\sigma_1\eta - \sigma_1).
\end{align*}
Since 
\begin{align*}
    \eta \le \frac{\sigma_1}{\Upsilon+4\mu+8\sigma\mu},
\end{align*}
we have
\begin{align}\label{eq:final-sc}
    \phi_{t+2}  \le \frac{1}{1+\mu\eta/2}\phi_{t+1},
\end{align}

which give us
\begin{align*}
    \begin{cases}
        \sigma_1 d^2(z_{t},\hat z_{t+1}) \le \phi_t \le (\frac{1}{1+\mu\eta/2})^t \phi_1\\
        d^2(\hat z_{t+1},z^*) \le \phi_{t+1}  \le (\frac{1}{1+\mu\eta/2})^{t+1} \phi_1.
    \end{cases}
\end{align*}
Finally, we get
\begin{align*}
    d^2(z_t,z^*) &\le 2 (d^2(z_{t},\hat z_{t+1}) + d^2(\hat z_{t+1},z^*))\\
    &\le 2 (1+\frac{1}{\sigma_1}) (\frac{1}{1+\mu\eta/2})^t \phi_1\\
    & =   2 (1+\frac{1}{\sigma_1}) (\frac{1}{1+\mu\eta/2})^t d^2(z_1,z^*),
\end{align*}
which completes our proof.